\documentclass[reqno]{article}
\usepackage{amsmath,amsfonts,amssymb,bbm,amsthm}

\numberwithin{equation}{section}

\newtheorem{theorem}{Theorem}[section]
\newtheorem{lemma}[theorem]{Lemma}
\newtheorem{proposition}[theorem]{Proposition}
\newtheorem{claim}[theorem]{Claim}
\newtheorem{corollary}[theorem]{Corollary}

\theoremstyle{definition}
\newtheorem{nremark}[theorem]{Remark}
\newtheorem*{remark}{Remark}
\newtheorem{assumption}[theorem]{Assumption}

\renewcommand{\le}{\leqslant}
\renewcommand{\ge}{\geqslant}

\newcommand\noproof{\hfill$\Box$}

\newcommand\ZZ{{\mathbb Z}}
\newcommand\RR{{\mathbb R}}
\newcommand\Bi{{\mathrm{Bin}}}
\newcommand\eps{\varepsilon}
\newcommand\la{\lambda}
\newcommand\las{\lambda_*}

\renewcommand\Pr{{\mathbb P}}
\newcommand\Q{{\mathbb Q}}
\newcommand\E{{\mathbb E}}
\newcommand\Var{{\mathrm{Var}}}
\newcommand\Covar{{\mathrm{Cov}}}
\newcommand\dto{\overset{\mathrm{d}}{\to}}
\newcommand\cc{{\mathrm{c}}}
\newcommand\cF{\mathcal{F}}
\newcommand\cH{\mathcal{H}}
\newcommand\cG{\mathcal{G}}
\newcommand\cA{\mathcal{A}}
\newcommand\cB{\mathcal{B}}
\newcommand\cC{\mathcal{C}}
\newcommand\cD{\mathcal{D}}
\newcommand\cL{\mathcal{L}}
\newcommand\cQ{\mathcal{Q}}
\newcommand\cI{\mathcal{I}}
\newcommand\cP{\mathcal{P}}

\newcommand\ubig{\mathcal{U}_\mathrm{big}}
\newcommand\ucx{\mathcal{U}_\mathrm{cx}}

\newcommand\cE{\mathcal{E}}

\newcommand\GI{G_{\mathrm{inc}}}
\newcommand\tL{\widetilde{L}}

\newcommand\trho{\widetilde{\rho}}

\newcommand\bb[1]{\bigl(#1\bigr)}
\newcommand\Bb[1]{\Bigl(#1\Bigr)}
\newcommand\bm[1]{\bigl|#1\bigr|}
\newcommand\ceil[1]{\lceil#1\rceil}
\newcommand\floor[1]{\lfloor#1\rfloor}

\newcommand\core{C}
\newcommand\excore{C^+}
\newcommand\eco{\cC_1^+}
\newcommand\range{R}

\newcommand\dy{\mathrm{d}y}
\newcommand\dz{\mathrm{d}z}
\newcommand\da{\mathrm{d}a}
\newcommand\db{\mathrm{d}b}

\newcommand\Hrnp{H^r_{n,p}}
\newcommand\Hrnm{H^r_{n,m}}

\newcommand\Htnp{H^2_{n,p}}

\newcommand\Hrnptwo{H^r_{n,p_2}}

\newcommand\Hrnpi{H^r_{n,p_i}}
\newcommand\Hrmp{H^r_{m,p}}
\newcommand\Hrnsp{H^r_{n-s,p}}
\newcommand\Hrnsop{H^r_{n-\mu_n,p}}
\newcommand\Nc{N_{\mathrm{con}}}

\newcommand\cN{\mathcal{N}}
\newcommand\cM{\mathcal{M}}

\newcommand\partit[2]{\{#1:#2\}}

\newcommand\ind[1]{\mathbbm{1}_{#1}}

\begin{document}
\title{Counting connected hypergraphs via the probabilistic method}
\author{B\'ela Bollob\'as%
\thanks{Department of Pure Mathematics and Mathematical Statistics,
Wilberforce Road, Cambridge CB3 0WB, UK and
Department of Mathematical Sciences, University of Memphis, Memphis TN 38152, USA.
E-mail: {\tt b.bollobas@dpmms.cam.ac.uk}.}
\thanks{Research supported in part by NSF grant DMS-1301614 and
EU MULTIPLEX grant 317532.}
\and Oliver Riordan%
\thanks{Mathematical Institute, University of Oxford, Radcliffe Observatory Quarter, Woodstock Road, Oxford OX2 6GG, UK.
E-mail: {\tt riordan@maths.ox.ac.uk}.}}
\date{April 23, 2014; revised November 17, 2015}
\maketitle

\begin{abstract}
In 1990 Bender, Canfield and McKay gave an asymptotic formula for the number of
connected graphs on $[n]=\{1,2,\ldots,n\}$ with $m$ edges, whenever $n$ and the nullity
$m-n+1$ tend to infinity.
Let $C_r(n,t)$ be the number of connected $r$-uniform hypergraphs on $[n]$
with nullity $t=(r-1)m-n+1$, where $m$ is the number of edges. For $r\ge 3$,
asymptotic formulae for $C_r(n,t)$ are known only for
partial ranges of the parameters: in 1997 Karo\'nski and \L uczak gave one for
$t=o(\log n/\log\log n)$, and recently Behrisch, Coja-Oghlan and Kang gave
one for $t=\Theta(n)$.
Here we prove such a formula for any fixed $r\ge 3$ and any $t=t(n)$ satisfying
$t=o(n)$ and $t\to\infty$ as $n\to\infty$, complementing the last result.
This leaves open only the case $t/n\to\infty$, which we expect to be much simpler,
and will consider in future work. The proof is based on probabilistic methods, 
and in particular on a bivariate local limit theorem for the number of vertices and edges
in the largest component of a certain random hypergraph. We deduce this from
the corresponding central limit theorem by smoothing techniques.
\end{abstract}


\section{Introduction}

Our aim in this paper is to prove a result about $r$-uniform hypergraphs that can
be viewed in two complementary ways, either as a probabilistic result or as an enumerative one.
In this section we shall state the enumerative form; in the next section we
switch to the probabilistic viewpoint, which we shall adopt for most of the paper,
and in particular in the proofs.

If $H$ is an $r$-uniform hypergraph then
\[
 |H|\le c(H)+(r-1)e(H),
\]
where $|H|$ is the number of vertices of $H$, $e(H)$ is the number of edges,
and $c(H)$ is the number of components, with equality if and only if 
$H$ is a forest, i.e., every component of $H$ is a tree.
Define the \emph{nullity} $n(H)$ of $H$ as
\begin{equation}\label{nulldef}
 n(H) = c(H)+(r-1)e(H)-|H|,
\end{equation}
so $n(H)\ge 0$, and $H$ is a tree iff $c(H)=1$ and $n(H)=0$. Note for later that, if $H$ is connected, then $|H|+n(H)-1$ must be a multiple of $r-1$.
If we replace each
hyperedge of $H$ by a tree on the same set of $r$ vertices,
then $n(H)$ is simply the nullity of the resulting (multi-)graph.
Connected graphs or hypergraphs are naturally parameterised by the number of vertices
and the nullity, although often the \emph{excess} $n(H)-1$ is considered instead.

One of the most basic questions about any class of combinatorial (or other) structures is:
how many such structures are there with given `size' parameters? Or, sometimes
more naturally, how many `irreducible' structures? For (labelled) graphs and hypergraphs,
the first question is trivial, but the second, taking `irreducible' to mean
connected, certainly is not, and it is no surprise that it has been extensively
studied. Given integers $r\ge 2$, $s\ge 1$ and $t\ge 0$, let $C_r(s,t)$
be the number of connected
$r$-uniform hypergraphs on $[s]=\{1,2,\ldots,s\}$ having nullity $t$.
(Thus $C_r(s,t)=0$ if $r-1$ does not divide $s+t-1$.)
Starting with Cayley's formula $C_2(s,0)=s^{s-2}$, the asymptotic evaluation
of $C_2(s,t)$ was studied by Wright~\cite{Wright1,Wright2,Wright3,Wright4}
and others for increasingly broad ranges of $t=t(s)$, culminating in the results of
Bender, Canfield and McKay~\cite{BCMcK} giving an asymptotic formula for $C_2(s,t)$
whenever $s\to\infty$, for any function $t=t(s)$.

For $r\ge 3$, much less is known. Selivanov~\cite{Selivanov} gave an exact
formula for the number $C_r(s,0)$ of trees; 
the remaining results we shall mention are all asymptotic,
with $r$ fixed, $s\to\infty$, and $t$ some function of $s$.
Karo\'nski and \L uczak~\cite{KL_sparse} gave an asymptotic
formula for $C_r(s,t)$ when
$t=o(\log s/\log\log s)$, so the hypergraphs counted are quite close to trees.
In an extended abstract from 2006, Andriamampianina and Ravelomanana~\cite{AR} outlined
an extension of this to the case $t=o(s^{1/3})$. Recently, Behrisch, Coja-Oghlan and Kang~\cite{BC-OK2b}
gave an asymptotic formula for $C_r(s,t)$ when $t=\Theta(s)$; their proof is based on probabilistic
methods, which seem to work best when $t$ is relatively large, rather than the enumerative
methods most successful for small $t$.
Independently and essentially simultaneously with the present work,
Sato and Wormald~\cite{SatoWormald} (see also Sato~\cite{Sato}) have given
an asymptotic formula for $C_r(s,t)$ when $r=3$, $t=o(s)$ and $t/(s^{1/3}\log^2 s)\to\infty$.

Our main result complements those in~\cite{BC-OK2b}, and greatly extends those in~\cite{KL_sparse,SatoWormald},
covering
the entire range $t\to\infty$, $t=o(s)$. The formula we obtain is rather complicated; to state it we
need some definitions.

Given an integer $r\ge 2$ and a real number $0<\rho<1$, define
\begin{equation}\label{Prdef}
 \Psi_r(\rho) = -\frac{r-1}{r}\frac{\log(1-\rho)}{\rho}\frac{1-(1-\rho)^r}{1-(1-\rho)^{r-1}} -1.
\end{equation}
For any $r\ge 2$ it is easy to see that $\Psi_r(\rho)$ is strictly increasing on $(0,1)$, 
since each of the factors $-\log(1-\rho)/\rho$ and $(1-(1-\rho)^r)/(1-(1-\rho)^{r-1})$ is.
Since $\Psi_r$ is continuous, considering the limits at $0$ and $1$ we see
that $\Psi_r$ gives a bijection from $(0,1)$ to $(0,\infty)$. 

\begin{theorem}\label{thenum}
Let $r\ge 2$ be fixed, and let $t=t(s)$ satisfy $t\to\infty$ and $t=o(s)$ as $s\to\infty$.
Then when $s+t-1$ is divisible by $r-1$ the number $C_r(s,t)$ of connected
$r$-uniform hypergraphs on $[s]$ with nullity $t$
satisfies
\begin{equation}\label{enumform}
 C_r(s,t) \sim
  \frac{\sqrt{3}}{2\sqrt{\pi}} \frac{r-1}{\sqrt{s}}
 \left( \frac{ e \big(1-(1-\rho)^r\big) s^r }{m\ r!\ \rho^r} \right)^m
 \bigl( \rho(1-\rho)^{(1-\rho)/\rho}\bigr)^s
\end{equation}
as $s\to\infty$, where $\rho>0$ is the unique positive solution to
\begin{equation}\label{rdefc}
 \Psi_r(\rho) = \frac{t-1}{s},
\end{equation}
and $m=(s+t-1)/(r-1)$ is the number of edges of any such hypergraph.
Moreover, the probability $P_r(s,t)$ that a random $m$-edge $r$-uniform hypergraph on $[s]$
is connected satisfies
\begin{equation}\label{Pform}
 P_r(s,t)  \sim e^{r/2+\ind{r=2}} \sqrt{\frac{3(r-1)}{2}}  \left(\frac{1-(1-\rho)^r}{\rho^r}\right)^m
 \bigl( \rho(1-\rho)^{(1-\rho)/\rho}\bigr)^s,
\end{equation}
where $\ind{A}$ denotes the indicator function of $A$.
\end{theorem}

To understand this result it may help to note that $\Psi_r(x)=(r-1)x^2/12+O(x^3)$ as $x\to 0$, so
\[
 \rho \sim 2\sqrt{\frac{3}{r-1}\frac{t}{s}}
\]
when $t/s\to 0$. Also, it may be useful to note that rearranging \eqref{rdefc} gives $m/s=(\Psi_r(\rho)+1)/(r-1)$,
so we can rewrite the formulae \eqref{enumform} and \eqref{Pform} as functions of $s$ and $\rho$ only (or $m$ and $\rho$ only)
if we wish.

There are many ways to write a formula such as \eqref{enumform}, and checking whether two
such formulae agree may require some calculation. In the Appendix we present such calculations
showing that Theorem~\ref{thenum} matches the results of~\cite{BC-OK2pre,BC-OK2abs,BCMcK,KL_sparse,SatoWormald}
where the ranges of applicability overlap, as well as the corrected version of~\cite{BC-OK2b}.
In particular, for the graph case (which of course is not our main focus),
\eqref{enumform} is consistent with (indeed, implied by) the Bender--Canfield--McKay formula~\cite{BCMcK}.
For hypergraphs, Theorem~\ref{thenum} shows that the asymptotic formula of Karo\'nski and \L uczak~\cite{KL_sparse}
extends not only to $t=o(s^{1/3})$, as they suspected, but to any $t=o(s^{1/2})$ (and no further).

We shall return to the topic of estimating $C_r(s,t)$ when $t/s\to\infty$ in a future paper~\cite{dense}.
Although we have not yet checked all the details, this regime seems to be much easier
to analyze than that considered here or by Behrisch, Coja-Oghlan and Kang.
The key point is that, following the approach taken in the next section, the random hypergraph
that one needs to analyze has average degree tending to infinity, which means that its behaviour
is relatively simple. In particular, with high probability all small components are trees.

\section{Probabilistic reformulation}

In this section we shall state a probabilistic result that turns out to be equivalent to
Theorem~\ref{thenum}; as we shall see, the formulae in this setting are significantly
simpler. In the rest of the paper we shall use probabilistic methods
to prove this reformulation, deducing Theorem~\ref{thenum} in Section~\ref{sec_deduce}.

For $2\le r\le n$ and $0<p<1$,
let $\Hrnp$ be the random $r$-uniform hypergraph with vertex set $[n]=\{1,2,\ldots,n\}$ in which
each of the $\binom{n}{r}$ possible hyperedges is present independently
with probability $p$.
Throughout we consider $r\ge 2$ fixed, $n\to\infty$, and 
\[
 p=p(n)=\la (r-2)!n^{-r+1},
\]
where $\la=\la(n)=\Theta(1)$; often, we write $\la$ as $\la(n)=1+\eps(n)$.
It is well known (see Section~\ref{sec_past})
that the model $\Hrnp$ undergoes a phase transition at $\la=1$ analogous to that
established by Erd\H os and R\'enyi~\cite{ERgiant} in the
graph case, and indeed that the `window' of this phase transition is given by $\la=1+\eps$
with $\eps^3n=O(1)$; see~\cite{BR_hyp}.
For this reason, we call the model $\Hrnp$ \emph{subcritical} if $\la=1-\eps$
with $\eps=\eps(n)$ satisfying $\eps^3n\to\infty$, and \emph{supercritical} if $\la=1+\eps$
with $\eps^3n\to\infty$. Here we study the supercritical phase, so
throughout this paper we make the following assumption unless specified otherwise.

\begin{assumption}\label{A1}\rm (Weak Assumption.)
The quantities $p(n)$, $\la(n)$ and $\eps(n)>0$ are related by
$\la=1+\eps$ and $p=\la (r-2)!n^{-r+1}$. Moreover, $r\ge 2$ is fixed
and, as $n\to\infty$, we have $\eps^3n\to\infty$ and $\eps=O(1)$.
\end{assumption}

Much of the time we additionally suppose that $\eps\to 0$, i.e., assume the following.

\begin{assumption}\label{A0}\rm (Standard Assumption.)
The conditions of Assumption~\ref{A1} hold, and in addition $\eps\to 0$ as $n\to\infty$.
\end{assumption}

Given a hypergraph $H$, let $\cL_1(H)$ denote the component with the most vertices,
chosen according to any fixed rule if there is a tie. Let $L_1(H)=|\cL_1(H)|$,
$M_1(H)=e(\cL_1(H))$ and $N_1(H)=n(\cL_1(H))$ be the order, size and
nullity of this component. Our next result gives an asymptotic formula for the probability
that the triple $(L_1(\Hrnp),M_1(\Hrnp),N_1(\Hrnp))$ takes any specific
value within the `typical' range, throughout the supercritical regime.
Of course, since these three parameters are dependent,
the result can be stated in terms of any two of them; here we consider $L_1$ and $N_1$.
To state the result we need a few definitions.

For $\la>1$ let $\rho_{\la}$ be the unique positive solution to
\begin{equation}\label{rldef}
 1-\rho_{\la} = e^{-\la \rho_\la},
\end{equation}
so $\rho_\la$ is the survival probability of a Galton--Watson branching process
whose offspring distribution is Poisson with mean $\la$,
and define $\rho_{r,\la}$ by 
\begin{equation}\label{rkldef}
 1-\rho_{r,\la} = (1-\rho_{\la})^{1/(r-1)}.
\end{equation}
It is easy to see that $\rho_{r,\la}$ is the survival probability of a certain branching
process naturally associated to the neighbourhood exploration process in $\Hrnp$,
$p=\la (r-2)!n^{-r+1}$,
where each particle has a Poisson $\mathrm{Po}(\la/(r-1))$ number of groups of $r-1$ children.
From \eqref{rldef} and \eqref{rkldef} it is easy to check that
\begin{equation}\label{rlprop}
 \la\mapsto \rho_{r,\la}\hbox{ is a continuous function }(1,\infty)\to (0,1).
\end{equation}
Turning to the analogous parameter relevant to $N_1(\Hrnp)$, set
\begin{equation}\label{rhosdef}
 \rho_{r,\la}^* = \frac{\la}{r}\bb{1-(1-\rho_{r,\la})^r} - \rho_{r,\la}.
\end{equation}
As noted in~\cite{BR_hyp2}, if $\la=1+\eps$ then, as $\eps\to 0$ from above, we have
\begin{equation}\label{newrasympt}
 \rho_{r,\la} \sim \frac{2\eps}{r-1} \qquad\hbox{and}\qquad
 \rho_{r,\la}^* \sim \frac{2}{3(r-1)^2}\eps^3.
\end{equation}

\begin{theorem}\label{thprob}
Let $r\ge 2$ be fixed, let $p=p(n)=(1+\eps)(r-2)!n^{-r+1}$
where $\eps=\eps(n)$ satisfies $\eps\to 0$ and $\eps^3 n\to\infty$,
set $\la=\la(n)=1+\eps$ and define
$\rho_{r,\la}$ and $\rho^*_{r,\la}$ as above.
Then, whenever $x_n=\rho_{r,\la}n+O(\sqrt{n/\eps})$
and $y_n=\rho^*_{r,\la}n+O(\sqrt{\eps^3n})$ with $x_n+y_n-1$ divisible by $r-1$,
we have
\begin{equation}\label{pointprob}
 \Pr\bb{ L_1(\Hrnp) = x_n,\ N_1(\Hrnp)=y_n }
 \sim \frac{r-1}{\sigma_n\sigma^*_n} 
    f\left(\frac{x_n-\rho_{r,\la}n}{\sigma_n},\frac{y_n-\rho^*_{r,\la}n}{\sigma^*_n}\right)
\end{equation}
as $n\to\infty$, where $\sigma_n=\sqrt{2n/\eps}$, $\sigma_n^*=\sqrt{10/3}(r-1)^{-1}\sqrt{\eps^3n}$,
and
\begin{equation}\label{pdf}
 f(a,b) = \frac{1}{2\pi\sqrt{2/5}}\exp\left(-\frac{5}{4}(a^2-2\sqrt{3/5}\,ab+b^2)\right)
\end{equation}
is the probability density function of a bivariate Gaussian distribution with
mean~$0$, unit variances, and covariance $\sqrt{3/5}$.
\end{theorem}
We shall comment briefly on the uniformity
of the asymptotics in~\eqref{pointprob} above in Remark~\ref{rem:asy1} below.
For ease of comparison with other results, note that combining \eqref{pointprob} and \eqref{pdf} results
in the expression
\begin{equation}\label{pp2}
 \frac{\sqrt{6}}{8\pi}\frac{(r-1)^2}{\eps n} \exp\left(-\frac{5}{4}(a^2-2\sqrt{3/5}\,ab+b^2)\right),
\end{equation}
with $a$ and $b$ the arguments of $f$ in \eqref{pointprob}.

The probability that the largest component of $\Hrnp$ has $\ell$  vertices and $m$ edges
is very closely related to the number
of connected hypergraphs with $\ell$ vertices and $m$ edges. This relationship
was used by Karo\'nski and {\L}uczak~\cite{KL_giant} to prove the special case of Theorem~\ref{thprob}
when $\eps^3n\to\infty$ but $\eps^3n=o(\log n/\log\log n)$. 
Behrisch, Coja-Oghlan and Kang~\cite{BC-OK1,BC-OK2a} used probabilistic methods to prove a
result corresponding to Theorem~\ref{thprob} but with $\eps=\Theta(1)$ (i.e.,
roughly speaking, the case $\la>1$ constant), and then, in~\cite{BC-OK2b}, used this to deduce their enumerative
result mentioned in the previous section. We shall deduce Theorem~\ref{thenum}
from Theorem~\ref{thprob} in Section~\ref{sec_deduce}.

At a very high level, the strategy of the proof of Theorem~\ref{thprob} is similar
to that followed by Behrisch, Coja-Oghlan and Kang~\cite{BC-OK2a} for the
case $\eps=\Theta(1)$: we start from
the corresponding central limit theorem (proved very recently in~\cite{BR_hyp2}),
and apply `smoothing' arguments to deduce the local limit theorem. However,
the details are very different: Behrisch, Coja-Oghlan and Kang apply this technique to a univariate 
result for $L_1$ only, and then use a different argument going via the hypergraph
model analogous to $G(n,m)$ to deduce a bivariate result. This method
does not appear to work when $\eps\to 0$. Instead, we apply two smoothing arguments; one
to handle the nullity (or excess), and then one for the number of vertices.

Bivariate local limit results do not necessarily imply the corresponding
univariate local limit results, due to the possibility of a `bad' event $\cB$
on which one of the two parameters takes a `typical' value and the other does not,
with $\Pr(\cB)=o(1)$ but $\Pr(\cB)$ large compared to the relevant point probabilities.
However, the method used to prove Theorem~\ref{thprob} gives the following local limit
results for $L_1(\Hrnp)$ and $N_1(\Hrnp)$ separately.

\begin{theorem}\label{thL1only}
Let $r\ge 2$ be fixed, and let $p=p(n)=(1+\eps)(r-2)!n^{-r+1}$
where $\eps=\eps(n)$ satisfies $\eps\to 0$ and $\eps^3 n\to\infty$.
Set $\la=\la(n)=1+\eps$ and define $\rho_{r,\la}$ as in \eqref{rkldef}.
Then whenever $x_n=\rho_{r,\la}n+O(\sqrt{n/\eps})$ we have
\[
 \Pr\bb{ L_1(\Hrnp) = x_n} \sim \frac{1}{2\sqrt{\pi n/\eps}} 
 \exp\left(-\frac{(x_n-\rho_{r,\la} n)^2}{4n/\eps}\right)
\]
as $n\to\infty$.
\end{theorem}

\begin{theorem}\label{thN1}
Let $r\ge 2$ be fixed, let $p=p(n)=(1+\eps)(r-2)!n^{-r+1}$ where $\eps=\eps(n)\to 0$
and $\eps^3n\to\infty$, and set $\la=\la(n)=1+\eps$. For any $t_n\ge 0$ we have
\[
 \Pr\bb{ N_1(\Hrnp) = t_n } = \frac{1}{\sigma^*_n\sqrt{2\pi}} \exp\left(-\frac{(t_n-\rho^*_{r,\la}n)^2}{2(\sigma^*_n)^2}\right) + o(1/\sigma^*_n),
\]
where $\rho^*_{r,\la}$ is defined in \eqref{rhosdef} and $\sigma^*_n=\sqrt{10/3}(r-1)^{-1}\sqrt{\eps^3n}$.
\end{theorem}

Our main results assume our Standard Assumption~\ref{A0};
however, all our arguments can be extended, with
varying amounts of additional work (and more complicated statements), to require
only our Weak Assumption~\ref{A1}. Since the results of Behrisch, Coja-Oghlan and Kang~\cite{BC-OK1,BC-OK2a}
cover the case $\eps=\Theta(1)$, we assume $\eps\to 0$ much of the time for simplicity.

In the probabilistic setting, a local (central) limit theorem is not the last possible word.
One could ask for moderate and/or large deviation results (indeed, Eyal Lubetzky has asked
us this question). We have not pursued these questions, but for a wide range of the parameters
Lemma~\ref{lbigep} shows that the probability that the largest component of $\Hrnp$
has $s$ vertices and nullity $t$ is asymptotic to the expected number of components
of $\Hrnp$ with these parameters. This expectation can of course be calculated using
Theorem~\ref{thenum}. This method should give tight results for all moderate deviations
and some (but not all) large deviations.

\begin{nremark}\rm
Instead of the model $\Hrnp$ one could consider the analogue $\Hrnm$ of the original Erd\H os--R\'enyi
size model, where we select an $m$-edge $r$-uniform hypergraph on $[n]$ uniformly at random. Relating
$m$ and $p$ by $p=m/\binom{n}{r}$, Theorem~\ref{thprob} implies an analogous result
for this model. (This is not completely obvious, but can be shown using Theorem~\ref{thenum} as an intermediate
step; alternatively, one can use Lemma~\ref{lbigep} and its analogue for $\Hrnm$, and directly
relate the expected number of $s$-vertex $k$-edge components in the models $\Hrnp$ and $\Hrnm$.)
Behrisch, Coja-Oghlan and Kang~\cite{BC-OK2a} prove such a result in the denser setting,
i.e., when $\la>1$ is constant.
Here, unlike in~\cite{BC-OK2a}, the parameters of the local limit theorem in $\Hrnm$ are exactly the same
as those in $\Hrnp$. Very informally this should be no surprise, since the conversion between models
corresponds to changing the number of edges by a random number of order $O(\sqrt{n})$. Such a change changes the typical
size of the giant component by $O(\sqrt{n})$ vertices, which (in our range) is small compared to the
standard deviation $\sqrt{n/\eps}$. Similarly, the change in the nullity from switching from one model
to the other is $O(\eps^2\sqrt{n}) =o(\sqrt{\eps^3n})$.
\end{nremark}

\subsection{Related work}\label{sec_past}

We have already mentioned a number of previous enumerative results related to
Theorem~\ref{thenum}. In this subsection we shall outline a number of previous probabilistic results
related to Theorem~\ref{thprob}, but first we introduce some general terminology.

Let $(A_n)$ be a sequence of integer-valued random variables. We say that $(A_n)$ satisfies
a \emph{global limit theorem} with parameters $\mu_n$ and $\sigma_n$
if $(A_n-\mu_n)/\sigma_n$ converges in distribution to some
distribution $Z$ on the reals whose density function $\phi(x)$
is continuous and strictly positive.
We say that $(A_n)$ satisfies the corresponding \emph{local limit theorem}
if, for any sequence $(x_n)$ of integers with $x_n=\mu_n+O(\sigma_n)$, we have
\begin{equation}\label{llt}
 \Pr( A_n=x_n) \sim \frac{\phi\bb{ (x_n-\mu_n)/\sigma_n }}{\sigma_n}
\end{equation}
as $n\to\infty$. In the examples considered
here, $Z$ will always be the standard normal distribution $N(0,1)$,
but this is not necessary for the general arguments.  These definitions
extend in a natural way to \emph{bivariate} global and local limit
theorems for sequences $(A_n,B_n)$. In these terms, Theorem~\ref{thprob}
is a bivariate local limit theorem for the pair $(L_1(\Hrnp),N_1(\Hrnp))$.

\begin{nremark}\label{rem:asy1}\rm
Let us comment in some detail on the issue of uniformity in asymptotics such as
\eqref{llt} above, since this may perhaps cause some confusion. In general, we 
adopt the approach of quantifying over sequences, since this seems intuitive
and avoids lengthy sequences of quantifiers. For example, writing $\eta(n,x_n)$ for the ratio
of the two sides of \eqref{llt} above,  the precise interpretation of \eqref{llt}
is the following: for any sequence $(x_n)$ with the property that
$\sup_n |x_n-\mu_n|/\sigma_n<\infty$, we have $\eta(n,x_n)\to 1$ as $n\to\infty$.
Thus the rate at which $\eta(n,x_n)$ tends to 1 is allowed to depend on the choice
of the sequence $(x_n)$. 

Of course, such a statement automatically gives a certain kind of uniformity: given a constant $C$,
for each $n$ let $x_n^{\pm}$ denote the choices of $x_n$ with $|x_n-\mu_n|\le C\sigma_n$
that maximize/minimize the ratio $\eta(n,x_n)$. Applying \eqref{llt} to the sequences
$(x_n^+)$ and $(x_n^-)$ gives
$\eta(n,x_n^{\pm})\to 1$, so we have the uniform statement
\[
 \max_{x\::\:|x-\mu_n|\le C\sigma_n} \eta(n,x) \to 1
\]
as $n\to\infty$, and the same for $\min$.

In most of our results, we quantify over $r\ge 2$, the choice of a sequence $(p(n))$ satisfying
certain assumptions, and then perhaps additional sequences such as
the sequences $(x_n)$ and $(y_n)$ appearing in
Theorem~\ref{thprob}. The results then state that with all these choices fixed, a certain
sequence indexed by $n$ is $O(1)$ or $o(1)$.
As above, although the bounds are not claimed to be uniform, bounds that are uniform
over suitable sets of choices follow immediately.
\end{nremark}

As usual we say that an event $E=E_n$ (formally a sequence $(E_n)$ of events) holds
\emph{with high probability}, or \emph{whp}, if $\Pr(E_n)\to 1$ as $n\to\infty$.
Analogous to the classical 1960 result of Erd\H{o}s and R\'enyi~\cite{ERgiant} for the case of graphs,
in 1985 Schmidt-Pruzan and Shamir~\cite{S-PS} showed that if $r\ge 2$ is constant (which we
assume throughout) and $p=p(n)=\la (r-2)!n^{-r+1}$, then the random hypergraph
$\Hrnp$ undergoes a phase
transition at $\la=1$: for $\la<1$ constant, whp $L_1(\Hrnp)$ is at most a constant
times $\log n$, if $\la=1$ then $L_1(\Hrnp)$ is of order $n^{2/3}$, and if $\la>1$
is constant then whp $L_1(\Hrnp)\ge c_{r,\la} n$ for some constant $c_{r,\la}>0$.
The model studied in~\cite{S-PS} is in fact more general, allowing edges of different
sizes up to $O(\log n)$.

The case where the `branching factor' $\la$ is bounded and bounded away from $1$ is
essentially equivalent to that where $\la>1$ is constant; we shall not distinguish them 
in this discussion. Still considering this case, in 2007 Coja-Oghlan, Moore and Sanwalani~\cite{C-OMS}
refined the results of Schmidt-Pruzan and Shamir, finding in particular the asymptotic value $\rho_{r,\la}n$ 
of $L_1(\Hrnp)$ in the supercritical case, and giving an asymptotic formula
for its variance.
In 2010 Behrisch, Coja-Oghlan and Kang~\cite{BC-OK1} went further when they established
the limiting distribution of $L_1(\Hrnp)$ in the regime $\la>1$ constant:
they used random walk
and martingale methods to establish a central limit theorem, and then
a smoothing technique, combined with multi-round exposure 
(ideas that appear in a slightly different form in~\cite{C-OMS}), to deduce
the corresponding local limit theorem. In~\cite{BC-OK2a} they deduced from
this a \emph{bivariate} local limit theorem for $L_1(\Hrnp)$ and $M_1(\Hrnp)$ (equivalent
to one for $L_1(\Hrnp)$ and $N_1(\Hrnp)$)
under the same assumption $\la>1$ constant. This result is directly analogous
to Theorem~\ref{thprob} except that $\eps=\Theta(1)$ rather than $\eps\to 0$,
and, as shown in~\cite{BC-OK2b}, leads to an enumerative result
analogous to Theorem~\ref{thenum}, but for hypergraphs with nullity $\Theta(s)$,
where $s$ is the number of vertices.

Turning to the case where $\la=\la(n)\to 1$, let us write $\la$ as $1+\eps$ with
$\eps=\eps(n)$. Building on enumerative results of theirs~\cite{KL_sparse} from 1997,
in 2002 Karo\'nski and \L uczak~\cite{KL_giant} proved a bivariate local limit theorem
for $L_1(\Hrnp)$ and $N_1(\Hrnp)$ just above the `critical window'
$\eps=O(n^{-1/3})$ of the phase transition, in the range where $\eps^3n\to\infty$ 
but $\eps^3n=o(\log n/\log\log n)$.
In an extended abstract from 2006, Andriamampianina and Ravelomanana~\cite{AR} outlined
an extension of the enumerative results of Karo\'nski and \L uczak~\cite{KL_sparse} to treat
hypergraphs with much larger excess (or nullity); this implies an extension of the local
limit theorem of~\cite{KL_giant} to the range where $\eps^3n\to\infty$ but $\eps^4n\to 0$.
These results illustrate a general phenomenon in this field: it seems that the barely
supercritical case is more accessible to enumerative methods, and the strongly
supercritical case ($\la>1$ constant) to probabilistic methods.

In the special case of graphs, even more detailed results have been proved.
Following many earlier results (see,
for example, the references in~\cite{PWio}), in 2006 Luczak and \L uczak proved a local
limit theorem for $L_1(\Htnp)$ throughout the entire supercritical regime, i.e., when $\la=1+\eps$
with $\eps^3n\to\infty$ and $\eps=O(1)$, as part of a more general result about
the random cluster model. Slightly earlier, Pittel and Wormald~\cite{PWio} had come
very close to proving a \emph{trivariate} local limit theorem for $L_1(\Htnp)$, $N_1(\Htnp)$
and a third parameter, the number of vertices in the `core'.
More precisely, they proved a trivariate local limit theorem for the conditional
distribution of these parameters where
the conditioning is on the event that there
is a unique giant component of approximately the right size, an event that holds
with probability $1-o(1)$. With 
hindsight it is easy to remove the conditioning using, for example, Lemma~\ref{lbigep}.

\smallskip
Returning to hypergraphs, if we ask for results covering the entire (weakly) supercritical
regime $\eps^3n\to\infty$, $\eps\to 0$, it is only recently that anything non-trivial has
been proved about the giant component.
Indeed, as far as we are aware, the first result of this type
is the central limit theorem for $L_1(\Hrnp)$ proved in~\cite{BR_hyp}, using
random walk and martingale arguments. A bivariate central
limit theorem for $L_1(\Hrnp)$ and $N_1(\Hrnp)$ was proved very
recently in~\cite{BR_hyp2}, using similar methods.
Here we shall use smoothing ideas as in~\cite{C-OMS,BC-OK2a}, but applied in a 
very different way, to deduce the corresponding bivariate local limit theorem,
Theorem~\ref{thprob}; Theorem~\ref{thenum} will then follow easily.

The methods of Sato and Wormald~\cite{SatoWormald} are extensions
of those used by Pittel and Wormald~\cite{PWio} and so, in the range in which
they apply (i.e., $r=3$, and $p=(1+\eps)(r-2)!n^{-r+1}$ where $\eps=\eps(n)\to 0$
but $\eps^4n/\log^{3/2}n\to\infty$), may potentially lead to a trivariate
local limit result for $L_1$, $N_1$ and the number of vertices in the core. 
As far as we are aware,
whether such a result can be proved throughout the range $\eps\to 0$ but $\eps^3n\to\infty$,
or for $r>3$, is currently open.

\medskip
In the next section we illustrate the basic strategy of our proof of Theorem~\ref{thprob}
by showing how the same idea can be applied in a much simpler setting. Then, in Subsection~\ref{ss_complications},
we describe some of the complications that will arise when we implement this idea to prove Theorem~\ref{thprob}.
Only then, in Subsection~\ref{ss_org}, do we describe the organization of the rest of the paper.
The reason for this is that almost all of the paper is devoted to the proof of Theorem~\ref{thprob},
and our description of the key steps in and structure of this proof will only make
sense after the discussion earlier in Section~\ref{sec_example}.
Formally, next to nothing in Section~\ref{sec_example} is required in the later sections;
the exception is that we use Proposition~\ref{gtol} in the proof of Theorem~\ref{thN1}.

\section{Smoothing: a simple example}\label{sec_example}

The following trivial, standard observation captures the intuition
that `local smoothness' is what is needed to pass from a global limit theorem
to the corresponding local one.

\begin{proposition}\label{gtol}
Suppose that a sequence $(A_n)$ of random variables
satisfies a global limit theorem with parameters $\mu_n$ and $\sigma_n$,
and that $\Pr(A_n=x_n)-\Pr(A_n=x_n')=o(1/\sigma_n)$ as $n\to\infty$
whenever $x_n=\mu_n+O(\sigma_n)$ and $x_n-x_n'=o(\sigma_n)$.
Then $(A_n)$ satisfies the corresponding local limit theorem.
\end{proposition}
Once again, we quantify over sequences: the precise assumption is that for every pair of sequences $(x_n)$ and $(x_n')$
such that $(x_n-x_n')/\sigma_n\to 0$ and $\sup_n |x_n-\mu_n|/\sigma_n <\infty$,
we have $\sigma_n(\Pr(A_n=x_n)-\Pr(A_n=x_n'))\to 0$.
\begin{proof}
Let $\phi(x)$ be the density function associated to the global limit theorem,
and $\Phi(x)=\int_{y<x} \phi(y) \dy$ the corresponding distribution function.
Fix a sequence $(x_n)$ with $x_n=\mu_n+O(\sigma_n)$; by our definition of a local limit theorem
it suffices to show that $\Pr(A_n=x_n)\sim \phi((x_n-\mu_n)/\sigma_n)/\sigma_n$.
Let $C=2\sup_n |x_n-\mu_n|/\sigma_n$, which is finite by assumption.

The global limit theorem implies that for any fixed $x\in [-C,C]$ we have
\[
 \Pr(A_n\le \mu_n+x\sigma_n) = \Phi(x) +o(1)
\]
as $n\to\infty$; since $\Phi(x)$ is continuous the same estimate holds uniformly
in $x\in [-C,C]$.
It follows that if $\delta_n\to 0$ slowly enough, then
\begin{align*}
 \Pr\bb{ x_n - \delta_n\sigma_n < A_n \le x_n + \delta_n\sigma_n }
 &\sim \Phi\left(\frac{x_n-\mu_n}{\sigma_n} + \delta_n\right) - \Phi\left(\frac{x_n-\mu_n}{\sigma_n} - \delta_n\right)  \\
 &\sim 2\delta_n \phi\left(\frac{x_n-\mu_n}{\sigma_n}\right).
\end{align*}
Let $I_n$ be the set of integers $x$ with $x_n-\delta_n\sigma_n < x \le x_n+\delta_n\sigma_n$,
and let $x_n^\pm\in I_n$ be chosen to maximize and minimize $\Pr(A_n=x)$.
Since $x_n^+=\mu_n+O(\sigma_n)$ and $x_n^+-x_n^- = o(\sigma_n)$, by assumption
$\Pr(A_n=x_n^+)$ and $\Pr(A_n=x_n^-)$ differ by $o(1/\sigma_n)$.
It follows that all $2\delta_n\sigma_n+O(1)$ values of $\Pr(A_n=x)$ for $x\in I$
are within $o(1/\sigma_n)$ of each other and hence of their average, which
is $(1+o(1))\phi((x_n-\mu_n)/\sigma_n)/\sigma_n$.
\end{proof}

A standard technique for establishing the smoothness required
by Proposition~\ref{gtol} is to find a `smooth part'
within the distribution of $A_n$.
Given a sequence $(\sigma_n)$ of positive real numbers,
we call a sequence $(\cD_n)$ of sets of probability distributions
on the integers \emph{$\sigma_n$\nobreakdash-smooth}
if the following conditions hold whenever $(Y_n)$ is a sequence
of random variables such that the distribution of $Y_n$ is in $\cD_n$:
\begin{equation}\label{smcond}
\hbox{if $y_n-y_n'=o(\sigma_n)$ then $|\Pr(Y_n=y_n)-\Pr(Y_n=y_n')|=o(1/\sigma_n)$.}
\end{equation}

To give a simple example of a smooth sequence, suppose that
$\sigma_n\to\infty$, fix a constant $c>0$,
and let $\cD_n$ be the family of all binomial distributions with variance
at least $c\sigma_n^2$. It is easy to check that $(\cD_n)$ is $\sigma_n$-smooth, for example
directly from the formula for the binomial distribution.
Note that the number of trials in the binomial distributions need not be $n$,
or even $\Theta(n)$.

The following trivial observation describes at a high level the general
strategy that we shall use to prove Theorem~\ref{thprob};
of course there will be many complications to overcome.

\begin{lemma}\label{ss}
Let $(\sigma_n)$ be a sequence of positive reals,
and let $(\cD_n)$ be $\sigma_n$-smooth.
Let $(\cF_n)$ be a sequence of $\sigma$-algebras,
and suppose that we can write $A_n$ as $X_n+Y_n$, where $X_n$ and $Y_n$ are integer-valued,
$X_n$ is $\cF_n$-measurable, and the conditional distribution of $Y_n$ given $\cF_n$ is always
in $\cD_n$.
If $(A_n)$ satisfies a global limit theorem with parameters $\mu_n$
and $\sigma_n$, then $(A_n)$ satisfies the corresponding local limit theorem.
\end{lemma}
\begin{proof}
Let $(x_n)$ and $(x_n')$ be sequences of integers with $x_n-x_n'=o(\sigma_n)$.
(We may also assume $x_n=\mu_n+O(\sigma_n)$, but do not need this assumption.)
Writing $\Omega_n$ for the probability space on which $A_n$ is defined,
by \eqref{smcond} we have
\begin{multline*}
 \sup_{\Omega_n} \bm{\Pr(A_n=x_n\mid \cF_n) - \Pr(A_n=x_n'\mid \cF_n)} \\
 \le  \sup_{\Omega_n} \sup_{a\in\ZZ} \bm{\Pr(Y_n=a\mid \cF_n) - \Pr(Y_n=a+x_n'-x_n\mid \cF_n)} = o(1/\sigma_n).
\end{multline*}
(As usual, to obtain this uniform bound we consider $a_n\in \ZZ$ and $\omega_n\in \Omega_n$
(almost) achieving
the supremum
over $a$ and $\Omega_n$ above; then we apply \eqref{smcond} with $y_n=a_n$ and $y_n'=a_n+x_n'-x_n$,
to the conditional distribution of $Y_n$ given $\cF_n$ evaluated at $\omega_n$.)
It follows that $|\Pr(A_n=x_n)-\Pr(A_n=x_n')|=o(1/\sigma_n)$,
so we may apply Proposition~\ref{gtol}.
\end{proof}

This `smooth part' technique is easiest to apply in the case of sums
of independent variables; in this setting McDonald~\cite{McDonald}, for example,
used it with each $\cD_n$ consisting of a single
binomial distribution with appropriate parameters. Similar ideas in a combinatorial
setting were used by Scott and Tateno~\cite{ScottTateno}. Behrisch, Coja-Oghlan and Kang~\cite{BC-OK1}
used it to prove the special case of Theorem~\ref{thL1only} where $\eps=\Theta(1)$, with the
$\sigma$-algebra $\cF_n$ corresponding to the first part of a multi-round exposure
of the edges of $\Hrnp$. Their particular decomposition cannot be used to prove
Theorem~\ref{thL1only}, since the variance of the relevant variable $Y_n$ is
too small when $\eps\to 0$; we return to this later.

\begin{remark}
A variant of the method above is to replace the condition~\eqref{smcond} by the stronger condition
$\Pr(Y_n=y_n+1)=\Pr(Y_n=y_n)+O(1/\sigma_n^2)$, as in Davis and McDonald~\cite{DMcD}, for example.
In situations where $Y_n$ has a simple
distribution, this condition may be just as easy to verify as \eqref{smcond}; applying
it leads to a slightly simpler argument overall. In more complicated situations,
including those where the decomposition $X_n+Y_n$ in Lemma~\ref{ss} holds only most of the time,
rather than always, it is likely to be better to consider probabilities
of values $o(\sigma_n)$ apart, as above. Then the error bounds needed in the estimates
of the point probabilities are looser; this is vital in our argument in Section~\ref{sec_main}, for example.
\end{remark}

As a simple warm-up for our main result, let us outline how Lemma~\ref{ss} may be
applied to the variable $A_n=L_1(G_n)$, where $G_n=\Htnp=G(n,p)$ is the standard
Erd\H{o}s--R\'enyi (binomial) random graph with $p=p(n)=\la/n$
with $\la>1$ constant. Since the result here is not new, and our aim
is to illustrate in a simple setting some of the ideas we shall use later,
we shall assume the following fact without proof.
Recall that the \emph{$2$-core}, or simply \emph{core}, $C(G)$
of a graph $G$, introduced in~\cite{BBkcore}, is the maximal subgraph with minimum degree at least~2.

\begin{proposition}\label{factone}
Let $\la>1$ be constant. There is a constant $c=c(\la)>0$ such that
$G_n=G(n,\la/n)$ has the following properties with probability
$1-o(n^{-1/2})$: the core $C(G_n)$
of $G_n$ has a unique component $\cC_1$ with at least $cn$ vertices, and
$\cC_1$ is a subgraph of the largest component of $G_n$;
furthermore, $G_n$ has at least $cn$ isolated vertices.\noproof
\end{proposition}

Here then is our illustration
of smoothing for the Erd\H os--R\'enyi model, in the simple case
of constant branching factor.
In this case the central limit theorem
was established by Pittel and Wormald~\cite{PWio} and the local one
by Luczak and \L uczak~\cite{LuczakLuczak};
our aim here is to show how one can deduce one 
from the other.

\begin{theorem}\label{th_eg}
Let $p=p(n)=\la/n$ where $\la>1$ is constant, set $G_n=G(n,p)$
and let $A_n=L_1(G_n)$. If $(A_n)$ satisfies a global limit theorem
with $\sigma_n=\Theta(\sqrt{n})$ then it satisfies the corresponding
local limit theorem.
\end{theorem}
\begin{proof}
Given any graph $G$, let $G^-$ be the \emph{reduced graph}
obtained from $G$ by deleting all pendent edges incident
with the core $C(G)$ of $G$.
In other words, $G^-$ is the spanning
subgraph of $G$ obtained by deleting those edges
$e=vw$ in which $v$ has degree 1 and $w$ is in $C(G)$.
Note that $G$ and $G^-$ have the same core.
It follows that if $H$ is any graph that can arise as $G^-$ for some $G$,
then a graph $G$ with $V(G)=V(H)$ has $G^-=H$ if and only if $G$
is formed from $H$ in the following way:
for each isolated vertex $v$ of $H$, either
do nothing, or add an edge from $v$ to some vertex $w$ of the core $C(H)$ of $H$.
Since the probability of a graph $G$ in the model $G(n,p)$
is proportional to $(p/(1-p))^{e(G)}$, it follows that for any graph $H$
whose core $C(H)$ has $m$ vertices, 
the conditional distribution of $G_n=G(n,p)$ given that $G_n^-=H$
may be described as follows:

\smallskip
for each isolated vertex $v$ of $H$, with probability $pm/(pm+1-p)$ pick
a uniformly random vertex $w$ of $C(H)$ and join $v$ to $w$; otherwise do nothing. The
decisions associated to different $v$ are independent.

\smallskip
Let $\cF_n$ be the $\sigma$-algebra generated by the random variable $G_n^-$,
let $X_n$ be the number of vertices in the component of $H=G_n^-$ containing
the largest component $\cC_1$ of its core (chosen according to any
fixed rule if there is a tie), and let $Y_n$
be the number of vertices `rejoined' \emph{to this component $\cC_1$} when constructing
$G_n$ from $G_n^-$ as above.
Let $A_n'=X_n+Y_n$, noting that whenever $\cC_1$
is a subgraph of the largest component of $G_n$, we have $A_n'=L_1(G_n)$.
Clearly, $X_n$ is $\cF_n$-measurable.
Moreover, from the independence over vertices $v$,
the conditional distribution of $Y_n$ given $\cF_n$ is the binomial 
distribution $\Bi(i(G_n^-),\pi)$ where $i(H)$ denotes the number
of isolated vertices of a graph $H$ and
$\pi=\pi(G_n^-)=p|\cC_1|/(p|C(G_n^-)|+1-p)$.

Let $c>0$ be the constant appearing in Proposition~\ref{factone}.
Let $E_n$ be the event that the core $C(G_n)=C(G_n^-)$ has a unique component
with at least $cn$ vertices, and that $i(G_n^-)\ge cn$.
Note that $E_n\in \cF_n$. Also, since $i(G_n^-)\ge i(G_n)$,
by Proposition~\ref{factone} we have $\Pr(E_n)=1-o(n^{-1/2})$.
Whenever $E_n$ holds we have $c\le p|\cC_1|\le p|C(G_n^-)|=O(1)$ so, since $1-p\sim 1$,
the probability $\pi$ is bounded away from $0$ and $1$.
Hence, since $i(G_n^-)\ge cn$, the variance $i(G_n^-)\pi(1-\pi)$
of the (binomial) conditional distribution of $Y_n$ is at least $an$ for some constant $a>0$.
Letting $\cD_n$ be the family of all binomial distributions 
with variance at least $an$, then whenever $E_n$ holds,
the conditional distribution of $Y_n$ given $\cF_n$ is in $\cD_n$.
As noted above, the sequence $(\cD_n)$ is $\sqrt{n}$-smooth.

Recall that $A_n'=X_n+Y_n$ is the number of vertices
in the component of $G_n$ containing the largest component $\cC_1$ of $C(G_n)=C(G_n^-)$ (chosen according
to any fixed rule if there is a tie) so, by Proposition~\ref{factone}, $A_n'=L_1(G_n)$ with probability
at least $1-o(n^{-1/2})$.
Since $E_n$ holds whp, the conditional
distribution of $A_n'$ given $E_n$ satisfies the same global limit
theorem as the unconditional distribution of $A_n=L_1(G_n)$ does;
let $\mu_n$ and $\sigma_n=\Theta(\sqrt{n})$ be the parameters of this global limit theorem, and $\phi$ the
associated limiting density function.
Having conditioned on $E_n$, we now apply Lemma~\ref{ss}, which involves conditioning further on $\cF_n$
and using the fact that $(\cD_n)$ is $\sqrt{n}$-smooth.\footnote{To spell this out,
let $(\Omega_n,\Pr_n)$ be the (finite) probability space on which $G_n$ is defined, and let
$\Q_n$ be the probability measure $\Pr_n(\cdot\mid E_n)$ on $\Omega_n$. We apply Lemma~\ref{ss}
to the sequence of probability spaces $(\Omega_n,\Q_n)$, on which the random variables $A_n'$
satisfy the required global limit theorem. Since $E_n\in\cF_n$, then
when $\omega\in E_n$ we have $\Q_n(\cdot\mid \cF_n)(\omega)=\Pr_n(\cdot\mid \cF_n)(\omega)$ 
(by the tower-law). So, working on $(\Omega_n,\Q_n)$, when $\omega\in E_n$
the conditional distribution of $Y_n$ given $\cF_n$
is in $\cD_n$; what happens when $\omega\notin E_n$ is irrelevant since $\Q_n(E_n^\cc)=0$.
Hence Lemma~\ref{ss} gives an asymptotic formula for $\Q_n(A_n'=x_n)=\Pr_n(A_n'=x_n\mid E_n)$.}
We obtain the result that for any $x_n$ satisfying $x_n-\mu_n=O(\sqrt{n})$ we have
\[
 \Pr(A_n'=x_n\mid E_n) = \frac{\phi\bb{(x_n-\mu_n)/\sigma_n}}{\sigma_n} + o(n^{-1/2}).
\]
Since $\Pr(E_n)=1-o(n^{-1/2})$ and $\Pr(A_n'\ne A_n)=o(n^{-1/2})$ we have
$\Pr(A_n=x_n)=\Pr(A_n'=x_n\mid E_n)+o(n^{-1/2})$, giving the result.
\end{proof}

\subsection{Smoothing in the proof of Theorem~\ref{thprob}}\label{ss_complications}

In the rest of the paper we shall use a version of the above
technique to prove Theorem~\ref{thprob}. Since this proof is rather
long, and on reading (or writing!) it for the first time one might wonder
why it is so complicated, in this section we outline some of the problems
that occur when adapting the proof of Theorem~\ref{th_eg}.
Some of these concern the transition from graphs to hypergraphs,
some arise when allowing $\eps\to 0$, and some concern the extension
to a bivariate result. It is allowing $\eps\to 0$ that turns out to
cause by far the most difficulty. (Recall that
$p=\la(n) (r-2)!n^{-r+1}$ where $\la(n)=1+\eps(n)$ is the `branching factor'.)

Firstly, it turns out that (in both the graph and hypergraph cases)
the number of vertices of degree $1$ joined
directly to the core is $\Theta(\eps^2 n)$.  This means that the
variance obtained by deleting and reattaching such vertices will be
$\Theta(\eps^2 n)$, which is much smaller than the variance
$\Theta(n/\eps)$ of $L_1=L_1(\Hrnp)$ when $\eps\to 0$.  For this reason we need
to remove and reattach larger trees; indeed, it turns out that we need
to consider trees up to size $\Theta(\eps^{-2}$), which is essentially
the largest size that appears. (The bulk of the variance comes from
the large trees.) This complicates things, since each
tree contributes a different number of vertices to the giant
component. 

Secondly, there are various `good events' $E$ that we need to hold
for various parts of our smoothing argument.
As in the simple example above, one is that the core is
not too much smaller than it should be, and another is that the largest component
of the core is contained in the largest component of the whole
graph. Some of the bad events $E^\cc$ turn out to have probability
$\exp(-\Theta(\eps^3 n))$ (since the core is really characterized by
the kernel, which has $\Theta(\eps^3 n)$ vertices). So if $\eps^3n\to\infty$ slowly,
the unconditional probabilities of these events may be much larger than the probabilities
such as $\Pr(L_1=x_n)=\Theta(\sqrt{\eps/n})$ that we wish to estimate. The solution
is to show that $\Pr(E\mid L_1=x_n)=1-o(1)$, so
$\Pr(L_1=x_n) \sim \Pr(\{L_1=x_n\} \cap E)$. Then we can effectively
condition on $E$ (though being careful to keep independence where it is needed).

Thirdly, unlike for graphs, in the hypergraph case, even the simple
operation of deleting all `pendant edges' attached to the core (i.e.,
hyperedges with one vertex in the core and the other vertices
in no other hyperedges) is not so simple to invert. The inverse
involves selecting \emph{disjoint} sets of $r-1$ isolated vertices
to rejoin to the core. The condition that the sets must be disjoint
means that the number that do rejoin no longer has a binomial distribution.
We deal with this by randomly \emph{`marking'} some vertices throughout
the graph. Roughly speaking, we detach pendant edges attached either
to the core or to marked vertices, meaning that we remember
that a certain $(r-1)$-tuple was attached either to the core
or to a marked vertex. Then all choices of \emph{where} to reattach
the tuples do turn out to be independent. Of course, we actually detach larger
trees, not just pendant edges. In fact, rather than consider individual
trees, we shall directly study the
forests attached to the core and to a suitable set of marked vertices.

Finally, for the bivariate result we need to show that the nullity $N_1$
of the largest component also has a smooth
distribution; for this we use the same basic smoothing technique applied in a
different (and much simpler) way than for $L_1$. Fortunately, since
our smoothing argument for $L_1$ involves operations on the hypergraph
that do not affect $N_1$, these two
separate smoothing arguments combine to give the \emph{joint} smoothness
of $L_1$ and $N_1$ needed to prove Theorem~\ref{thprob}.

One might wonder whether our approach is really easier than (or indeed
different from) proving a local limit theorem directly.  Whether or not it is easier, the
fact remains that the local limit theorem was previously only known for restricted
ranges of the key parameter $\eps(n)$.  As to whether the approaches are genuinely
different, we believe that the answer is `yes'.  A key observation is that we study only
\emph{part} of the variation in the size of the giant component. The general method means
that, writing $\sigma_n^2$ for the variance of the quantity ($L_1$ or $N_1$)
we are studying, our `smoothing distribution' needs variance $\Theta(\sigma_n^2)$, but it can be an
arbitrarily small constant times $\sigma_n^2$.  This is vital since it means that in many of
our estimates we have a constant
factor elbow room.  This is unlikely to be the case in any direct proof of the local limit theorem,
since it would
lead to a significant error in the variance of $L_1$ or $N_1$. Here the variances of $L_1$
and $N_1$ are part
of the \emph{input} (the global limit assumption), and we really are establishing only
smoothness, rather than reevaluating the whole distribution.

\subsection{Organization of the rest of the paper}\label{ss_org}

The rest of the paper is organized as follows.
In Section~\ref{sec_key} we state two results from~\cite{BR_hyp2} that we shall need;
one of these is the global (central) limit theorem corresponding to Theorem~\ref{thprob}.
Then we state two key intermediate results, Theorems~\ref{thexcess} and~\ref{thL1}.
The first establishes smoothness of $N_1$, showing (a little more than) that
nearby values have almost equal probabilities. The second establishes (essentially) smoothness of the
distribution of $L_1$ conditional on $N_1$; as we note in the next section, these
results easily imply Theorem~\ref{thprob}.

In Section~\ref{sec_mr} we prove Theorem~\ref{thexcess}, using multi-round exposure arguments
reminiscent of those used by Behrisch, Coja-Oghlan and Kang~\cite{BC-OK2a}. In the subsequent
sections we prepare the ground for the (much more complicated) proof of Theorem~\ref{thL1}.
First, in Section~\ref{sec_tc} we present a result of Selivanov~\cite{Selivanov} enumerating hypergraph forests
subject to certain constraints, and a simple consequence concerning random forests. Then, in
Section~\ref{sec_dist}, we use Selivanov's formula to show that a certain distribution
associated to detaching and reattaching forests from the core and `marked' vertices
is $\sqrt{n/\eps}$-smooth as defined earlier in this section, so it can play
the role of $Y_n$ above when studying the distribution of $L_1$.
Next, in Section~\ref{sec_dual}, we
state a precise form of the supercritical/subcritical duality result for the random hypergraph $\Hrnp$;
in Section~\ref{sec_sub} we use this to establish some properties of the `small' components of $\Hrnp$
that we shall need later. In Section~\ref{sec_ext} we formally define `marked vertices'
and the extended core of
$\Hrnp$, and show that with high conditional probability it has the properties we need.
After this preparation, in Section~\ref{sec_main} we prove Theorem~\ref{thprob};
in Section~\ref{sec_deduce} we show that Theorem~\ref{thenum} follows.
Finally, in the Appendix we give detailed calculations comparing our formulae
with those in~\cite{BC-OK2pre,BC-OK2abs,BC-OK2b,BCMcK,KL_sparse,SatoWormald}.

\section{The key ingredients}\label{sec_key}

In this section we state two results from~\cite{BR_hyp2} that we shall need
as `inputs' to our smoothing arguments. Then we state our two main intermediate
results, and show how they combine to give Theorem~\ref{thprob}.

\subsection{Inputs}

Building on methods we used in~\cite{BR_hyp} to prove the central limit
theorem for $L_1=L_1(\Hrnp)$,
in~\cite{BR_hyp2} we proved the following bivariate (global) central limit theorem
for the order $L_1$ and nullity $N_1$ of the largest component of $\Hrnp$.
Here, and throughout, $\rho_{r,\la}$ and $\rho_{r,\la}^*$ are as defined in \eqref{rkldef}
and \eqref{rhosdef}.

\begin{theorem}\label{thglobal2}
Let $r\ge 2$ be fixed, and let $p=p(n)=(1+\eps)(r-2)!n^{-r+1}$ where $\eps=\eps(n)\to 0$
and $\eps^3n\to\infty$. Let $L_1$ and $N_1$ be the order and nullity of the largest
component $\cL_1$ of $\Hrnp$. Then
\[
 \left( \frac{L_1-\rho_{r,\la}n}{\sqrt{2n/\eps}} ,
  \frac{N_1-\rho_{r,\la}^*n}{\sqrt{10/3}(r-1)^{-1}\sqrt{\eps^3 n}} \right)
\dto (Z_1,Z_2),
\]
where $\dto$ denotes convergence in distribution, and $(Z_1,Z_2)$
has a bivariate Gaussian distribution
with mean~$0$, $\Var[Z_1]=\Var[Z_2]=1$ and $\Covar[Z_1,Z_2]=\sqrt{3/5}$.\noproof
\end{theorem}

In particular, recalling \eqref{newrasympt}, $L_1$ is asymptotically Gaussian with mean
$\Theta(\eps n)$ and variance $\Theta(n/\eps)$,
and $N_1$ is asymptotically Gaussian with mean $\Theta(\eps^3n)$ and variance $\Theta(\eps^3 n$).

In Section~\ref{sec_mr} we shall need the following large-deviation bounds on $L_1$
and $L_2$, the order of the second largest component of $\Hrnp$; this result is
also proved in~\cite{BR_hyp2}.

\begin{theorem}\label{thsupertail}
Let $r\ge 2$ be fixed, and let $p=p(n)=(1+\eps)(r-2)!n^{-r+1}$ where $\eps=O(1)$
and $\eps^3n\to\infty$. If $\omega=\omega(n)\to\infty$ and 
$\omega=O(\sqrt{\eps^3n})$ then
\begin{equation}\label{L1conc}
 \Pr\Bb{ |L_1(\Hrnp)-\rho_{r,\la} n| \ge \omega\sqrt{n/\eps} } = \exp(-\Omega(\omega^2)).
\end{equation}
Moreover, if $L=L(n)$ satisfies $\eps^2L\to\infty$ and $L=O(\eps n)$, then
\[
 \Pr( L_2(\Hrnp) > L ) \le C \frac{\eps n}{L} \exp(-c \eps^2L),
\]
for some constants $c,C>0$.%
\noproof
\end{theorem}
Here, as usual, the constants $c,C$ and the implicit constant in the $\Omega$ notation in \eqref{L1conc}
are allowed to depend on all previous choices: on $r$, the function $p(n)$, and the functions $\omega(n)$
and $L(n)$; see Remark~\ref{rem:asy1}.

\subsection{Main steps}

Theorem~\ref{thprob} is the bivariate local limit version of Theorem~\ref{thglobal2}.
To deduce it from Theorem~\ref{thglobal2}, we must show that `nearby' potential
values of the pair $(L_1,N_1)$ have essentially the same probability.
(Recalling \eqref{nulldef}, for $(s,t)$ to be a potential value, $r-1$ must divide $s+t-1$.)
We proceed in two stages. In the first, we show that $N_1$ has a smooth distribution,
which will already allow us to prove Theorem~\ref{thN1}.
More precisely,
we shall prove the following result in Section~\ref{sec_mr}.
We consider the pair $(L_1-(r-2)N_1,N_1)$
rather than $(L_1,N_1)$ for technical reasons that will become clear
during the proof; this makes little difference, since the
standard deviation of $N_1$ is much smaller than that of~$L_1$.

\begin{theorem}\label{thexcess}
Let $r\ge 2$ be fixed, and let $p=p(n)=(1+\eps)(r-2)!n^{-r+1}$ where $\eps=\eps(n)=O(1)$
and $\eps^3n\to\infty$. For any sequences $(t_n)$ and $(t_n')$ with
$t_n,t_n'\ge 0$ and $t_n-t_n'=o(\sqrt{\eps^3n})$, and any $I_n\subset\ZZ$,
we have
\begin{multline*}
 \Pr\bb{ N_1=t_n\hbox{ and }L_1-(r-2)N_1\in I_n } - \Pr\bb{ N_1=t_n'\hbox{ and }L_1-(r-2)N_1\in I_n } \\
 = o((\eps^3n)^{-1/2}).
\end{multline*}
\end{theorem}

By Proposition~\ref{gtol}, Theorems~\ref{thglobal2} and~\ref{thexcess} imply Theorem~\ref{thN1}.
Indeed, Theorem~\ref{thglobal2} immediately implies that $N_1=N_1(\Hrnp)$
satisfies a central limit theorem with parameters $\rho^*_{r,\la}n$ for the mean
and $\sigma_n^*=\sqrt{10/3}(r-1)^{-1}\sqrt{\eps^3n}$ for the standard deviation.
Since $\sigma_n^*=\Theta(\sqrt{\eps^3n})$,
taking $I_n=\ZZ$ in Theorem~\ref{thexcess} we see that if $t_n-t_n'=o(\sigma_n^*)$
then $\Pr(N_1=t_n)-\Pr(N_1=t_n')=o(1/\sigma_n^*)$. Hence Theorem~\ref{thN1}
follows by Proposition~\ref{gtol}.

In the next result, and much of the rest of the paper, we only consider potential
values of $L_1$ in a `typical' range. To be precise, having fixed a function $p(n)$
(and thus $\eps(n)$ and $\la(n)$) satisfying our Weak Assumption~\ref{A1},
let $\delta=\delta(n)$ satisfy
\begin{equation}\label{deltadef}
 \delta\to 0\hbox{\quad and\quad}\delta\ge (\eps^3n)^{-1/3},
\end{equation}
and let
\begin{equation}\label{rangedef}
 \range = \range_n =\range_{n,p} = [(1-\delta)\rho_{r,\la}n, (1+\delta)\rho_{r,\la}n].
\end{equation}
(To be concrete, we may just set $\delta=(\eps^3 n)^{-1/3}$, but the precise value is irrelevant
as long as the conditions above hold.)
Recalling \eqref{newrasympt} and \eqref{rlprop}, under our Weak Assumption~\ref{A1} we have
$\rho_{r,\la}=\Theta(\eps)$
and $\rho_{r,\la}$ bounded away from $1$. Hence there are constants $c,C>0$ (depending on the function
$\eps(n)$) such that, for $n$ large enough,
\begin{equation}\label{rangecont}
 \range_n \subseteq [c\eps n, C\eps n] \hbox{\quad and\quad} \range_n\subseteq [c\eps n,(1-c)n].
\end{equation}
By Theorem~\ref{thsupertail}, applied with
$\omega=\omega(n)=\delta\rho_{r,\la}n/(\sqrt{n/\eps})=\Theta(\delta\sqrt{\eps^3n})$, under
our Weak Assumption~\ref{A1} we have
\begin{equation}\label{rangewvhp}
 \Pr(L_1(\Hrnp)\notin\range) \le \exp(-c\delta^2\eps^3n)
 \le \exp(-c(\eps^3n)^{1/3}) = O(1/(\eps^3 n)).
\end{equation}

The bulk of the paper will be devoted to the proof of the following result establishing,
essentially, smoothness of the conditional distribution of $L_1$ given~$N_1$.

\begin{theorem}\label{thL1}
Let $r\ge 2$ be fixed, let $p=p(n)=(1+\eps)(r-2)!n^{-r+1}$
where $\eps=\eps(n)$ satisfies $\eps\to 0$ and $\eps^3 n\to\infty$,
and set $L_1=L_1(\Hrnp)$. Define $\range=\range_n$ as in \eqref{rangedef}.
If $(x_n)$, $(y_n)$ and $(t_n)$ are sequences of integers with
$x_n,y_n\in \range_n$, $x_n-y_n=o(\sqrt{n/\eps})$, $t_n\ge 2$, and
\[
 x_n \equiv  y_n \equiv 1-t_n \hbox{\quad modulo }(r-1),
\]
then
\[
 \Pr(L_1=x_n, N_1=t_n)-\Pr(L_1=y_n, N_1=t_n) = o(1/(\eps n)).
\]
\end{theorem}

Theorems~\ref{thexcess} and Theorem~\ref{thL1} will be proved in Sections~\ref{sec_mr}--\ref{sec_main}.
First, let us show how they imply Theorem~\ref{thprob}. Although the argument is straightforward,
since Theorem~\ref{thL1} is our main result, we shall spell out the details.

\begin{proof}[Proof of Theorem~\ref{thprob}.]
Throughout we fix $r\ge 2$, and a function $p=p(n)=(1+\eps)(r-2)!n^{-r+1}$ such that $\eps=\eps(n)$
satisfies $\eps\to 0$ and $\eps^3n\to\infty$.
Let
\[
 \sigma_n=\sqrt{2n/\eps}\hbox{\qquad and\qquad}
  \sigma_n^*=\sqrt{10/3}(r-1)^{-1}\sqrt{\eps^3n}=\Theta(\sqrt{\eps^3n}).
\]
Indicating the dependence on $n$ for once, let $L_{1,n}=L_1(\Hrnp)$ and 
$N_{1,n}=N_1(\Hrnp)$. It will be convenient to consider
the linear combination
\[ 
 \tL_{1,n}=L_{1,n}-(r-2)N_{1,n}.
\]
Recalling the definitions \eqref{rkldef} and \eqref{rhosdef} of $\rho_{r,\la}$ and $\rho^*_{r,\la}$,
set
\[
 \trho_{r,\la}=\rho_{r,\la}-(r-2)\rho^*_{r,\la}.
\]
Since $\sigma_n^*=o(\sigma_n)$,  
Theorem~\ref{thglobal2} immediately implies that
\begin{equation}\label{ctilde}
 \left( \frac{\tL_{1,n}-\trho_{r,\la}n}{\sigma_n} ,
  \frac{N_{1,n}-\rho_{r,\la}^*n}{\sigma_n^*} \right)
\dto (Z_1,Z_2),
\end{equation}
where $(Z_1,Z_2)$
has a bivariate Gaussian distribution
with mean~$0$, $\Var[Z_1]=\Var[Z_2]=1$ and $\Covar[Z_1,Z_2]=\sqrt{3/5}$;
the probability density function $f(a,b)$ of this distribution is given in \eqref{pdf}.

Let $(x_n)$ and $(y_n)$ be sequences with $x_n=\rho_{r,\la}n+O(\sigma_n)$ (i.e., $\sup_n |x_n-\rho_{r,\la}n|/\sigma_n<\infty$)
and $y_n=\rho^*_{r,\la}n+O(\sigma^*_n)$, such that $x_n+y_n-1$ is a multiple of $r-1$ for all $n$;
our aim is to prove \eqref{pointprob}
for these sequences. By a standard subsequence argument, we may assume without
loss of generality that
\[
 \frac{x_n-\rho_{r,\la}n}{\sigma_n}\to x
\hbox{\qquad and\qquad}
  \frac{y_n-\rho_{r,\la}^*n}{\sigma^*_n}\to y
\]
for some $x,y\in \RR$.
Since the density $f(a,b)$ is continuous and strictly positive, what
we must show is exactly that
\begin{equation}\label{aimo1}
 \Pr\bb{ L_{1,n}=x_n,\, N_{1,n}=y_n } = \frac{(r-1)f(x,y)+o(1)}{\sigma_n\sigma^*_n}.
\end{equation}
(As usual, the $o(1)$ term represents a quantity that tends to $0$ as $n\to\infty$; the
rate may depend on all the choices made so far.)

It will be convenient to consider more explicit reformulations of Theorems~\ref{thexcess} and~\ref{thL1}.
By Theorem~\ref{thexcess}, for every constant $\alpha>0$ there is a constant $\beta>0$
and an integer $n_0$ such that the following holds:
whenever $n\ge n_0$, $t,t'\ge 0$ with $|t-t'|\le \beta\sigma_n^*$, and $I\subset \ZZ$, then
\begin{equation}\label{eref}
 \bigl|\Pr\bb{ N_{1,n}=t,\, \tL_{1,n}\in I } - \Pr\bb{ N_{1,n}=t',\, \tL_{1,n}\in I } \bigr| \le  \alpha/\sigma_n^*.
\end{equation}
Indeed, if \eqref{eref} does not hold, then picking an $\alpha$ for which it fails,
for each $k$ we may find an $n_k>n_{k-1}$ and $I_{n_k}$, $t_{n_k}$ and $t_{n_k}'$
such that $|t_{n_k}-t_{n_k}'|\le \sigma_n^*/k$ and
$\Pr\bb{N_{1,n_k}=t_{n_k},\, \tL_{1,n_k}\in I_{n_k}}$ and 
$\Pr\bb{N_{1,n_k}=t_{n_k}',\, \tL_{1,n_k}\in I_{n_k}}$ differ by at least $\alpha/\sigma_n^*$. 
Completing the sequences $t_n$, $t_n'$ and $I_n$ appropriately gives a counterexample
to Theorem~\ref{thexcess}.

Similarly, since $\sigma_n=\Theta(\sqrt{n/\eps})$ and $\sigma_n\sigma^*_n=\Theta(\eps n)$,
Theorem~\ref{thL1} implies that for any constant $\eta>0$
there are $\gamma_1>0$ and $n_0$ such that whenever $n\ge n_0$, $t\ge 2$ and
$s, s'\in \range_n$ with $|s-s'|\le \gamma_1\sigma_n$ and  $s\equiv s'\equiv 1-t$ modulo $r-1$, then
\begin{equation}\label{vref}
  \bigl|\Pr(L_{1,n}=s,\, N_{1,n}=t)-\Pr(L_{1,n}=s',\, N_{1,n}=t)\bigr| \le \frac{\eta}{\sigma_n\sigma^*_n}.
\end{equation}

Let $\eta>0$ be constant. We shall show that if $n$ is large enough, then
\begin{equation}\label{aimeta}
 \left| \Pr\bb{ L_{1,n}=x_n,\, N_{1,n}=y_n } - \frac{(r-1)f(x,y)}{\sigma_n\sigma^*_n} \right|
 \le \frac{4r\eta}{\sigma_n\sigma^*_n},
\end{equation}
proving \eqref{aimo1} and thus Theorem~\ref{thprob}.

Define $\gamma_1$ as in \eqref{vref}.
Since $f(\cdot,\cdot)$ is continuous at $(x,y)$, we may choose $\gamma_2>0$ such that whenever $|a-x|\le \gamma_2$
and $|b-x|\le \gamma_2$, we have $|f(a,b)-f(x,y)|\le \eta$. Set $\gamma=\min\{\gamma_1,\gamma_2\}$
and let
\[
 I_n=[\trho_{r,\la}n+(x-\gamma/2)\sigma_n,\, \trho_{r,\la}n+(x+\gamma/2)\sigma_n].
\]
For $n\ge 1$ and $t\ge 0$ let
\[
 \pi_{n,t} = \Pr\bb{ N_{1,n}=t,\, \tL_{1,n}\in I_n }.
\]
By \eqref{eref}, applied with $\alpha=\eta\gamma$, there is a constant $\beta>0$,
which we may assume to be less than $\gamma_2$, such that for all large enough $n$
we have
\begin{equation}\label{pidiff}
 | \pi_{n,t} - \pi_{n,t'} |  \le  \eta\gamma/\sigma_n^*
\end{equation}
whenever $t,t'$ lie in the interval
\[
 J_n=[\rho_{r,\la}^*n+(y-\beta/2)\sigma^*_n,\, \rho_{r,\la}^*n+(y+\beta/2)\sigma^*_n].
\]
(Here we have used the fact that for $n$ large $J_n$ consists only of positive integers,
which holds since $\sigma^*_n=o(\rho_{r,\la}^*n)$.)
Let
\[
 a_n = \frac{1}{|J_n|} \sum_{t\in J_n} \pi_{n,t} 
 = \frac{1}{|J_n|} \Pr\left( (\tL_{1,n},N_{1,n}) \in I_n\times J_n \right).
\]
Since $\sigma^*_n\to\infty$ and $\beta$ is constant, we have $|J_n|\sim \beta\sigma^*_n$.
It follows from \eqref{ctilde} that
\[
 a_n \sim \frac{1}{\beta\sigma^*_n} 
   \int_{a=x-\gamma/2}^{x+\gamma/2} \int_{b=y-\beta/2}^{y+\beta/2} f(a,b) \,\da\,\db.
\]
Since $\beta$ and $\gamma$ are at most $\gamma_2$, for all $(a,b)$ in the
region of area $\beta\gamma$ over which we integrate we have $|f(a,b)-f(x,y)|\le \eta$.
Hence, for $n$ large enough,
\[
 |a_n-f(x,y)\gamma/\sigma^*_n| \le 2\eta\gamma/\sigma^*_n.
\]

Now $a_n$ is the average of the values $\pi_{n,t}$ over $t\in J_n$, so the bound \eqref{pidiff} implies
that all of these values are within $\eta\gamma/\sigma^*_n$ of $a_n$.
For $n$ large enough, $y_n\in J_n$, so
\begin{equation}\label{pyn}
 |\pi_{n,y_n} -f(x,y)\gamma/\sigma^*_n |\le 3\eta\gamma/\sigma_n^*.
\end{equation}

Since the component of $\Hrnp$ with $L_{1,n}$ vertices and nullity $N_{1,n}$ is by definition
connected, \eqref{nulldef} gives $L_{1,n}+N_{1,n}\equiv 1$ modulo $r-1$. Hence
\begin{multline}\label{pisum}
 \pi_{n,y_n} = \Pr\bb{N_{1,n}=y_n,\, L_{1,n}-(r-2)y_n\in I_n} \\
 = \sum_{s\in S_n} \Pr\bb{L_{1,n}=s,\,N_{1,n}=y_n}
\end{multline}
where $S_n$ consists of all integers in $I_n+(r-2)y_n$ congruent to $1-y_n$ modulo $r-1$.
Hence
\begin{equation}\label{Ssize}
 |S_n| = \frac{|I_n|}{r-1}+O(1) = \frac{\gamma\sigma_n}{r-1}+O(1) \sim \frac{\gamma\sigma_n}{r-1}.
\end{equation}

Recall that $x_n=\rho_{r,\la}n+x\sigma_n+o(\sigma_n)$ and
$y_n=\rho^*_{r,\la}n+O(\sigma_n^*)=\rho^*_{r,\la}n+o(\sigma_n)$.
Thus $x_n-(r-2)y_n=\trho_{r,\la}n+(x+o(1))\sigma_n$ and so for $n$ large enough $x_n-(y-2)y_n\in I_n$
and so $x_n\in S_n$.
Furthermore $s\in S_n$ implies $|s-\rho_{r,\la}n|\le |x_n-\rho_{r,\la}n|+\gamma\sigma_n = O(\sigma_n)$.
Hence, for $n$ large enough, $S_n\subseteq \range_n$.
It follows by \eqref{vref} that the probabilities summed in \eqref{pisum} are all within
$\eta/(\sigma_n\sigma_n^*)$ of each other and hence of their average, which by \eqref{pyn}
and \eqref{Ssize} is within $3r\eta/(\sigma_n^*\sigma_n)$ of $(r-1)f(x,y)/(\sigma_n\sigma_n^*)$.
Since $x_n\in S_n$ this concludes the proof of \eqref{aimeta}
and hence that of Theorem~\ref{thprob}.
\end{proof}

\section{Smoothing the excess: multi-round exposure}\label{sec_mr}

In this section we prove Theorem~\ref{thexcess}. The arguments in this section do
not obviously simplify in the case $\eps\to 0$, so throughout we work with
our Weak Assumption~\ref{A1}, i.e.,
we let $p=p(n)=(1+\eps)(r-2)!n^{-r+1}$
where $\eps=\eps(n)$ satisfies $\eps=O(1)$ and $\eps^3 n\to\infty$.

Set 
\[
 p_1=(1+\eps/2) (r-2)! n^{-r+1}
\]
and define $p_2$ by $p=p_1+p_2-p_1p_2$, noting that 
\begin{equation}\label{p2sim}
 p_2\sim (\eps /2)(r-2)! n^{-r+1} = \Theta(\eps n^{-r+1}).
\end{equation}
Using a now standard idea originally due to Erd\H os and R\'enyi~\cite{ERgiant}, we shall view $\Hrnp$ as $H_1\cup H_2$
where $H_1$ and $H_2$ are independent, and $H_i$ has the distribution $\Hrnpi$.
To prove Theorem~\ref{thexcess} we first `reveal' (i.e., condition on) $H_1$. Then we reveal many but not
all edges of $H_2$. We do this in such a way that the remaining edges of $H_2$ must be of a simple
type. We then show that the conditional distribution of the number of these edges present is essentially binomial. Since
each will contribute $1$ to $N_1=n(\Hrnp)$, this will allow us to prove the result. The strategy is inspired
by a related argument of Behrisch, Coja-Oghlan and Kang~\cite{BC-OK1}, 
itself based on ideas of Coja-Oghlan, Moore and Sanwalani~\cite{C-OMS}, though the details
are very different since the objective is different. (Their argument is used to
`smooth' $L_1$ rather than $N_1$, and requires  $\eps$ bounded away from zero.)

We start with a simple lemma showing
that the distribution we shall use for smoothing is indeed smooth in the relevant sense.
\begin{lemma}\label{F'smooth}
Let $r\ge 3$ be fixed.
Given integers $i,\ell>0$ and a real number $0<\pi<1$, for $0\le a\le i/(r-2)$ let
\begin{equation}\label{nadef}
 n_a=n_{i,\ell,a}=\frac{1}{a!} \binom{i}{r-2}\binom{i-(r-2)}{r-2}\cdots \binom{i-(a-1)(r-2)}{r-2} \binom{\ell}{2}^a,
\end{equation}
and let $Y_{i,\ell,\pi}$ be the probability distribution on the non-negative integers defined by
\[
 \Pr(Y_{i,\ell,\pi}=a) = p_a = p_{i,\ell,\pi,a} =  \pi^a n_a\ \Big/\ \sum_{b=0}^{i/(r-2)} \pi^b n_b.
\]
Let $\eps=\eps(n)$ satisfy $\eps^3n\to\infty$ and $\eps=O(1)$, set $\sigma_0=\sigma_0(n)=\sqrt{\eps^3 n}$,
and let $i=i(n)$, $\ell=\ell(n)$ and $\pi=\pi(n)$ satisfy  $i=\Theta(n)$, $\ell=\Theta(\eps n)$
and $\pi= \Theta(\eps n^{-r+1})$.
Then, whenever $(y_n)$ and $(y_n')$ satisfy $y_n-y_n'=o(\sigma_0)$, we have
\begin{equation}\label{Ynsmooth1}
 \Pr(Y_n=y_n)-\Pr(Y_n=y_n') = o(1/\sigma_0),
\end{equation}
where $Y_n=Y_{i(n),\ell(n),\pi(n)}$.
\end{lemma}
Although the reader need not check this, Lemma~\ref{F'smooth} says that certain sequences $(\cD_n)$
of sets of probability distributions of the type $Y_{i,\ell,\pi}$ are $\sigma_0(n)$-smooth
in the sense of \eqref{smcond}.
\begin{proof}
Fix sequences $\eps(n)$, $i(n)$, $\ell(n)$ and $\pi(n)$ satisfying the conditions above; in what follows,
much of the time we suppress the dependence on $n$ in the notation.

Let $(x)_y$ denote the falling factorial $x(x-1)\cdots (x-y+1)$. Then, with $n$ fixed,
for $a+1\le i/(r-2)$ we have
\begin{equation}\label{qa}
 q_a=\frac{p_{a+1}}{p_a} = \frac{1}{a+1} \frac{\pi}{(r-2)!} \binom{\ell}{2} (i-a(r-2))_{r-2}.
\end{equation}
The sequence $(q_a)$ is strictly decreasing, so $(p_a)$ is unimodal.

For $a=a(n)$ satisfying $i-a(r-2)=\Omega(n)$, by the assumptions on $i$, $\ell$ and $\pi$ above we have
\[
 q_a=\Theta\bb{ (a+1)^{-1} (\eps n^{-r+1}) (\eps n)^2 n^{r-2} } = \Theta(\eps^3 n/(a+1)).
\]
For $i-a(r-2)=o(n)$ it is easy to see that $q_a=o(1)$. 
Let $a_0=a_0(n)$ be the minimal integer such that $q_{a_0}\le 1$.
Then we have $a_0=\Theta(\eps^3 n)$ and hence $i-a_0=\Omega(n)$.

Writing $\sigma_0=\sigma_0(n)=\sqrt{\eps^3n}$,
it follows from \eqref{qa} that for $a=a(n)=a_0+O(\sigma_0)$ we have
\begin{equation}\label{qbd}
 q_a = q_{a_0}(1+O(\sigma_0/a_0)) = q_{a_0}(1+O(\sigma_0^{-1})) =1+O(\sigma_0^{-1}).
\end{equation}
Since $q_a=p_{a+1}/p_a$, this has the following
consequence: for any sequences $a_1=a_1(n)$ and $a_2=a_2(n)$ such that $a_i=a_0+O(\sigma_0)$,
$a_1-a_2=o(\sigma_0)$ and $a_1<a_2$,
we have\footnote{To deduce \eqref{pasim} we need \eqref{qbd}
to hold uniformly in $a$ with $a_1(n)\le a< a_2(n)$. To see that it does, choose
the `worst-case' $a=a(n)$ in this range for each $n$ and apply \eqref{qbd} to the resulting
sequence.}
 \begin{equation}\label{pasim}
 p_{a_2}/p_{a_1} = \prod_{a_1\le a<a_2} q_a = (1+O(\sigma_0^{-1}))^{o(\sigma_0)} = 1+o(1).
\end{equation}
From the unimodality of $(q_a)$ and the definition of $a_0$ we have $\max_a p_a = p_{a_0}$.
It is easy to see that $p_{a_0}=O(1/\sigma_0)$: otherwise, we could use \eqref{pasim}
to deduce that $\sum_a p_a>1$, a contradiction.
Hence, $\max_a p_a=p_{a_0}=O(1/\sigma_0)$. Thus, from \eqref{pasim}, for $a_i=a_0+O(\sigma_0)$ we have
\begin{equation}\label{padiff}
 a_1-a_2=o(\sigma_0) \implies p_{a_2} - p_{a_1} = o(1/\sigma_0).
\end{equation}
For $a>a_0$, by unimodality we have
\[
 1=\sum_b p_b \ge \sum_{a_0<b\le a} p_b \ge (a-a_0) p_a,
\]
so if $(a-a_0)/\sigma_0\to\infty$ then $p_a=o(1/\sigma_0)$. Similarly,
if $(a_0-a)/\sigma_0\to\infty$ then $p_a=o(1/\sigma_0)$. It follows that \eqref{padiff}
holds for any sequences $a_1(n)$, $a_2(n)$ with $a_1-a_2=o(1/\sigma_0)$, which is exactly
\eqref{Ynsmooth1}.
\end{proof}


\begin{proof}[Proof of Theorem~\ref{thexcess}.]
Define $p_1$, $p_2$, $H_1$ and $H_2$ as at the start of the section, and set
\[
 \sigma_0= \sqrt{\eps^3n}.
\]
(Recall that, up to a constant factor, $\sigma_0^2$ is the variance of $N_1(\Hrnp)$.)
We shall first apply Theorem~\ref{thsupertail} to $H_1$, noting that $(\eps/2)^3n\to\infty$.
Let $\cC_1$ be the component of $H_1$ with the most vertices
(chosen according to any rule if there is a tie).
Since $\rho_{r,1+\eps/2}=\Theta(\eps)$, by Theorem~\ref{thsupertail} there are constants $0<c<C$ such that
the event
\[
 \cE_1 = \{ c\eps n \le  |\cC_1| \le C\eps n \}
\]
satisfies
\[
 \Pr( \cE_1^\cc) = \exp(-\Omega(\eps^3n)) = o( 1/\sigma_0).
\]
By the last part of Theorem~\ref{thsupertail},
\[
 \Pr( L_2(\Hrnp) \ge c\eps n) \le \exp(-\Omega(\eps^3n)) = o( 1/\sigma_0).
\]
Let $\cE_2$ be the event that $\cC_1$ is contained in the
largest component $\cL_1$ of $\Hrnp=H_1\cup H_2$. Since $H_1\subset\Hrnp$,
we have
\[
 \Pr(\cE_2^\cc) \le \Pr(\cE_1^\cc) + \Pr(L_2(\Hrnp)\ge c\eps n) =o(1/\sigma_0).
\]
Let $i(H)$ denote the number of isolated vertices in a hypergraph $H$.
It is easy to see that $\E[i(\Hrnp)]=\Theta(n)$. Let $c'$ be a constant
such that $\E[i(\Hrnp)]\ge 2c'n$ for large enough $n$,
and let $\cE_3$ be the event
\begin{equation}\label{e3def}
 \cE_3 = \{ i(\Hrnp)\ge c'n \}.
\end{equation}
Then standard concentration arguments (e.g., a simple application of the Hoeffding--Azuma inequality)
show that
\[
 \Pr(\cE_3^\cc) = \exp(-\Omega(n)) =o(1/\sigma_0).
\]

Reveal all edges of $H_1$, which of course determines $\cC_1$. We shall reveal some partial information
about $H_2$ in a two-step process.

First, test $r$-sets (i.e., potential edges) for their presence in $H_2$ according
to the following algorithm: if there is any untested $r$-set $e$ which does not consist of two
vertices in $\cC_1$ and $r-2$ vertices that are isolated in the current hypergraph $H$,
then pick some such $r$-set $e$ and test whether it is present in $H_2$. Otherwise, stop.
By the `current hypergraph' we mean
the hypergraph formed by the edges revealed so far, so $H_1\subset H\subset H_1\cup H_2=\Hrnp$.

Let $H$ be the hypergraph revealed at the end of the algorithm,
let $\cI$ be the set of isolated vertices of $H$, and let $U$ be the set of untested $r$-sets
when the algorithm stops. Then $U$
has a very simple form: it consists precisely of all $\binom{|\cC_1|}{2}\binom{|\cI|}{r-2}$
$r$-sets with two vertices in $\cC_1$ and $r-2$ in $\cI$.
To see this, note first that if there were any untested $r$-set not of this form, the algorithm
would not have stopped. Conversely, since any isolated vertices in the final hypergraph $H$ were isolated
throughout the running of the algorithm, and $\cC_1$ (a component of $H_1$, \emph{not}
of the current graph) does not change as the algorithm runs,
any $r$-set of this form cannot have been tested.

At this point, each untested edge is present independently with conditional probability $p_2$.

In the second step, we 
reveal the set $F$ of edges $e$ in $U$ present in $H_2$ with the property that some vertex
of $e\cap \cI$ is incident with one or more other edges of $H_2$. Let $\cI'$
be the set of vertices in $\cI$ not incident with edges in $F$.

Let $\cF$ denote the the $\sigma$-algebra generated by all the information revealed so far,
and let $F'$ be the set of edges of $H_2$ not yet revealed. Then $F'$ consists of edges
with two vertices in $\cC_1$ and $r-2$ in $\cI'$, with the corresponding subsets of $\cI'$
disjoint. Further more, given $\cF$ (which determines $\cC_1$ and $\cI'$),
any set $F'$ of edges satisfying this description is possible.
Let $Y_n=|F'|$; this will be our smoothing random variable. 
Recalling the definition \eqref{nadef} of $n_{i,\ell,a}$, there are exactly $n_{|\cI'|,|\cC_1|,a}$
possible sets $F'$ with $a$ edges. Let $\pi=p_2/(1-p_2)$.
Since the probability of a hypergraph in the model $\Hrnptwo$ is proportional to $\pi$
raised to the power of the number of edges,
we see that (for $r\ge 3$) the conditional distribution of $Y_n=|F'|$ given $\cF$ is exactly the distribution
$Y_{|\cI'|,|\cC_1|,\pi}$ defined in Lemma~\ref{F'smooth}.

Let $\cE$ be the event
\[
 \cE = \cE_1 \cap \{|\cI'|\ge c' n\},
\]
where $c'$ is as in the definition \eqref{e3def} of $\cE_3$. Note that $\cE$ is $\cF$-measurable.
Since every isolated vertex of $\Hrnp$ is in $\cI'$, we have
\begin{equation}\label{eprob}
 \Pr(\cE^\cc) \le \Pr(\cE_1^\cc)+\Pr(\cE_3^\cc) = o(1/\sigma_0).
\end{equation}
When $\cE$ holds, then $|\cC_1|=\Theta(\eps n)$ and $|\cI'|=\Theta(n)$; 
from \eqref{p2sim} we always have $\pi=p_2/(1-p_2)=\Theta(\eps n^{-r+1})$.
Let $(\omega_n)$ be a sequence of elements of the probability space(s) on which $\Hrnp$ is defined, with $\omega_n\in \cE=\cE_n$.
By Lemma~\ref{F'smooth},\footnote{For $r=2$ (which is not our main focus)
we cannot apply Lemma~\ref{F'smooth}. However, in this case
$F'$ is simply the set of edges of $H_2$ with both ends in $\cC_1$.
This has a binomial distribution with parameters $\Theta(\eps^2n^2)$ and $\Theta(\eps n^{-1})$;
the family of such distributions is $\sigma_0$-smooth, so \eqref{Ynsmooth} holds in this case also.}
for any such sequence $(\omega_n)$ and for any sequences $y_n$, $y_n'$ with $y_n-y_n'=o(\sigma_0)$ we have
\begin{equation}\label{Ynsmooth}
 \Pr\bb{ Y_n=y_n \mid \cF }(\omega_n) - \Pr\bb{Y_n=y_n' \mid \cF}(\omega_n) = o(1/\sigma_0).
\end{equation}

Fix sequences $t_n,t_n'\ge 0$ with $t_n-t_n'=o(\sigma_0)$ and a sequence $(I_n)$ of subsets of $\ZZ$.
Our aim is to show that
\begin{multline}\label{excessaim}
 \Pr\bb{ N_1=t_n,\, L_1-(r-2)N_1\in I_n }  \\ 
- \Pr\bb{ N_1=t_n',\, L_1-(r-2)N_1\in I_n } = o(1/\sigma_0).
\end{multline}
Let $\cC$ be the component of $H\supset H_1$ containing $\cC_1$, and $\cC'$
the component of $\Hrnp$ containing $\cC$ (and hence $\cC_1$). 
Let
\[
 X_n=n(\cC) \hbox{\quad and\quad} Z_n=|\cC|-(r-2)n(\cC) = |\cC|-(r-2)X_n.
\]
Then $X_n$ and $Z_n$ are $\cF$-measurable, so from \eqref{Ynsmooth}, for any $\omega_n\in \cE$ we have
\begin{multline*}
 \Pr\bb{ X_n+Y_n=t_n,\,Z_n\in I_n\mid \cF}(\omega_n) \\
  - \Pr\bb{X_n+Y_n=t_n',\,Z_n\in I_n\mid \cF}(\omega_n) = o(1/\sigma_0).
\end{multline*}
As usual, this bound holds uniformly in $\omega_n\in \cE=\cE_n$, since we are free to choose $\omega_n$
to maximize the difference. Taking the expectation, and recalling that $\cE$ is $\cF$-measurable and
$\Pr(\cE^\cc)=o(1/\sigma_0)$, it follows that
\begin{equation}\label{excessalmost}
 \Pr(X_n+Y_n=t_n,\,Z_n\in I_n) - \Pr(X_n+Y_n=t_n',\,Z_n\in I_n) = o(1/\sigma_0).
\end{equation}

Now each edge in $F'$ meets $\cC$ in two vertices, and has no vertices outside
$\cC$ in common with any other edge of $F'$. Thus
\[
 n(\cC') = X_n + Y_n\hbox{\quad and\quad} |\cC'| = |\cC| +(r-2)Y_n,
\]
so
\[
 |\cC'|-(r-2)n(\cC') 
 = |\cC|-(r-2)X_n =Z_n.
\]
When $\cE_2$ holds, then $\cC'=\cL_1$. Hence, whenever $\cE_2$ holds, we have
\begin{equation}\label{shift}
 N_1 = X_n +Y_n \hbox{\quad and\quad} L_1-(r-2)N_1 = Z_n.
\end{equation}
Recalling that $\Pr(\cE_2)=1-o(1/\sigma_0)$, our aim \eqref{excessaim} follows
from \eqref{excessalmost} and \eqref{shift},
completing the proof of Theorem~\ref{thexcess}.
\end{proof}

\section{Trees and forests}\label{sec_tc}

For $m\ge 2$, an \emph{$m$-cycle} in a hypergraph $H$ consists of distinct
vertices $v_1,\ldots,v_m$ and distinct edges $e_1,\ldots,e_m$
such that each $e_i$ contains both $v_i$ and $v_{i+1}$, with $v_{m+1}$
defined to be $v_1$. 
Thus a $2$-cycle consists of two edges sharing at least two vertices.
Note that an $m$-cycle corresponds to a cycle of length $2m$
in the bipartite vertex-edge incidence graph $\GI(H)$ associated to $H$.

A hypergraph $H$ is a \emph{tree} if it is connected and contains no cycles,
or, equivalently, if $H$ can be built up by starting with a single
vertex, and adding new edges one-by-one so that each meets
the current hypergraph in exactly one vertex. Note that $H$ is a tree
if and only if $\GI(H)$ is a tree.

By an \emph{$r$-tree} we simply mean an $r$-uniform hypergraph
that is a tree. An \emph{$r$-forest} is a vertex-disjoint union of $r$-trees.
For $A\subset V$, an \emph{$A$-rooted $r$-forest on $V$} is an $r$-forest
with vertex set $V$ such that each component contains exactly one
vertex from $A$; in particular, there are $|A|$ components.
Note that $A$-rooted $r$-forests on $V$ exist if and only
if $|V|=|A|+(r-1)k$ for some integer $k\ge 0$ (the number of edges).
For $r=2$, the formula $a n^{n-a-1}$ for the number of $[a]$-rooted 2-forests
on $[n]$ was observed by Cayley~\cite{Cayley} and proved by R\'enyi~\cite{Renyi}.
We shall make repeated use of the following generalization
to hypergraphs, due to Selivanov~\cite{Selivanov}.

\begin{lemma}\label{l_rtree}
Let $r\ge 2$, $a\ge 1$ and $k\ge 0$ be integers, and set $n=a+(r-1)k$.
The number $F_{a,k}=F_{a,k}^{(r)}$ of $[a]$-rooted $r$-forests on $[n]=\{1,2,\ldots,n\}$ satisfies
\begin{equation}\label{Fan}
 F_{a,k} = a n^{k-1} \partit{k}{r-1},
\end{equation}
where
\[
 \partit{k}{t} = \frac{(kt)!}{k!\, t!^k}
\]
is the number of partitions of a set of size $kt$ into $k$ parts of size $t$.\noproof
\end{lemma}
For completeness we give a proof in the Appendix, since the original source is perhaps
a little obscure. (We only became aware of it from Karo\'nski and \L uczak~\cite{KL_sparse}).

One consequence of Lemma~\ref{l_rtree} is the following surprisingly simple
bound on the expected number of vertices at a given distance from the root set in a random
$[a]$-rooted $r$-forest. Recall that $(x)_y$ denotes the falling factorial $x(x-1)\cdots (x-y+1)$.

\begin{lemma}\label{l_dt}
Let $r\ge 2$, $a\ge 1$ and $k,\ell\ge 0$ be integers, and set $n=a+(r-1)k$.
Choosing an $[a]$-rooted $r$-forest on $[n]$ uniformly at random, the expected number of vertices
at graph distance exactly $\ell$ from $[a]$ is equal to
\[
 (a+(r-1)\ell) \frac{(r-1)^\ell (k)_\ell}{n^\ell}
\]
and is (hence) at most $a+(r-1)\ell$.
\end{lemma}
\begin{proof}
Let $N$ be the number of ordered pairs $(F,v)$ where $F$ is an $[a]$-rooted $r$-forest
on $[n]$ and $v\in [n]$ is at graph distance $\ell$ from $[a]$ in $F$.
Since there is a unique path from $v$ to $[a]$ in $F$, we can instead view
$N$ as the number of tuples $(F,v_0,e_1,\ldots,v_{\ell-1},e_\ell,v_\ell)$
where $F$ is an $[a]$-rooted $r$-forest on $[n]$,
$v_0\in [a]$, and $v_0e_1\cdots e_\ell v_\ell$ is a path in $F$.
(The bijection from such tuples to pairs $(F,v)$ maps $v_\ell$ to $v$.)

With $F$ not yet determined, there are $a$ choices for $v_0$,
then $\binom{(r-1)k}{r-1}$ choices for the remaining vertices that
with $v_0$ make up $e_1$. Then there are $r-1$ choices for $v_1$,
then $\binom{(r-1)(k-1)}{r-1}$ choices for the rest of $e_2$, and so on,
giving
\begin{align*}
 N_1 &= a (r-1)^\ell \binom{(r-1)k}{r-1}\cdots\binom{(r-1)(k-\ell+1)}{r-1} \\
 &= a(r-1)^\ell \frac{((r-1)k)!}{((r-1)(k-\ell))! (r-1)!^\ell}
\end{align*}
choices for $(v_0,e_1,\ldots,e_\ell,v_\ell)$. Now we must choose an $[a]$-rooted
$r$-forest $F$ on $[n]$ containing the edges $e_1,\ldots,e_\ell$; this is the same as choosing
an $[S]$-rooted $r$-forest $F'$ on $[n]$ where $S=[a]\cup e_1\cup\cdots\cup e_\ell$
is a set of $a+(r-1)\ell$ vertices. By Lemma~\ref{l_rtree}
we thus have
\begin{align*}
 N &= (a+(r-1)\ell) n^{k-\ell-1} \frac{((r-1)(k-\ell))!}{(k-\ell)!(r-1)!^{k-\ell}} N_1 \\
 &= (a+(r-1)\ell) n^{k-\ell-1} a (r-1)^\ell \frac{((r-1)k)!}{(k-\ell)! (r-1)!^k}.
\end{align*}
The expectation we wish to calculate
is precisely $N$ divided by the number of $[a]$-rooted $r$-forests on $[n]$.
By Lemma~\ref{l_rtree} the expectation is thus
\begin{align*}
 (a+(r-1)\ell) n^{-\ell} (r-1)^\ell \frac{k!}{(k-\ell)!}
  &= (a+(r-1)\ell) \frac{(r-1)^\ell (k)_\ell}{n^\ell} \\
  &\le (a+(r-1)\ell) ((r-1)k/n)^\ell \\
  &\le a+(r-1)\ell,
\end{align*}
as required.
\end{proof}
Note that, surprisingly, $k$ does not appear in the final upper bound in the lemma above.

\section{The smoothing distribution}\label{sec_dist}

Given positive integers $m$ and $a$, let $A_1\subset A\subset V$
with $|A_1|=a$, $|A|=2a$ and $|V|=2a+(r-1)m$,
and let $F$ be an $A$-rooted $r$-forest on $V$ chosen uniformly at random.
Let $Y_{m,a}$ be the total number
of edges of $F$ in components rooted in $A_1$. Note that $F$ has $m$ edges, so $0\le Y_{m,a}\le m$.

\begin{lemma}\label{ldist}
Let $m=m(n)$ and $a=a(n)$ satisfy $m=o(a^2)$ and $m=\Omega(a)$, let $Y_n=Y_{m,a}$,
and set $\sigma_n=m^{3/2}a^{-1}$.
Then for any integers $x_n$, $y_n$ with $x_n-y_n=o(\sigma_n)$ we have
\[
 \Pr(Y_n=x_n)-\Pr(Y_n=y_n)=o(1/\sigma_n),
\]
and $\Pr(Y_n=x_n)=O(1/\sigma_n)$.
\end{lemma}
In the terminology of Section~\ref{sec_example}, the sequence of distributions
$Y_{m(n),a(n)}$ is $\sigma_n$-smooth.
\begin{proof}
As usual, we suppress the dependence on $n$ in the notation, for example writing $\sigma$ for $\sigma_n$.

Note first that our assumptions imply that $a=O(m)=o(a^2)$, so certainly $a\to\infty$ and thus $m\to\infty$.
Note for later that $\sigma/m=m^{1/2}a^{-1}=\sqrt{m/a^2}$, so
\[
 \sigma=o(m).
\]

Let $p_k=p_{n,k}=\Pr(Y_n=k)$.
Considering first the choices for the vertices outside $A$ appearing in the subforest rooted at $A_1$, we
see that
\[
 p_k = \binom{(r-1)m}{(r-1)k}  \frac{F_{a,k}F_{a,m-k}}{F_{2a,m}},
\]
where $F_{a,k}$ denotes the number of $X$-rooted $r$-forests on $Y$ when $X\subset Y$ with $|X|=a$ and 
$|Y|=a+(r-1)k$.
From now on, let us write $t$ for $r-1$, since this will appear so often in the following calculations.
By Lemma~\ref{l_rtree}, writing $\ell$ for $m-k$, for $0\le k \le m$ we have
\begin{eqnarray}
 p_k  
 &=& \binom{tm}{tk} \frac{a (a+tk)^{k-1} (tk)!k!^{-1}t!^{-k} a (a+t\ell)^{\ell-1} (t\ell)!\ell!^{-1}t!^{-\ell}}
{2a(2a+tm)^{m-1} (tm)!m!^{-1}t!^{-m}} \nonumber \\
 &=& \frac{a}{2}\binom{m}{k}\frac{ (a+tk)^{k-1} (a+t\ell)^{\ell-1} }{(2a+tm)^{m-1}}. \label{pk}
\end{eqnarray}
We shall prove the following three statements concerning functions $k$, $k_1$ and $k_2$ of $n$
bounded between $0$ and $m(n)$, where $\sigma=\sigma(n)=m^{3/2}a^{-1}$:
\begin{equation}\label{aim2}
 \hbox{If $k_1=k_2+o(\sigma)$ and $k_1,k_2=m/2+O(\sigma)$ then $p_{k_1}\sim p_{k_2}$.}
\end{equation}
\begin{equation}\label{aim1}
 \hbox{If $k=m/2+O(\sigma)$ then $p_k=O(1/\sigma)$.}
\end{equation}
\begin{equation}\label{aim3}
\hbox{If $|k-m/2|/\sigma\to\infty$ then $p_k=o(1/\sigma)$.}
\end{equation}
(As usual, we quantifying over sequences here: the formal statement of \eqref{aim1}, for example,
is that for any sequence $k(n)$ such that $\limsup_n |k(n)-m(n)/2|/\sigma(n)<\infty$,
we have $\limsup_n p_{n,k(n)}\sigma(n)<\infty$.)

Suppose for the moment that \eqref{aim2}--\eqref{aim3} hold, and consider sequences $k_1=k_1(n)$ and $k_2=k_2(n)$
with $k_1-k_2=o(\sigma)$. The lemma asserts that then
\begin{equation}\label{finalaim}
 p_{k_1}-p_{k_2}=o(1/\sigma) \hbox{\quad and\quad} p_{k_1}=O(1/\sigma).
\end{equation}
In the special case where $k_1=m/2+O(\sigma)$, the relations \eqref{aim2} and \eqref{aim1} give \eqref{finalaim}.
In the special case where $|k_1-m/2|/\sigma\to\infty$, then also $|k_2-m/2|/\sigma\to\infty$,
so by \eqref{aim3} both $p_{k_1}$ and $p_{k_2}$ are $o(1/\sigma)$, and \eqref{finalaim} follows.
The general case now follows by a standard subsequence argument: a counterexample would have a subsequence
falling into one of these two special cases.

Our aim is now to prove \eqref{aim2}--\eqref{aim3}.
Let us first deal with the extreme values, i.e., cases where $k$ is very close to $0$
or to $m$.
We shall show that when $k\le c_0a$ for some constant $c_0$,
then $p_{k+1}\ge p_k$, so if we can show that $p_k=o(1/\sigma)$ for $k=\ceil{c_0a}$, then
the same bound for $k<\ceil{c_0a}$ follows. Here $c_0$ may depend on the sequences
$m(n)$ and $a(n)$, but not on $k(n)$.

From \eqref{pk} we see that for $0\le k<m$ we have
\begin{eqnarray*}
 \frac{p_{k+1}}{p_k} 
&=& \frac{\ell}{k+1} \ \frac{(a+t(k+1))^k}{(a+tk)^{k-1}} \ \frac{(a+t(\ell-1))^{\ell-2}}{(a+t\ell)^{\ell-1}} \\
&=& \frac{\ell}{k+1} \ \frac{a+tk}{a+t(\ell-1)} 
    \left(1+\frac{t}{a+tk}\right)^k\left(1+\frac{t}{a+t(\ell-1)}\right)^{-(\ell-1)} \\
 &=& \frac{a+tk}{k+1}\ \frac{\ell}{a+t(\ell-1)}\ \Theta(1),
\end{eqnarray*}
since $(1+x)^i=\exp(O(ix))=\Theta(1)$ when $x\ge 0$ and $|ix|\le 1$.
For $k\le m/2$, say, we have $\ell=m-k=\Theta(m)$ and
$a+t(\ell-1)=\Theta(a+m)=\Theta(m)$, so
$p_{k+1}/p_k=\Theta((a+tk)/(k+1))$. It follows that there exists a constant $c_0$ such
that for $k\le c_0 a$ we have $p_{k+1}/p_k\ge 1$, so
\begin{equation}\label{ksmall}
 \max_{k\le c_0a} p_k \le p_{\ceil{c_0a}}.
\end{equation}

Since $m=\Omega(a)$, we may choose $c_0$ small enough that $\ceil{c_0a}\le m/4$, say.
In proving \eqref{aim3}, we may assume by symmetry that $k\le m/2$.
Since $\sigma=o(m)$, we have $|\ceil{c_0a}-m/2|/\sigma\ge m/(4\sigma)\to\infty$, so 
in the light of \eqref{ksmall}, to prove
\eqref{aim3} it suffices to show that
\begin{equation}\label{aim3p}
 \hbox{If $(m/2-k)/\sigma\to\infty$ and $k\ge c_0 a$ then $p_k=o(1/\sigma)$.}
\end{equation}

From this point our aim is to prove \eqref{aim2}, \eqref{aim1} and \eqref{aim3p}.
Since all three statements only involve $k=k(n)$ such that $k,\ell=\Omega(a)$,
from now on we impose this condition. In this case, from \eqref{pk} and Stirling's formula we have
\[
 p_k\sim \frac{a}{2\sqrt{2\pi}}\ \frac{m^m}{k^k\ell^\ell} \ \sqrt{\frac{m}{k\ell}}
 \ \frac{tm+2a}{(tk+a)(t\ell+a)}\ \frac{(tk+a)^k(t\ell+a)^\ell}{(tm+2a)^m}.
\]
Roughly speaking, we shall write this expression as a polynomial factor times
an exponential factor. Then we expand the function inside the exponential around $k=m/2$
to see that $p_k$ is small when $k$ is far from $m/2$, and does not change too rapidly
when $k$ is close to $m/2$. The complication is that the polynomial factor `blows up'
as $k/m$ approaches $0$ or $1$, and it is only the condition $m=o(a^2)$ that ensures
that this `blow up' is beaten by the exponential factor.

Setting
\[
  x=k/m \hbox{\quad and\quad} \beta=a/(tm),
\]
and noting that by assumption $\beta=O(1)$, we have
\begin{eqnarray}
 p_k &\sim& \frac{a}{2\sqrt{2\pi}}\ \sqrt{\frac{1}{x(1-x)m}}\ \frac{1}{tm} \ \frac{1+2\beta}{(x+\beta)(1-x+\beta)}
 \nonumber \\
&&\hskip 3cm \left( x^{-x}(1-x)^{-(1-x)} \frac{(x+\beta)^x(1-x+\beta)^{1-x}}{1+2\beta}\right)^m\nonumber \\
 &=& \frac{a(1+2\beta)}{2\sqrt{2\pi} t m^{3/2}} f(x) \exp(-m g(x))\nonumber \\
 &=& \frac{c}{\sigma} f(x)\exp(-m g(x)),\label{cfg}
\end{eqnarray}
where $c=(1+2\beta)/(2t\sqrt{2\pi})=\Theta(1)$ is independent of $k$,
\begin{equation}\label{fdef}
 f(x) = x^{-1/2}(1-x)^{-1/2}(x+\beta)^{-1}(1-x+\beta)^{-1},
\end{equation}
and
\[
  g(x) = x\log x+(1-x)\log(1-x) -x\log(x+\beta)-(1-x)\log(1-x+\beta) +\log(1+2\beta).
\]

It is easy to see that $g(1/2)=0$. Moreover,
\[
 g'(x) = \log x-\log(1-x)-\log(x+\beta)+\log(1-x+\beta)+\frac{\beta}{x+\beta}-\frac{\beta}{1-x-\beta}
\]
is also zero at $x=1/2$, and (after a little calculation) we see that
\begin{equation}\label{g''}
 g''(x) = \beta^2 \left(\frac{1}{x(x+\beta)^2}+\frac{1}{(1-x)(1-x+\beta)^2}\right) >0.
\end{equation}
Since $\beta=O(1)$, for $1/8\le x\le 7/8$, say,
the bracket in \eqref{g''} is uniformly $\Theta(1)$,
so we have $g''(x)=\Theta(\beta^2)$. Integrating twice, we see
that for $x\in [1/8,7/8]$ we have
\begin{equation}\label{g'g}
 |g'(x)| = \Theta( \beta^2|x-1/2|) \hbox{\quad and\quad} g(x) = \Theta(\beta^2(x-1/2)^2).
\end{equation}

Recalling that $\beta=a/(tm)$ and $\sigma=m^{3/2}/a$, note that
\begin{equation}\label{bsm}
 \beta^2\sigma^2/m^2 = \frac{a^2}{t^2m^2} \frac{m^3}{a^2}\frac{1}{m^2} = \frac{1}{t^2m} =\Theta(m^{-1}).
\end{equation}
Let $k_1$ and $k_2$ satisfy $k_i=m/2+O(\sigma)$ and $k_1-k_2=o(\sigma)$, and set $x_i=k_i/m$.
Then $x_i=1/2+O(\sigma/m)$, and $x_1-x_2=o(\sigma/m)$. By the Mean Value Theorem,
there is some $\xi=1/2+O(\sigma/m)$ for which
\begin{multline*}
 |g(x_1)-g(x_2)| = |g'(\xi)||x_1-x_2| = O(\beta^2 |\xi-1/2| |x_1-x_2|) \\
 = o(\beta^2\sigma^2/m^2) = o(1/m),
\end{multline*}
from \eqref{g'g} and \eqref{bsm}.
From \eqref{fdef}, since $x_1,x_2\sim 1/2$
we have $f(x_1)\sim f(1/2)\sim f(x_2)$, and it follows from \eqref{cfg}
that $p_{k_1}\sim p_{k_2}$, proving \eqref{aim2}.
For \eqref{aim1}, simply note that $g(x)\ge 0$ always, while if $k=m/2+O(\sigma)$ then
$x=k/m$ satisfies $x=1/2+O(\sigma/m)=1/2+o(1)$, so $x$ is bounded away from $0$ and $1$
and \eqref{fdef} gives $f(x)=O(1)$. Hence \eqref{cfg} gives $p_k=O(1/\sigma)$, proving
\eqref{aim1}.

Finally, we turn to the proof of \eqref{aim3p}, considering $k$ `far' from $m/2$,
but not too close to $0$ or to $m$. First, note that since $\beta=O(a/m)$ and, by assumption,
$m=o(a^2)$, we have
\[
 \beta^2m\to\infty.
\]
Let $c_0 a\le k\le m/2$ with $(m/2-k)/\sigma\to\infty$ and set $x=k/m$,
so $x<1/2$ and $(1/2-x)/(\sigma/m)\to\infty$. If $x\ge 1/8$ then $f(x)=\Theta(1)$ while from \eqref{g'g}
we have $g(x)=\Omega(\beta^2)$ and hence $m g(x)\to\infty$.
Thus \eqref{cfg} gives $p_k=o(1/\sigma)$, as required.

Suppose instead that $x<1/8$; note that $x=k/m\ge c_0a/m=c_1\beta$, where $c_1=c_0t$ is a positive constant.
For $y\ge c_1\beta$ we have $\beta=O(y)$ and hence $y^{-1}(y+\beta)^{-2}=\Omega(y^{-3})$,
so in this range \eqref{g''} gives $g''(y)\ge c\beta^2 y^{-3}$ for some constant $c>0$. It follows easily that
there is a constant $c'$ such that for $c_1\beta\le x\le 1/8$ we have $g(x)\ge c'\beta^2 x^{-1}$.
[Indeed, for $c_1\beta\le y\le 1/4$ we have
$-g'(y)=\int_{y}^{1/2}g''(z) \dz \ge \int_{y}^{2y}g''(z) \dz =\Omega(\beta^2 y^{-2})$,
and then $g(x)=\int_{x}^{1/2}(-g'(y)) \dy \ge \int_{x}^{2x} (-g'(y)) \dy = \Omega(\beta^2 x^{-1})$.]
Hence, for $c_1\beta\le x\le 1/8$ we have
\[
 f(x)\exp(-m g(x)) = O(x^{-3/2})\exp(-m g(x)) = O(x^{-3/2} \exp(-c'\beta^2 m x^{-1})).
\]
Since $\beta^2m\to\infty$, it follows that $f(x)\exp(-m g(x))\to 0$ uniformly in this range, which
with \eqref{cfg} gives $p_k=o(1/\sigma)$, completing the proof of \eqref{aim3p} and hence of
the lemma.
\end{proof}

With a small amount of further work, the argument above extends to show that (under the given
assumptions) $Y_{m,a}$ satisfies a Gaussian local limit theorem. We shall not need this,
so we omit the details.

\section{Discrete duality}\label{sec_dual}

Recall that $\Hrnp$ denotes the random $r$-uniform hypergraph on $[n]$ in which each
of the $\binom{n}{r}$ possible edges is present independently with probability $p$.
As in the introduction, we write
$p=p(n)$ as $\lambda(n)(r-2)!n^{-r+1}$,
so $\la=1$ corresponds to the critical point of the phase transition.
More generally, for any $r$, $n$ and $p$ we call
\begin{equation}\label{bfdef}
 \lambda= pn^{r-1}/(r-2)! 
\end{equation}
the \emph{branching factor} of $\Hrnp$. For $\la>1$ recall that
$\rho_{r,\la}$, defined in \eqref{rldef},
is the survival probability of a certain branching process associated to $\Hrnp$. In
particular, when $r=2$ this process is just a Galton--Watson process with a Poisson
offspring distribution with mean $\la$; we write $\rho_\la=\rho_{2,\la}$
for its survival probability.

Given any $\la>1$, define $\las<1$, the parameter \emph{dual} to $\la$, by
\begin{equation}\label{lasdef}
 \las e^{-\las} = \la e^{-\la}.
\end{equation}
It is easy to check that $\las=\la(1-\rho_\la)$, where $\rho_\la=\rho_{2,\la}$,
and that for $\la>1$ with $\la=O(1)$ we have
\begin{equation}\label{lasrate}
 \las=1-\Theta(\la -1) \hbox{\quad and\quad} \las=\Theta(1).
\end{equation}
In other words, for any $A>0$ there exist $c, C>0$ such that
$\la\in (1,A]$ implies $(1-\las)/(\la-1)\in [c,C]$ and $\las\in [c,1)$ (recall that $\las<1$ by definition).
The second, crude bound in \eqref{lasrate} is only relevant when $\la$ is large.

In the regime we are interested in, we have $\lambda=1+\eps$ with $\eps=\eps(n)$
bounded and $\eps^3n\to\infty$, so by the results of~\cite{KL_giant,BR_hyp},
$\Hrnp$ is supercritical.
Defining $\delta=\delta(n)\ge (\eps^3n)^{-1/3}$ and $\range = \range_n= [(1-\delta)\rho_{r,\la}n, (1+\delta)\rho_{r,\la}n]$
as in \eqref{rangedef}, by \eqref{rangewvhp} we have
\begin{equation}\label{rangewhp}
 \Pr(L_1\in\range) = 1-O(1/(\eps^3n)) = 1-o(1).
\end{equation}
We shall only consider possible values of $L_1$ lying in $\range$. We start
with a simple calculation, showing that if $s\in\range$ then $\Hrnsp$ is subcritical
(but not too strongly so).

\begin{lemma}\label{dual1}
Under our Weak Assumption~\ref{A1}, for any $s=s(n)\in\range$,
the branching factor $\la'=\la(\Hrnsp)$ of the random hypergraph $\Hrnsp$
satisfies $\la'=1-\Theta(\eps)$ and $\la'=\Theta(1)$.
\end{lemma}
\begin{proof}
Let $\mu_n=\rho_{r,\la}n$. Ignoring the fact that $\mu_n$ need not be an integer,
if we define the branching factor $\la(\Hrnsop)$ by \eqref{bfdef}, with
$n-\mu_n$ in place of $n$, 
then
\[
 \la(\Hrnsop) = (1-\mu_n/n)^{r-1}\la = (1-\rho_{r,\la})^{r-1}\la = (1-\rho_\la)\la = \las,
\]
which is $1-\Theta(\eps)$ by \eqref{lasrate}. 
For $s\in\range$ we have $(n-s)/(n-\mu_n)=1+O(\delta\eps)=1+o(\eps)$, so, since $r$ is constant,
$\la'=\la(\Hrnsp)=(1+o(\eps))^{r-1}\la(\Hrnsop)=1-\Theta(\eps)$ also.
To see that $\la'=\Theta(1)$ (i.e., is bounded away from zero), recall from \eqref{rangecont}
that $s\in\range$ implies $s\le (1-c)n$ for some constant $c>0$. Then
$\la' = (1-s/n)^{r-1}\la \ge c^{r-1}\la \ge c^{r-1}$.
\end{proof}
Note that here we do not really need $\delta$ to tend to zero:
it would suffice to assume that $\delta$
is at most some small constant depending on the upper bound on $\eps$.

A simple consequence of the fact that $\Hrnsp$ is subcritical is that
it is unlikely to contain a component with $s$ or more vertices. We state
a convenient form of this result rather than the strongest version possible.

\begin{lemma}\label{lccrit}
Under our Weak Assumption~\ref{A1}, for any $s=s(n)\in\range$, whp $L_1(\Hrnsp)<n^{2/3}<s$.
\end{lemma}
\begin{proof}
From either Karo\'nski and \L uczak~\cite[Theorem 6]{KL_giant}
or~\cite[Theorem 2]{BR_hyp2} (which gives a better probability bound but a worse constant
$c_r$), there is a constant $c_r>0$ such that if $\Hrmp$ has branching 
factor $1-\eta$ where $\eta^3m\to\infty$, then whp
\[
 L_1(\Hrmp)\le c_r\eta^{-2}\log(\eta^3m) = o(m^{2/3}).
\]
For $s\in\range$, by Lemma~\ref{dual1} the branching factor of $\Hrnsp$ is $1-\eta$
with $\eta=\Theta(\eps)$. Since $m=n-s=\Theta(n)$ and $\eps^3n\to\infty$, we have $\eta^3m\to\infty$,
so whp $L_1(\Hrmp)<m^{2/3}<n^{2/3}$.
The result follows since $s=\Theta(\eps n)$, so $s/n^{2/3}\to\infty$ and in particular $s>n^{2/3}$
if $n$ is large enough.
\end{proof}

Let $\cL_1$ be the component of $\Hrnp$ with the most vertices, if there is a unique
such component. In the case of ties we order (the vertex sets of) possible components
arbitrarily (e.g., by the lowest numbered vertex present), and use this order
to break the tie. Of course $|\cL_1|=L_1$.
The following explicit version of the discrete duality principle
says that we may treat the graph outside $\cL_1$ as a subcritical instance of
the same hypergraph model. We write $\cH_s$ for the set of all labelled
$r$-uniform hypergraphs with exactly $s$ vertices. We always assume
implicitly that any conditional probability is defined: i.e., if the event
being conditioned on has probability $0$, there is nothing to prove.

\begin{lemma}\label{dual2}
Suppose that our Weak Assumption~\ref{A1} holds,
and define $\range=\range_n$ as in \eqref{rangedef}.
Let $\cQ$ be any isomorphism invariant property of hypergraphs,
and $f$ any isomorphism invariant function from hypergraphs to the non-negative reals.
Then, for any $s=s(n)\in\range$ and any $\cP=\cP(n) \subset \cH_s$, we have
\[
 \Pr\bb{\Hrnp \setminus \cL_1 \text{ has }\cQ \mid \cL_1\in\cP} \le (1+o(1)) \Pr( \Hrnsp \text{ has }\cQ)
\]
and
\[
 \E\bb{f(\Hrnp \setminus \cL_1) \mid \cL_1\in\cP} \le (1+o(1)) \E( f(\Hrnsp) ),
\]
as $n\to\infty$,
with the error terms uniform over all $s\in \range$ and $\cP\subset\cH_s$.
\end{lemma}
The most natural case here is $\cP=\cH_s$, in which case we are simply conditioning
on the event $L_1=s$. Often we shall take $\cP$ to be the set of hypergraphs with $s$
vertices and nullity $t$; then we are conditioning on the event $\{L_1=s, N_1=t\}$.
 
\begin{proof}
Although
we have emphasized the uniformity of the error terms for clarity, this uniformity
is automatic, considering the worst-case choice of $s=s(n)$ and $\cP=\cP(n)$.

Without loss of generality $\cP$ consists of a single hypergraph $H_s$
with vertex set $S\subset [n]$ with $|S|=s$.
From the definitions of $\Hrnp$ and of $\cL_1$, the conditional distribution
of $\Hrnp\setminus \cL_1$ given that $\cL_1=H_s$ is that
of the random hypergraph $H'=\Hrnsp$ 
on the vertex set $[n]\setminus S$ conditioned on the event $\cE$ that

\smallskip
(i) $H'$ contains no component with more than $s$ vertices, and

(ii) $H'$ has no $s$-vertex component that beats $H_s$ in the 
tie-break order used in defining $\cL_1$.

\smallskip\noindent
By Lemma~\ref{lccrit}, $\Pr(\cE)=1-o(1)$. Hence,
\begin{multline*}
  \Pr\bb{\Hrnp \setminus \cL_1 \text{ has }\cQ \mid \cL_1=H_s }
  = \Pr\bb{H' \hbox{ has }\cQ \mid \cE} \\
\le \frac{ \Pr(H' \hbox{ has }\cQ)}{\Pr(\cE)}
 = (1+o(1)) \Pr(H' \hbox{ has }\cQ),
\end{multline*}
proving the first statement. For the second, argue similarly, or express
$\E(f(H))$ as $\int_0^\infty \Pr(f(H)\ge t)\mathrm{d}t$
and apply the first statement.
\end{proof}

A variant of the argument above gives the following result, which
may be seen as an extension of an observation of Karo\'nski and {\L}uczak~\cite[p.\ 133]{KL_giant}.
By a \emph{property} of hypergraphs we simply mean a set of hypergraphs; we do not assume that this set is
closed under isomorphism. As usual, let $\cL_1$ be a component of $\Hrnp$ with the maximal
number of vertices, chosen according to any fixed rule if there is a tie.

\begin{lemma}\label{lbigep}
Let $\cQ_s$ be any property of $s$-vertex hypergraphs,
and let $N_s$ be the expected number of components of $\Hrnp$ having property $\cQ_s$.
Let $\ubig$ be the event that $\Hrnp$ has at most one component with more than $n^{2/3}$ vertices,
and set $\cA_s=\{N_s>0\}\cap \ubig$ and $\cB_s=\{N_s>0\}\cap \ubig^\cc$.
Under our Weak Assumption~\ref{A1} we have
\begin{equation}\label{ep}
 \Pr(\cL_1\in\cQ_s)\sim \Pr(\cA_s) \sim \E[N_s]
\end{equation}
and
\begin{equation}\label{ep2}
 \Pr(\cB_s) = o(\Pr(\cL_1\in\cQ_s)),
\end{equation}
uniformly over all $s\in \range$ and all properties $\cQ_s$, where $\range$ is defined in \eqref{rangedef}.
\end{lemma}
Note that $\ubig$ holds whp by (for example) the second statement of Theorem~\ref{thsupertail}.
\begin{proof}
Clearly
\begin{equation}\label{NS1}
 \E[N_s]\ge \Pr(N_s>0) \ge \Pr(\cL_1\in\cQ_s)  \ge \Pr(\cA_s).
\end{equation}
Let $N^+\ge N_s$ denote the number of components of $\Hrnp$ with more than $n^{2/3}$ vertices.
If $\cA_s$ holds, then $N_s=1$. If $\cA_s$ does not hold and
$N_s>0$, then $N^+\ge 2$. Hence
\[
 N_s\le \ind{\cA_s} + N_s \ind{N^+\ge 2}
\]
and, taking expectations,
\begin{equation}\label{NS2}
 \E[N_s] \le \Pr(\cA_s) + \E[N_s \ind{N^+\ge 2}].
\end{equation}
For $S\subset [n]$ with $|S|=s$, let $\cQ_S$ be the event that $S$ is the vertex
set of a component of $\Hrnp$ having property $\cQ_s$. Then
\begin{multline*}
 \E[N_s \ind{N^+\ge 2}] = \E \sum_{S\::\:|S|=s} \ind{\cQ_S} \ind{N^+\ge 2}
= \sum_S \Pr(\cQ_S \cap \{N^+\ge 2\})\\
 = \sum_S \Pr(\cQ_S) \Pr(N^+\ge 2\mid \cQ_S)
 = \sum_S \Pr(\cQ_S) \Pr(L_1(\Hrnsp)>n^{2/3})\\
 = \E[N_s] \Pr(L_1(\Hrnsp)>n^{2/3}) = o(\E[N_s]),
\end{multline*}
by Lemma~\ref{lccrit}.

From \eqref{NS2} we now obtain $\Pr(\cA_s)\ge (1-o(1))\E[N_s]$, which 
combined with \eqref{NS1} completes the proof of \eqref{ep}.
The final statement \eqref{ep2} follows since $\Pr(\cB_s)=\Pr(N_s>0)-\Pr(\cA_s) \le \E[N_s]-\Pr(\cA_s)$.
\end{proof}

\section{Trees, paths and cycles outside the giant component}\label{sec_sub}

Throughout this section we assume our Weak Assumption~\ref{A1}.
In other words we fix an integer $r\ge 2$ and a function
$p=p(n)=(1+\eps)(r-2)!n^{-r+1}$ where $\eps=\eps(n)=O(1)$ and $\eps^3n\to\infty$.
We write $\la$ for $1+\eps$, which is the branching factor of $\Hrnp$ as defined in \eqref{bfdef}.

Our next lemma concerns trees outside the giant component.
As in Section~\ref{sec_dual} we consider the hypergraph $H'=\Hrmp$ where $m=n-s$ with $s\in\range$,
where $\range=\range_n$ is defined as in \eqref{rangedef}.

\begin{lemma}\label{smalltree}
Let $T_k$ denote the number of tree components of $H'=\Hrnsp$ with $k$ edges,
and $T^{(2)}_{k,\ell}$ the number of ordered pairs $(T,T')$
of distinct tree components of $H'$ with $e(T)=k$ and $e(T')=\ell$.
Then
\begin{equation}\label{mk}
 \mu_k = \E[T_k] = \Theta(n(k+1)^{-5/2})
\end{equation}
and
\begin{equation}\label{mkl}
 \mu_{k,\ell} = \E[T^{(2)}_{k,\ell}] = \mu_k\mu_\ell\bb{1+O(\eps(k+\ell)^2m^{-1})} \sim\mu_k\mu_\ell,
\end{equation}
uniformly in $0\le k$, $\ell\le 10/\eps^2$ and $s\in\range$.
\end{lemma}
\begin{proof}
It suffices to fix sequences $k=k(n)$, $\ell=\ell(n)$ and $s=s(n)$ satisfying
$0\le k,\ell\le 10/\eps^2$ and $s\in\range_n$, and prove \eqref{mk} and \eqref{mkl}
for these sequences, where in principle the implicit constants above and in the proof
that follows may depend on the choice of the sequences.
The claimed uniform bounds follow by considering appropriate worst-case sequences.

Suppressing the dependence on $n$ as usual, fix sequences $k$, $\ell$ and $s$ as above, and let $m=n-s$.
Note that $m=\Theta(n)$; see \eqref{rangecont}.
We shall apply Lemma~\ref{l_rtree} with $a=1$; recall the notation
$\partit{k}{t}=(kt)!/(k! t!^k)$ used there.

Considering first the number of choices for the $k(r-1)+1$ vertices,
then the number of trees $T$ on the given vertex set,
and finally the probability that the edges of $T$ are present but no other edges incident with $T$ are, we have
\begin{equation}\label{muk}
 \mu_k = \binom{m}{k(r-1)+1} (k(r-1)+1)^{k-1} \partit{k}{r-1}
p^k (1-p)^{t_{m,k}-k},
\end{equation}
where
\[
 t_{m,k} = \binom{m}{r} - \binom{m-k(r-1)-1}{r}
\]
is the number of potential hyperedges on an $m$-vertex set meeting a given $(k(r-1)+1)$-vertex set at least once.
Postponing the evaluation of $\mu_k$ for the moment, if we write a similar formula for
$\mu_{k,\ell}$, then most terms agree with the corresponding terms in $\mu_k\mu_\ell$. Indeed,
writing $a$ for $k(r-1)+1$ and $b$ for $\ell(r-1)+1$, it is easy to see that
\begin{equation}\label{mklc}
 \frac{\mu_{k,\ell}}{\mu_k\mu_\ell} = 
 \binom{m-a}{b} \binom{m}{b}^{-1} (1-p)^{-t_{m,k,\ell}},
\end{equation}
where $t_{m,k,\ell}$ is the number of potential hyperedges meeting both a given set of $a$
vertices and a given disjoint set of $b$ vertices.
Note that
\[
 t_{m,k,\ell}=ab\binom{m}{r-2}+O((a+b)^3m^{r-3}) = ab\frac{m^{r-2}}{(r-2)!} + O((a+b)^3m^{r-3}).
\] 
Writing
\[
 \la' = p m^{r-1}/(r-2)!
\]
for the branching factor of $H'=\Hrmp$ (see \eqref{bfdef}), since $p=O(n^{-r+1})=O(m^{-r+1})$
it follows that
\[
 pt_{m,k,\ell} = \la' ab/m +O((a+b)^3m^{-2}).
\]
Since, crudely, $p=O(1/m)$ and $ab=O((a+b)^3)$,
from this it certainly follows that $p^2t_{m,k,\ell}=O((a+b)^3m^{-2})$, so
\begin{equation}\label{l1p}
 \log\left((1-p)^{-t_{m,k,\ell}}\right) = pt_{m,k,\ell}+O(p^2t_{m,k,\ell})
 = \la' ab/m +O((a+b)^3m^{-2}).
\end{equation}

By Lemma~\ref{dual1} we have
\begin{equation}\label{lapbds}
 \la'=1-\Theta(\eps) \hbox{\quad and\quad} \la'=\Theta(1).
\end{equation}
Using the formula $\binom{m-a}{b}/\binom{m}{b}=\exp(-ab/m+O((a+b)^3/m^2))$,
valid for $a,b\le m/3$, say, from \eqref{mklc}--\eqref{lapbds}
we see that
\[
 \log\left(\frac{\mu_{k,\ell}}{\mu_k\mu_\ell}\right)
 = \frac{(\la'-1)ab}{m} + O((a+b)^3m^{-2}) = O(\eps(a+b)^2m^{-1}) =o(1),
\]
since $a+b=O(\eps^{-2})=o(\eps m)$. This proves \eqref{mkl}.

Let us temporarily adopt the convention of writing $f\approx g$ for $f=\Theta(g)$.
Returning to $\mu_k$, for $k=0$ we have $\mu_k=m(1-p)^{t_{m,0}}$.
Since $m\approx n$, $t_{m,0}=\binom{m-1}{r-1}\approx n^{r-1}$ and $p \approx n^{-r+1}$,
it follows that $\mu_0 \approx n$, as required.
From now on suppose that $1\le k\le 10/\eps^2$.
Since $p^2t_{m,k}=O(p^2km^{r-1})=O(pk)$ and $pk=o(1)$, from \eqref{muk} we have
\begin{eqnarray*}
 \mu_k &\sim& m\binom{m-1}{k(r-1)}(k(r-1)+1)^{k-2} \frac{(k(r-1))!}{k!(r-1)!^k}p^k \exp(-pt_{m,k}) \\
 &\approx& m (m-1)_{k(r-1)} \frac{k^{k-2}(r-1)^{k-2}}{k!(r-1)!^k}p^k \exp(-pt_{m,k})\\
 &\approx& mk^{-2} (m-1)_{k(r-1)} \frac{k^k}{k!(r-2)!^k}p^k \exp(-pt_{m,k}),
\end{eqnarray*}
where, as before, $(x)_y$ denotes the falling factorial $x(x-1)\cdots(x-y+1)$.
For $y\le x/2$,
\[
 (x-1)_y=x^y\exp\left(-\frac{y^2}{2x}+O(y/x)+O(y^3/x^2)\right).
\]
Since $m\approx n$, $\eps^3n\to\infty$ and $k\le 10/\eps^2$,
both $k/m$ and $k^3/m^2$ are $o(1)$. Hence
\begin{eqnarray}
 \mu_k &\approx&  mk^{-2} \left(\frac{m^{r-1}p}{(r-2)!}\right)^k \frac{k^k}{k!}
 \exp\left(-pt_{m,k}-\frac{(r-1)^2k^2}{2m}\right) \nonumber \\
 &\approx& mk^{-5/2} (\la')^k  \exp\left(k-pt_{m,k}-\frac{(r-1)^2k^2}{2m}\right), \label{mklast}
\end{eqnarray}
since $k^k/k! \approx e^k/\sqrt{k}$.

Now
\begin{eqnarray*}
 t_{m,k} &=& \frac{m^r-(m-k(r-1))^r}{r!} +O(m^{r-1}) \\
 &=& \frac{rk(r-1)m^{r-1}-\binom{r}{2}k^2(r-1)^2m^{r-2}}{r!} + O(m^{r-1}+k^3m^{r-3})\\
 &=& \frac{km^{r-1}}{(r-2)!} - \frac{k^2(r-1)^2m^{r-2}}{2(r-2)!} +O(m^{r-1}).
\end{eqnarray*}
Since $p=\la'(r-2)!/m^{r-1}$, it follows that
\[
 pt_{m,k} = \la' k -\la' \frac{(r-1)^2k^2}{2m}+O(1).
\]
Thus, 
recalling that $1-\la'=O(\eps)$, that $k=O(\eps^{-2})$, and that $\eps^3m\to\infty$,
the term inside the exponential in \eqref{mklast} is
\[
 k(1-\la') -(1-\la')\frac{(r-1)^2k^2}{2m} +O(1) = k(1-\la')+O(1).
\]
Hence, from \eqref{mklast},
\[
 \mu_k \approx mk^{-5/2} (\la' e^{1-\la'})^k.
\]
From the second bound in \eqref{lapbds}, $\la'$ is bounded away from $0$.
Since $(1-x)e^x=\exp(O(x^2))$ when $0<x<1$ is bounded away from 1,
it follows that $(\la' e^{1-\la'})^k = \exp(O((1-\la')^2k)) = \exp(O(\eps^2k)) \approx 1$,
completing the proof of \eqref{mk}.
\end{proof}

\begin{corollary}\label{csmallt}
Suppose that our Weak Assumption~\ref{A1} holds, and define $\range=\range_n$ as in \eqref{rangedef}.
There is a constant $c>0$ such that, for any $s=s(n)\in\range$ and $t=t(n)\ge 0$,
\[
 \Pr\bb{ T_{\ceil{\eps^{-2}},2\ceil{\eps^{-2}}}( \Hrnp\setminus\cL_1) \le c \eps^3 n \mid L_1=s, N_1=t } = o(1),
\]
where $T_{k,k'}(H)$ denotes the number components of a hypergraph $H$
that are trees with between $k$ and $k'$ edges (inclusive).
\end{corollary}
\begin{proof}
We must be a little careful with the uniformity in this proof: the choice of $c$ is not allowed
to depend on $s=s(n)$ and $t=t(n)$.

Let $H'=\Hrnsp$ as before and, ignoring the rounding to integers,
let $T=T_{n,s}=T_{\eps^{-2},2\eps^{-2}}(H')$.
Defining $\mu_k$ and $\mu_{k,\ell}$ as in Lemma~\ref{smalltree}, by that lemma we have
\begin{equation}\label{Et}
 \E[T] = \sum_{k=\eps^{-2}}^{2\eps^{-2}} \mu_k = \Theta(\eps^{-2}n(\eps^{-2})^{-5/2}) = \Theta(\eps^3n),
\end{equation}
and
\[
 \E[T(T-1)] = \sum_{k=\eps^{-2}}^{2\eps^{-2}} \sum_{\ell=\eps^{-2}}^{2\eps^{-2}} \mu_{k,\ell} \sim
 \sum_{k,\ell}\mu_k\mu_\ell = \E[T]^2
\]
uniformly in the choice of $s=s(n)\in\range_n$.
Let $a>0$ be the implicit constant in the lower bound in \eqref{Et}, which does not depend
on $s$. Since $\E[T]\ge a\eps^3n\to\infty$,
we have $\E[T^2]=\E[T(T-1)]+\E[T]\sim \E[T]^2$. Hence, by Chebyshev's inequality,
$\Pr(T \ge a\eps^3n/2)=1-o(1)$ as $n\to\infty$, uniformly in $s=s(n)$.

The result follows by Lemma~\ref{dual2}, applied with $\cP$ the set of all $s$-vertex hypergraphs with nullity $t$.
\end{proof}

We shall need some further, simpler results about the part of $\Hrnp$ lying outside the giant component.
The first concerns (essentially) the sum of the squares of the component sizes; it 
is perhaps in the literature, but since it is immediate, we give a proof for completeness.
Given a hypergraph $H$,
let $\Nc(H)$ denote the number of (ordered) pairs $(v,w)$ of (not necessarily distinct) vertices
of $H$ with the property that $v$ and $w$ are connected by a path, i.e., are in the same component.

\begin{lemma}\label{Ncon}
Suppose that our Weak Assumption~\ref{A1} holds, and define $\range=\range_n$ as in \eqref{rangedef}.
Let $s=s(n)\in\range$ and $t=t(n)\ge 0$. Then
$\E[\Nc(\Hrnsp)]=O(n/\eps)$ and
\[ 
 \E\bigl[ \Nc(\Hrnp\setminus\cL_1) \mid L_1=s, N_1=t \bigr] = O(n/\eps).
\]
\end{lemma}
\begin{proof}
By Lemma~\ref{dual2} it suffices to prove the first statement.
Let $H'=\Hrnsp$ and, for $\ell\ge 0$, let
$J_\ell$ be the number of ordered pairs of vertices $v,w\in H'$
joined by a path in $H'$ of length $\ell$, so $\Nc(H')\le\sum_\ell J_\ell$. Set $m=n-s$.
Writing a $v$--$w$ path of length $\ell$ as $v_0 e_1 v_1 e_2\cdots e_\ell v_\ell$, where
the $v_i$ are distinct vertices and the $e_i$ distinct hyperedges with 
$v_0=v$, $v_\ell=w$ and $e_i$ containing $v_{i-1}$ and $v_i$, there are at most
$m^{\ell+1}$ choices for the $v_i$, then at most
$\binom{m}{r-2}$ ways of extending each pair $v_{i-1}v_i$ to a hyperedge;
to obtain a path these edges must be distinct, so the probability that all are present
is $p^{\ell}$. Hence,
\[
 \E[J_\ell] \le  m^{\ell+1} \frac{m^{(r-2)\ell}}{(r-2)!^\ell} p^\ell
 = m\left(\frac{m^{r-1}p}{(r-2)!}\right)^\ell
 = m \la(H')^\ell,
\]
where $\la(H')$ is the `branching factor' of $H'=\Hrnsp$, defined by \eqref{bfdef}.
By Lemma~\ref{dual1}, $\la(H')=1-\Theta(\eps)$, so summing over $\ell$ we see that
\[
 \E[\Nc(H')] \le m (1-\la(H'))^{-1} = O(n/\eps),
\]
as claimed.
\end{proof}

By similar arguments, one can show that the expected number of vertices  on cycles
is $O(\eps^{-1})$, and that the expected number of vertices in components containing
cycles is $O(\eps^{-2})$. We do not need these bounds here.

We finish this section by considering \emph{complex} components, i.e., ones with nullity
at least $2$. Karo\'nski and \L uczak~\cite{KL_giant} prove a version of the following
lemma
for the `size model' $\Hrnm$. We give a (more detailed) proof for $\Hrnp$ for completeness.

\begin{lemma}\label{lcx}
Suppose that our Weak Assumption~\ref{A1} holds, and define $\range=\range_n$ as in \eqref{rangedef}.
For any $s=s(n)\in\range$, the expected number of complex components of $H'=\Hrnsp$
is $O(1/(\eps^3n))=o(1)$.
\end{lemma}
\begin{proof}
Writing $\GI(H)$ for the bipartite vertex-edge incidence graph of a hypergraph $H$, it is
easy to check that $n(H)=n(\GI(H))$. A minimal connected graph with nullity at least $2$ clearly
has nullity exactly $2$ (otherwise delete an edge in a cycle), and is easily seen to be either
a $\theta$-graph, consisting of two distinct vertices joined by three internally
vertex-disjoint paths, or a dumbbell, i.e., two edge-disjoint cycles connected
by a path of length at least 0. (The cycles are vertex-disjoint
unless the connecting path has length $0$.) Up to isomorphism, there are $O(\ell^2)$ 
such graphs with $\ell$ edges: having chosen whether the graph is of the $\theta$ or
dumbbell type, it is specified by choosing the lengths of three paths/cycles,
constrained to sum to $\ell$.

Let $\cG_\ell$ denote the set of isomorphism classes of $\ell$-edge bipartite graphs
of the form above, where we distinguish the vertex class $A$ corresponding
to hypergraph vertices from the class $B$ corresponding to hypergraph edges;
thus $|\cG_\ell|=O(\ell^2)$. If $H$ is a connected hypergraph with $n(H)\ge 2$, then $n(\GI(H))\ge 2$,
so $\GI(H)$ contains some $G\in \bigcup_\ell \cG_\ell$ as a subgraph. If
$G$ has vertex partition $A\cup B$, with $A=\{a_1,\ldots,a_k\}$ and $B=\{b_1,\ldots,b_t\}$,
then in particular $H$ has a subgraph $H_0$ consisting of $t$ hyperedges
with $\GI(H_0)$ containing $G$ as a subgraph. Fixing $G$ for the moment, let
us estimate the expected number of such subgraphs $H_0$ present in $H'=\Hrnsp$.

Writing $m=n-s$, there are $m(m-1)\ldots (m-k+1)\le m^k$ choices for the vertices
of $H'$ corresponding to $a_1,\ldots,a_k$. Let $d_i$ be the degree of $b_i$ in $G$.
For each $1\le i\le t$ we must choose
$r-d_i$ further vertices (other than those already specified by the neighbours
of $b_i$ in $G$) to complete the hyperedge corresponding to $b_i$. For each $i$ there
are at most $m^{r-d_i}/(r-d_i)!$ ways of doing this. Since all but at most two
vertices of $G$ have degree~$2$, and $\sum_i d_i=e(G)=\ell$, this gives in total
\[
 O\left(\frac{m^{rt-\ell}}{(r-2)!^t}\right)
\]
choices.
Finally, the probability that the resulting subgraph $H_0$ is present in $H'$
is exactly $p^t$. Hence, the expected number of such subgraphs $H_0$ corresponding
to a particular $G$ is bounded by a constant times
\[
 \frac{m^{k+rt-\ell}}{(r-2)!^t} p^t = m^{k+t-\ell} \left(\frac{m^{r-1}p}{(r-2)!}\right)^t
 = m^{-1} \la(H')^t,
\]
where in the last step we used the fact that $G$ has nullity $2$, so $k+t-\ell=|G|-e(G)=-1$,
and the definition of the branching factor $\la(H')$. 

Since $G$ has either two vertices of degree $3$ or one of degree $4$, and all other
vertices have degree $2$, we have $2t\le \ell=\sum_i d_i\le 2t+2$. Hence $t\ge \ell/2-1$.
Thus, summing over the $O(\ell^2)$ choices of $G\in G_\ell$ and then over $\ell$ we see that
the expectation $\mu$ of number of complex components of $H'$ satisfies
\[
 \mu = O\left(\sum_{\ell\ge 2} m^{-1}\ell^2 \la(H')^{\ell/2-1}\right)
 = O\left(\sum_{\ell\ge 2} m^{-1}\ell^2 \la(H')^{\ell/2}\right),
\]
using the bound $\la(H')=\Theta(1)$ from Lemma~\ref{dual1} in the last step.
Now $\la(H')=1-\Theta(\eps)$ by Lemma~\ref{dual1}; hence $\la(H')^{1/2}=1-\Theta(\eps)$.
Since
\[
 \sum_{\ell\ge 2} \ell^2 x^\ell\le 2\sum_{\ell\ge 0} (\ell+1)(\ell+2)x^\ell/2 = 2(1-x)^{-3}
\]
for $0\le x<1$, it follows that
$
 \mu = O(m^{-1}\eps^{-3}) = O(1/(\eps^3n)),
$
as claimed.
\end{proof}

Of course, instead of considering vertex-edge incidence graphs, we could directly count
the expected number of minimal complex hypergraphs present in $\Hrnsp$. However, there are significantly
more classes of minimal complex hypergraphs than minimal complex graphs, because
the special (degree more than $2$) vertices of the corresponding bipartite incidence graph
may correspond to vertices or edges of the hypergraph.

\begin{lemma}\label{C1unique}
Suppose that our Weak Assumption~\ref{A1} holds.
Let $\ucx$ be the event that $\cL_1$ is the unique complex component of $\Hrnp$.
Then for any $s=s(n)\in \range$ and $t=t(n)\ge 2$ we have
\[
 \Pr\bb{ \ucx^\cc \mid L_1=s, N_1=t } = O(1/(\eps^3n)).
\]
Furthermore, the probability that $\Hrnp\setminus\cL_1$ has a complex component is $O(1/(\eps^3 n))$.
\end{lemma}

\begin{proof}
Let $\cE$ be the event that $\Hrnp\setminus\cL_1$ has at least one complex component.
By Lemmas~\ref{dual2} and~\ref{lcx}, for $s\in\range$ and $t'\ge 0$ we have
\begin{equation}\label{ccx}
 \Pr\bb{ \cE \mid L_1=s, N_1=t' }=O(1/(\eps^3n)).
\end{equation}
Since $N_1=t\ge 2$ implies that $\cL_1$ is complex,
the first statement follows.

Since \eqref{ccx} holds for all $t'$, for any $s\in\range$ we have
\[
 \Pr(\cE\mid L_1=s)=O(1/(\eps^3n)).
\]
Recalling from \eqref{rangewvhp} that
$\Pr(L_1\notin\range)=O(1/(\eps^3 n))$,
it follows that $\Pr(\cE)=O(1/(\eps^3n))$.
\end{proof}

\section{Extended cores in hypergraphs}\label{sec_ext}

The strategy of our proof of Theorem~\ref{thprob} is as follows.
We shall randomly mark a small (order $\eps^2$) fraction of the vertices of $H=\Hrnp$,
and define the \emph{extended core} $C^+(H)$ by repeatedly deleting edges in which at least $r-1$ vertices
are unmarked and are contained in no other edges.
We shall show that, conditional on the event $\{L_1=s, N_1=t\}$, where
$s$ and $t$ are in the typical range,
certain events are likely to hold. In particular, it is likely that the largest
component $\eco$ of $C^+(H)$ is a subgraph of the largest component of $H$, that the number
$a_1$ of vertices in $\eco$ is $\Theta(\eps^2 n)$, and that the number $a_0$ of isolated
vertices in $C^+(H)\setminus \eco$ is also $\Theta(\eps^2 n)$. We condition on $C^+(H)$, and
pick $a=\min\{a_0,a_1\}$ vertices of $\eco$ and $a$ isolated vertices of $C^+(H)\setminus \eco$.
We also condition on the set $V$ of vertices joined by paths in $H$ to the selected vertices,
which we show satisfies $|V|=\Theta(\eps n)$ with high probability. Then we show that the conditional
distribution of the number of vertices in $V$ that are joined by paths to $\eco$ has a smooth distribution;
it is this number that will play the role that $Y_n$ plays in the proof of Theorem~\ref{th_eg}.

Turning to the details,
by the \emph{core} $\core(H)$ of a hypergraph $H$ we mean the (possibly empty)
hypergraph formed from $H$ by repeatedly deleting isolated vertices and
hyperedges $e$ in which at most one vertex is in a hyperedge other than
$e$. Equivalently, $\core(H)$ is the maximal sub-hypergraph of $H$ without
isolated vertices in which every
edge contains at least two vertices in other hyperedges.  Note that this is only
one of several possible generalizations of the concept of the core
of a graph~\cite{BBkcore}; another natural possibility is to take the maximal
sub-hypergraph with minimum degree at least $2$. A hypergraph $H$ consists of its
core, tree components, and the `mantle', made up of
trees each of which meets the core in a single vertex. It is a part of the mantle that
we shall use in our smoothing argument.

Note that the core of $H$ and that of its bipartite vertex--edge incidence
graph $\GI(H)$ correspond in a natural way, except that in the latter,
any vertices corresponding to vertices of $\core(H)$ that are in a single
edge of $\core(H)$ are deleted.

As discussed in Section~\ref{sec_example}, we would like to `detach and reattach' the trees attached
not only to the core, but also to an additional set of vertices of comparable size. To achieve
this, we define an `extended core', essentially by artificially placing a suitable number of extra
vertices into the core; we shall call these vertices `marked' vertices.

Let $(H,V^*)$ be a \emph{marked hypergraph}: a hypergraph $H=(V,E)$ together
with a subset $V^*$ of $V$. The vertices in $V^*$ will be called \emph{marked vertices}.
The \emph{extended core} $\excore(H,V^*)$ is the marked sub-hypergraph obtained by repeatedly deleting
unmarked isolated vertices, and hyperedges in which all or all but one vertices
are unmarked and have degree $1$. Equivalently, $\excore(H,V^*)$ is the maximal
sub-hypergraph in which every edge contains at least two vertices that are either
marked or in at least one other edge, and all isolated vertices are marked.
Note that the deletion operation defining the extended core preserves connectivity,
so the extended core of a connected hypergraph $H$ is either connected or, if $H$ is a tree
with no marked vertices (an `unmarked tree'), empty. Of course, $\excore(H,V^*)$ is the union of the extended
cores of the components of $H$.

\begin{proposition}\label{excore}
Any marked hypergraph $(H,V^*)$ is the union of its extended core $\excore=\excore(H,V^*)$,
a set $\{T_v\}_{v\in V(\excore)}$ of trees, each with with $v\in T_v$,
and a possibly empty set $\{U_i\}$ of trees, with the vertex sets
$V(T_v)\setminus v$ and $V(U_i)$ disjoint from each other and from $V(\excore)$.
\end{proposition}
In other words, noting that by definition all vertices outside $\excore(H,V^*)$ are unmarked,
we may reconstruct $(H,V^*)$ from its extended core by adding disjoint trees to each vertex $v$
of the extended core, unmarked expect
possibly at $v$, and possibly some further disjoint unmarked trees.
Later we shall refer to the set $M^+=\bigcup_{v\in \excore(H,V^*)} (V(T_v)\setminus v)$ as the (vertex set of)
the \emph{extended mantle} of $(H,V^*)$.
\begin{proof}
Simply reverse the edge-deletion algorithm defining the extended core.
\end{proof}

In this section and the next it will be convenient (though not essential) to assume
that $\eps\to 0$, i.e., our Standard Assumption~\ref{A0},
as in Theorem~\ref{thL1} whose proof we are preparing for. We also consider a constant
$0<\eta<1/100$ whose role will be explained at the start of the next section. Any implicit
constants or functions may depend on the choice of the functions $\eps=\eps(n)$ and the constant
$\eta>0$. As we shall see in Section~\ref{sec_main}, this will cause no problems when
we come to apply the results. Thus, in this section, we may regard $\eps=\eps(n)$ and $\eta>0$
as given, satisfying the following condition which we state for ease of reference.

\begin{assumption}\label{Ae}\rm
The integer $r\ge 2$ and real number $0<\eta<1/100$ are fixed.
The functions $p(n)$, $\la(n)$ and $\eps(n)$ are related by
$\la=1+\eps$ and $p=\la (r-2)!n^{-r+1}$. Moreover,
as $n\to\infty$, we have $\eps^3n\to\infty$ and $\eps\to 0$.
\end{assumption}

With $\eta>0$ and $\eps(n)$ given as above, set
\begin{equation}\label{alphadef}
 \alpha = \frac{\eta}{100r}.
\end{equation}
We shall mark the vertices of our random hypergraph $H=\Hrnp$ independently with probability
\[
 p_{\rm{mark}} = \alpha \eps^2 = \alpha \eps(n)^2.
\]

We shall treat $\Hrnp$ as a marked hypergraph without
explicitly indicating the set $V^*$ of marked vertices in the notation.
Let $\excore(\cL_1)$ be the extended core of the marked hypergraph $\cL_1$,
where, as usual, $\cL_1$ is the largest component of $\Hrnp$.
Thus $\excore(\cL_1)$ is a component of $\excore(\Hrnp)$, except
in the unlikely event that $\cL_1$ is an unmarked tree, in which case $\excore(\cL_1)=\emptyset$.
Recall that $L_1=|\cL_1|$. The next few lemmas gather
properties of $\excore(\Hrnp)$ and its `mantle' that we shall need.
A key point is that these results hold conditional on the giant
component $\cL_1$ having a specific order $s$ and nullity $t$, provided $s$ is in the typical
range $\range$ defined in \eqref{rangedef}.  For this reason
they do not obviously follow from `global' results saying that whp the
(extended) core has some property. Another key point
is that we can afford to give up constant factors in the estimates
of the size of the extended core and of its mantle. Throughout the rest of this section,
$p$, $\la$, $\eps$ and $\eta$ satisfy Assumption~\ref{Ae}, and we define $\range$
as in \eqref{rangedef}.
All new constants introduced below 
may depend on the choice of the function $\eps=\eps(n)$ and of $\eta$.

\begin{lemma}\label{ecbig}
Let $r\ge 2$, $\eta>0$ and $\eps=\eps(n)$ satisfying Assumption~\ref{Ae} be given.
Then there is a constant $c_1>0$ such that, for $n$ large enough, for any $s=s(n)\in \range$
and $t=t(n)\ge 1$ we have
\begin{equation}\label{ecbig1}
 \Pr\bb{ |\excore(\cL_1)| > c_1\eps^2n \mid L_1=s, N_1=t } \ge 1-\eta.
\end{equation}
\end{lemma}
\begin{proof}
We shall condition not only on the event $\{L_1=s,N_1=t\}$, but also on the vertex set of $\cL_1$
and on the entire hypergraph structure of its core $C(\cL_1)$.
The extended core contains the core; if the core is not already large enough, we shall show
that with conditional probability at least $1-\eta$, the interaction of the
marked vertices with the core generates an extended core of at least the required size.

Turning to the details,
by \eqref{rangecont} there is a constant $c_0>0$ that depends only on the function $\eps(n)$,
such that
\begin{equation}\label{slower}
 s\in\range_n \hbox{\quad implies\quad} s\ge c_0\eps n.
\end{equation}
We shall prove \eqref{ecbig1} with
\begin{equation}\label{c1choice}
 c_1 = \frac{\alpha c_0}{4} = \frac{\eta c_0}{400r}.
\end{equation}
First, by Chebyshev's inequality, if $X$ has a binomial distribution
with mean $\mu\ge 8/\eta$ (and so variance less than $\mu$) then $\Pr(X\ge \mu/2)\ge 1-\eta/2$. Hence,
from \eqref{slower} and the assumption $\eps^3n\to\infty$,
there is an $n_0$ such that for all $n\ge n_0$
\begin{equation}\label{binlarge}
 s\in\range_n \hbox{\quad implies\quad} \Pr\bb{\Bi(s,\alpha\eps^2) \ge \alpha\eps^2 s/2} \ge 1-\eta/2.
\end{equation}
By assumption $\eps=\eps(n)$ satisfies $\eps\to 0$ and $\eps^3n\to\infty$. Hence, increasing
$n_0$ if necessary, for all $n\ge n_0$ we have
\begin{equation}\label{epsc}
 \eps \le 1/100\hbox{\quad and\quad} \eps^3n \ge 8r/c_1.
\end{equation}

From now on, let $n\ge n_0$, $s\in\range_n$ and $t\ge 1$ be given.
We condition on the event $\cE$ that $L_1=s$, $N_1=t$, the vertex set of $\cL_1$ is some specific
set $V_1$ of $s$ vertices, and the usual (non-extended) core $\core(\cL_1)$ is some particular hypergraph
with vertex set $V_2\subset V_1$. We write $a=|V_2|$.
Our aim is to show that
\begin{equation}\label{cecaim}
 \Pr\bb{|\excore(\cL_1)|\le c_1\eps^2n \bigm| \cE } \le \eta.
\end{equation}
Since $\cL_1$ and $\core(\cL_1)$ have the same nullity, we may assume that $\core(\cL_1)$ has nullity $t$; in fact, we only
need the trivial consequence that $\core(\cL_1)$ is not empty.\footnote{%
In proving Theorem~\ref{thL1only}, we do not condition on the nullity $n(\cL_1)$.
This means we cannot \emph{a priori} assume that $\core(\cL_1)$ is non-empty. However, it is immediate
from the formulae given by Karo\'nski and \L uczak~\cite[Theorem 9]{KL_sparse} for the number of
connected hypergraphs on $s$ vertices with a given small excess 
that $\Pr(n(\cL_1)=0\mid L_1=s)=o(\Pr(n(\cL_1)=r-1\mid L_1=s))$, so $\Pr(n(\cL_1)=0\mid L_1=s)=o(1)$.
Hence we can indeed assume that $a\ge 1$.}
Since $\excore(\cL_1)\supset \core(\cL_1)$, if $a>c_1\eps^2n$ then the conditional probability in \eqref{cecaim} is 0.
Thus we may assume that
\begin{equation}\label{asmall}
 1\le a\le c_1\eps^2n.
\end{equation}

Relabelling, let us take the vertex set of $\cL_1$ to be $[s]$ and that of its
core to be $[a]\subset [s]$. From the definition of the core, $\cL_1$
is the union of its core and an $[a]$-rooted $r$-forest $F$ on $[s]$.
Since this forest $F$ does not affect the core, after conditioning on $\cE$ as above, 
$F$ is uniformly random on all such forests.
Recall that we mark vertices independently with probability $\alpha\eps^2$, where $\alpha$ is given
in \eqref{alphadef}. Since $\cL_1$ and its core $C(\cL_1)$
are defined without reference to the set $V^*$ of marked vertices,
each vertex of $[s]$ is marked independently of the others and of the random
forest $F$.

Set $\ell=\ceil{\eps^{-1}}$.
Call a marked vertex $v\in \cL_1$ \emph{bad} if either 

\smallskip
(i) it is at distance at most $\ell$ from $[a]=V(\core(\cL_1))$ or

(ii) it is joined to another marked vertex by a path in $F=\cL_1-C(\cL_1)$ of length
at most $2\ell$.

\smallskip\noindent
If $v$ is not bad, we call it \emph{good}.

Every marked vertex in $\cL_1$ is on a path to the core $\core(\cL_1)$. The union of these paths
is a subgraph $F^*$ of the forest $F$, and $\excore(\cL_1)=\core(\cL_1)\cup F^*$,
with each component of $F^*$ meeting $\core(\cL_1)$ in a single vertex.
For each good marked vertex $v$, consider the first
$\ell$ edges of the path to the core starting at $v$: these shortened paths
are necessarily disjoint, so $|\excore(\cL_1)|$ is at least $\ell$
times (in fact, at least $(r-1)\ell$ times) the number of good marked vertices.
As the number of marked vertices in $\cL_1$ has the binomial distribution $\Bi(s,\alpha\eps^2)$,
by \eqref{binlarge} the probability that there are at least $\alpha\eps^2s/2$ marked vertices in $\cL_1$
is at least $1-\eta/2$. We \emph{claim} that, conditional on $\cE$,
the expected number of bad marked vertices is at most $\eta\alpha\eps^2s/8$.
Assuming this then, by Markov's inequality, with probability at least $1-\eta/2$
there are at most $\alpha\eps^2s/4$ bad marked vertices, and hence with
probability at least $1-\eta$ there are 
at least $\alpha\eps^2s/2-\alpha\eps^2s/4=\alpha\eps^2s/4$ good marked vertices. But then,
recalling \eqref{slower} and \eqref{c1choice},
\[
 |\excore(\cL_1)|\ge \ell\alpha\eps^2s/4 \ge \alpha\eps s/4 \ge \alpha c_0 \eps^2n/4 = c_1\eps^2n.
\]

To prove the claim, let $v$ be a vertex in $[s]=V(\cL_1)$ chosen uniformly at random.
We must show that
the probability that $v$ is a bad marked vertex is at most $\eta\alpha\eps^2/8$.
So first condition on the event that $v$ is marked; it remains to show that the conditional
probability that (i) or (ii) holds is at most $\eta/8$.

For (i), this conditional probability is exactly $1/s$ times the expectation $\mu$
of the number of vertices in $[s]$
within distance $\ell$ of $[a]$.
From Lemma~\ref{l_dt} and \eqref{asmall},
\[
 \mu \le \sum_{0\le j\le \ell} (a+(r-1)j) \le 2a\ell +2r\ell^2 \le  2c_1\eps^2 n\ell + 2r\ell^2.
\]
Since $n\ge n_0$, from \eqref{epsc} we have $\ell=\ceil{\eps^{-1}}\le 2\eps^{-1}$, say,
and $8r/(\eps^3n)\le c_1$. Thus
\[
 \mu \le  4c_1\eps n +8r\eps^{-2} \le 5c_1\eps n = \frac{\eta c_0\eps n}{80r} \le \frac{\eta s}{80r},
\]
recalling \eqref{slower}. Hence the conditional probability $\mu/s$ that (i) holds is at most
$\eta/(80r) < \eta/16$.

For (ii), the components of the forest $F$ give a partition of the vertex set $[s]$ of $\cL_1$
into $a$ parts (some of which may be singletons). Let us condition on the vertex $v$
and on this partition. The component $T$ of $F$ containing $v$ is then a uniformly random $r$-tree
on its vertex set $X$. Viewing $v$ as the root, we can regard this $r$-tree
as a $\{v\}$-rooted $r$-forest, and then by Lemma~\ref{l_dt} the expected number of
vertices $w\ne v$ joined to $v$ by paths in $F$ (and hence in $T$) of length at most $2\ell$ is at most
\[
 \sum_{1\le j\le 2\ell} (1+(r-1)j) \le 4r\ell^2 = 4r\ceil{1/\eps}^2.
\]
Hence the probability that one or more such vertices are marked is at most
$4r\ceil{1/\eps}^2\alpha\eps^2$. From \eqref{epsc} and \eqref{alphadef} this
probability is at most $5r\alpha\le \eta/16$. Thus the conditional probability
that (i) or (ii) holds is at most $\eta/8$, completing the proof of the claim and hence of the lemma.
\end{proof}

We have shown that with high (conditional) probability, the extended core $\excore(\cL_1)$ of the largest
component is not too small. Roughly speaking, our next aim is to
show that with high probability the rest of the extended core,
i.e., $\excore(\Hrnp)\setminus \excore(\cL_1) = \excore(\Hrnp\setminus \cL_1)$ is neither too small
nor too big. While this is not too hard, it turns out that we can avoid some work
by considering instead the set
\begin{equation}\label{Idef}
 \cI = \{ \hbox{ isolated vertices in the hypergraph  }\excore(\Hrnp\setminus\cL_1)\ \}.
\end{equation}
By definition, an isolated vertex in $\excore(\Hrnp\setminus\cL_1)$ is marked (otherwise
it would be deleted in defining the extended core).
By Proposition~\ref{excore}, each $v\in \cI$ corresponds
to a tree component of $\Hrnp\setminus\cL_1$ containing
exactly one marked vertex, namely $v$.

\begin{lemma}\label{isolcore}
Let $r\ge 2$, $\eta>0$ and $\eps=\eps(n)$ satisfying Assumption~\ref{Ae} be given.
Then there is a constant $c_2>0$ such that,
for any $s=s(n)\in\range$ and $t=t(n)\ge 0$,
\[
 \Pr\bb{ c_2\eps^2n \le |\cI| \le 2\alpha\eps^2n \mid L_1=s, N_1=t } =1-o(1).
\]
\end{lemma}
\begin{proof}
The upper bound on $|\cI|$ is trivial.
Indeed, any vertex of $\cI$ is marked, by the definition of the extended core.
Given that $L_1=s$, and any further information about $\cL_1$,
the number of marked vertices in $\Hrnp\setminus \cL_1$
has the binomial distribution $\Bi(n-s,\alpha\eps^2)$, with mean at most $\alpha\eps^2n\to\infty$, so
with high probability this number is at most $2\alpha\eps^2n$.

Turning to the lower bound, by Lemma~\ref{dual2} it suffices to show that whp $H'=\Hrnsp$
(with vertices marked independently with probability $\alpha\eps^2$) has at least $c_2\eps^2n$
isolated vertices in its extended core. An elementary first and second moment calculation (or the
case $k=0$ of Lemma~\ref{smalltree}) shows that whp $H'$ has $\Theta(n)$ isolated vertices.
Since each is marked independently with probability $\alpha\eps^2$ and, if marked, is
an isolated vertex of $\excore(H')$, the result follows from concentration of the binomial distribution.
\end{proof}

Let $H$ be a hypergraph with extended core $\excore(H)$. We define the \emph{mantle} $M^+(H)$
to be the set of vertices of $H$ not in $\excore(H)$ but connected to it by paths. Thus
$\excore(H)\cup M^+(H)$ includes all vertices of $H$ except those in tree components
with no marked vertices.
By Proposition~\ref{excore}, each $w\in M^+(H)$ is connected by a path in the mantle to a unique
vertex $v\in \excore(H)$; for $A\subset V(\excore(H))$ we write $M^+(A)$ for the set of $w\in M^+(H)$
whose corresponding core vertex $v$ is in $A$.

\begin{lemma}\label{isolml}
Let $r\ge 2$, $\eta>0$ and $\eps=\eps(n)$ satisfying Assumption~\ref{Ae} be given.
Then there is a constant $c_3>0$ such that,
for any $s=s(n)\in\range$ and $t=t(n)\ge 0$, we have
\[
 \Pr\bb{ |M^+(\cI)| \ge c_3\eps n\mid L_1=s, N_1=t } =1-o(1).
\]
\end{lemma}
\begin{proof}
Condition on the event that $L_1=s$ and $N_1=t$.
From Corollary~\ref{csmallt}, with conditional probability $1-o(1)$ the hypergraph
$\Hrnp\setminus\cL_1$ contains at least $c\eps^3n$ tree components
each having between $\ceil{\eps^{-2}}$ and $2\ceil{\eps^{-2}}$ edges, and so $\Theta(\eps^{-2})$
vertices. Having revealed the graph $\Hrnp$,
for each such tree, the probability that it contains exactly one marked vertex is at least some
constant $c'>0$. So the conditional distribution of the number $X$ of such trees containing
exactly one marked vertex stochastically
dominates a Binomial distribution with mean $cc'\eps^3n$. Since $\eps^3n\to\infty$,
it follows that whp $X\ge cc'\eps^3n/2$. Since each tree counted by $X$ contains at least
$1+(r-1)\ceil{\eps^{-2}}$
vertices, and so contributes at least $(r-1)\ceil{\eps^{-2}}\ge\eps^{-2}$
vertices to $M^+(\cI)$, the result follows.
\end{proof}

\begin{lemma}\label{isolmu}
Let $r\ge 2$, $\eta>0$ and $\eps=\eps(n)$ satisfying Assumption~\ref{Ae} be given.
Then there is a constant $c_4>0$ such that, for $n$ large enough, 
for every $s\in\range$ and $t\ge 0$ we have
\[
 \Pr\bb{ |M^+(\cI)| \le c_4\eps n\mid L_1=s, N_1=t } \ge 1-\eta.
\] 
\end{lemma}
\begin{proof}
Given a marked hypergraph $H$, let $X(H)$ denote the number of vertices $v$ of $H$
with the property that $v$ is joined to some marked vertex of $H$ by a path in $H$.
Note that every vertex of $M^+(\cI)$ has this property in $H^-=\Hrnp\setminus\cL_1$,
so $|M^+(\cI)|\le X(H^-)$. Hence, by Markov's inequality,
it suffices to show that $\E[X(H^-)\mid L_1=s, N_1=t] = O(\eps n)$.\footnote{We need this bound to hold uniformly over $s\in\range_n$ and $t\ge 0$; for this we just consider the worst-case $s(n)$ and $t(n)$.}
Now $X(H^-)$ is at most the number of ordered pairs $(v,w)$ of vertices of
$H^-$ with $v$ marked and $v$, $w$ joined by a path,
so
\[
 \E[X(H^-)\mid L_1=s, N_1=t]\le \alpha\eps^2\E[\Nc(H^-)\mid L_1=s, N_1=t]
\]
which, by Lemma~\ref{Ncon}, is $O(\eps^2 n/\eps)=O(\eps n)$.
\end{proof}

\section{The core smoothing argument}\label{sec_main}

In this section we prove Theorem~\ref{thL1}; this is all that remains to complete
the proof of Theorem~\ref{thprob}. The strategy that we follow is outlined at the start of Section~\ref{sec_ext}.
Recall that we always relate $p=p(n)$ and $\eps=\eps(n)$ by
\[
 \la(n)=1+\eps(n) \hbox{\quad and\quad} p(n)=\la(n) (r-2)!n^{-r+1}.
\]
Define $\range=\range_n$ as in \eqref{rangedef};
in this section we shall consider sequences $(x_n)$, $(y_n)$ and $(t_n)$ of integers such that
\begin{equation}\label{scond}
 t_n\ge 2, \quad x_n,y_n\in\range_n, \quad x_n-y_n=o(\sqrt{n/\eps}), \quad\hbox{and}\quad
x_n\equiv y_n\equiv 1-t_n,
\end{equation}
where the congruence condition is modulo $r-1$. This
condition arises since otherwise there are no
$r$-uniform hypergraphs with nullity $t_n$ and $x_n$ or $y_n$ vertices.
The following lemma captures (a particular form of) what is needed to prove Theorem~\ref{thL1}.
Here $\alpha\pm\beta$ denotes a quantity in the range $[\alpha-\beta,\alpha+\beta]$.

\begin{lemma}\label{llast}
Suppose that $p(n)$ satisfies our Standard Assumption~\ref{A0}, that $0<\eta< 1/100$ is constant,
and that the sequences $(x_n)$, $(y_n)$ and $(t_n)$ satisfy \eqref{scond}.
Then
\begin{equation}\label{bigo}
 \Pr(L_1=y_n, N_1=t_n) = O(1/(\eps n)),
\end{equation}
and, for $n$ large enough,
\begin{equation}\label{etaapp}
 \Pr(L_1=x_n, N_1=t_n) = (1\pm 30\eta)\bb{\Pr(L_1=y_n, N_1=t_n) \pm \eta/(\eps n)}.
\end{equation}
\end{lemma}
As usual, the implicit constant in \eqref{bigo} may depend on all
choices so far, i.e., on the sequences $(p(n))$, $(x_n)$, $(y_n)$ and $(t_n)$ and constants $r$
and $\eta$, just of course not on $n$. (See Remark~\ref{rem:asy1}.)
The same applies to the implicit constant $n_0$ in `for $n$ large enough'.

Before proving Lemma~\ref{llast}, which will take most of the section, we show that it implies
Theorem~\ref{thL1}.
\begin{proof}[Proof of Theorem~\ref{thL1}, assuming Lemma~\ref{llast}.]
Theorem~\ref{thL1} asserts that, given $r\ge 2$, a sequence $(p(n))$ (and hence $\eps(n))$
satisfying Assumption~\ref{A0}, and sequences $(x_n)$, $(y_n)$ and $(t_n)$ satisfying \eqref{scond},
we have
\begin{equation}\label{L1diff2}
 \Pr(L_1=x_n, N_1=t_n)-\Pr(L_1=y_n, N_1=t_n) = o(1/(\eps n)).
\end{equation}
In proving this we may of course fix $r\ge 2$, $(p(n))$, $(x_n)$, $(y_n)$ and $(t_n)$ as above,
and $0<\delta\le 1$, say. Then we must show that for all large enough $n$ (depending on all choices so far),
we have
\begin{equation}\label{L1diff2a}
 \bigl| \Pr(L_1=x_n, N_1=t_n)-\Pr(L_1=y_n, N_1=t_n) \bigr| \le \delta/(\eps n).
\end{equation}

By the first part of Lemma~\ref{llast}, applied with $\eta=1/200$, say, there is a constant $C$ (which may depend
on all choices so far) such that $\Pr(L_1=y_n, N_1=t_n)\le C/(\eps n)$. We may assume $C>1$.
Let $\eta=\delta/(60C)\le \delta/4$.
By the second part of Lemma~\ref{llast},
if $n$ is large enough then
\[
 \Pr(L_1=x_n, N_1=t_n) = (1\pm 30\eta)\bb{\Pr(L_1=y_n, N_1=t_n) \pm \delta/(4\eps n)}.
\]
Since $(1+30\eta)\le 2$, this gives
\[
 \Pr(L_1=x_n, N_1=t_n) = \Pr(L_1=y_n, N_1=t_n) \pm \bb{30 \eta C/(\eps n) + \delta/(2\eps n)},
\]
which implies \eqref{L1diff2a} since $30 \eta C= \delta /2$.
\end{proof}

It remains to prove Lemma~\ref{llast}. In doing so we may of course fix
$r\ge 2$, sequences $(p(n))$, $(x_n)$, $(y_n)$, $(t_n)$, and a real number $0<\eta<1/100$
such that our Standard Assumption~\ref{A0} holds, as does \eqref{scond}.
Any new constants introduced may depend on these choices.
Note that Assumption~\ref{Ae} of Section~\ref{sec_ext} holds.

Define the largest component $\cL_1$ of $\Hrnp$ as before,
and the extended core $\excore(\Hrnp)$ and the set $\cI$ as in Section~\ref{sec_ext} (see \eqref{Idef}).
Define $\range$ as in \eqref{rangedef}. By \eqref{rangecont}, there
are constants $c_0>0$ and $c_5$ such that, for $n$ large,
\[
 \range = [(1-\delta)\rho_{r,\la}n,(1+\delta)\rho_{r,\la}n] \subset [c_0\eps n,c_5\eps n].
\]
Set
\[
 c=\min\{c_0,c_1,c_2,c_3\}\hbox{\quad and\quad} C=\max\{c_4,c_5\},
\]
where the constants $c_i$, $1\le i\le 4$, are as in Lemmas~\ref{ecbig}--\ref{isolmu}.

Let $\cA$ be the event that the following conditions hold:

\smallskip
(i) $|\excore(\cL_1)|\ge c\eps^2 n$,

(ii) $c\eps^2n \le |\cI| \le 2\alpha\eps^2 n$,

(iii) $|M^+(\cI)|\ge c\eps n$,

(iv) $|M^+(\cI)|\le C\eps n$\quad and

(v) $c\eps n \le |\excore(\cL_1)|+|M^+(\excore(\cL_1))|\le C\eps n$.

\begin{claim}\label{PrA}
For $n$ sufficiently large, for any $s\in\range$ and any $t\ge 2$ we have
\begin{equation}\label{cpA}
 \Pr( \cA \mid L_1=s, N_1=t )\ge 1-3\eta.
\end{equation}
\end{claim}
\begin{proof}
Lemmas~\ref{ecbig}, \ref{isolcore}, \ref{isolml} and~\ref{isolmu} imply that properties
(i)--(iv) hold with conditional probability at least $1-(2\eta+o(1))\ge 1-3\eta$ for $n$ large.
Whenever (i) holds then in particular $\excore(\cL_1)$ is not empty.
But then by, Proposition~\ref{excore}, $|\excore(\cL_1)|+|M^+(\excore(\cL_1))|=|\cL_1|=L_1=s\in\range$, so (v) holds.
\end{proof}

As before, let $\ucx$ be the event
\[
 \ucx = \{\ \cL_1\hbox{ is the unique complex component of }\Hrnp\ \},
\]
so $\ucx$ holds whp by Lemma~\ref{C1unique}.
Let $\eco$ be the component of $\excore(\Hrnp)$ with the highest nullity/excess,
chosen according to any fixed rule if there is a tie,
and let $\cI'$ be the set of isolated vertices of $\excore(\Hrnp)\setminus \eco$.
Note that if $\ucx$ holds, then $\excore(\Hrnp)$ has a unique complex component,
and we have $\eco=\excore(\cL_1)$ and so $\cI'=\cI$.
We shall define an event $\cB$ that is closely related to $\cA$,
but defined using $\eco$ and $\cI'$ in place of $\cC_1$ and $\cI$. The point is that
we would like to condition on the extended core (and some further information), and then use
the remaining randomness concerning which parts of the mantle are joined to the largest
component as our smoothing distribution. But until this remaining randomness has been revealed,
we do not know which component is largest, so we cannot easily condition on $\cA$.

Let $a_1=|\eco|$, $a_0=|\cI'|$, and $a=\min\{a_1,a_0\}$. Given the entire extended core,
pick sets $A_1\subset V(\eco)$ and $A_0\subset\cI'$ with $|A_1|=|A_0|=a$, for example
by choosing in each case the first $a$ eligible vertices in a fixed order. (This is
mostly a convenience; with a little more work we could work directly with $\eco$ and $\cI'$.)
Let $\cB$ be the event that the following hold:

\smallskip
(I) $c\eps^2 n\le a\le 2\alpha\eps^2n$\quad and

(II) $c\eps n/2 \le |M^+(A_1\cup A_0)| \le 2C\eps n$.

\begin{claim}\label{AUB}
If $n$ is large enough, then whenever $\cA\cap\ucx$ holds, so does $\cB$.
\end{claim}
\begin{proof}
Suppose that $\cA\cap \ucx$ holds.
Then, since $\ucx$ holds, $\eco=\excore(\cL_1)$. Since $|\excore(\cL_1)|\ge c\eps^2n$ by 
condition (i) of $\cA$, we have $a_1\ge c\eps^2n$. Also, $a_0=|\cI'|=|\cI|$ is between $c\eps^2n$
and $2\alpha\eps^2n$ by (ii). Since $a=\min\{a_1,a_0\}$, this gives (I). Consider next the upper
bound in (II). Since $\ucx$ holds, $A_1\cup A_0\subset \eco\cup\cI' = \excore(\cL_1)\cup \cI$,
so $M^+(A_1\cup A_0)\subset M^+(\excore(\cL_1))\cup M^+(\cI)$, and
(iv) and (v) imply $|M^+(A_1\cup A_0)|\le 2C\eps n$.
For the lower bound we have two cases: if $a_0\le a_1$ then $A_0=\cI'=\cI$ 
so $|M^+(A_1\cup A_0)|\ge |M^+(A_0)|=|M^+(\cI)|\ge c\eps n$ by (iii).
If $a_1\le a_0$ then $A_1=\eco=\excore(\cL_1)$, so from (v) we have
\[
 |M^+(A_1\cup A_0)|\ge |M^+(\excore(\cL_1))| \ge c\eps n-a_1
 = c\eps n-a \ge c\eps n-2\alpha\eps^2n,
\]
by (I). Since $\alpha\le 1$ and $\eps\to 0$,
if $n$ is large enough then it follows
that $|M^+(A_1\cup A_0)| \ge c\eps n/2$, so (II) holds.
\end{proof}

At this point the reader may forget the definition of $\cA$; we work with $\cB$ from now on.

\begin{claim}\label{clB}
For $n$ sufficiently large, for any $s\in\range$ and any $t\ge 2$ we have
\begin{equation}\label{cpBU}
 \Pr(\cB\cap \ucx\mid L_1=s, N_1=t)\ge 1-4\eta,
\end{equation}
\begin{equation}\label{PB1}
 \Pr(\cB) \ge 1-5\eta
\end{equation}
and
\begin{equation}\label{PB2}
 \Pr(\cB) \ge 1/2.
\end{equation}
\end{claim}
\begin{proof}
For any $s\in\range$ and $t\ge 2$, by Lemma~\ref{C1unique},
\begin{equation}\label{Ulikely}
  \Pr(\ucx\mid L_1=s, N_1=t)=1-o(1).
\end{equation}
Since $\cA\cap\ucx$ implies $\cB$, it follows from this and \eqref{cpA}
that, if $n$ is large enough, then \eqref{cpBU} holds.
In turn, we deduce that
\[
 \Pr(\cB)\ge \Pr(\cB\cap \ucx) \ge (1-4\eta)\Pr(L_1\in\range,\,N_1\ge 2).
\]
Since $L_1\in\range$ whp (from \eqref{rangewhp})
and (by Theorem~\ref{thglobal2}, say) $N_1\ge 2$ whp,
it follows that $\Pr(\cB)\ge 1-4\eta-o(1)$.
Hence \eqref{PB1} holds for $n$ large enough.
Of course \eqref{PB2} (stated only for convenient reference later) follows,
since $\eta\le 1/100$.
\end{proof}

We now have the pieces in place to complete the proof of Lemma~\ref{llast} and hence of Theorem~\ref{thL1}.

\begin{proof}[Proof of Lemma~\ref{llast}.]
We start by revealing the following
partial information about our random marked hypergraph $H=\Hrnp$. First reveal $\excore(H)$,
and in particular which vertices are marked. Define $\eco$, $A_1$ and $A_0$ as above, noting that these
depend only on $\excore(H)$. Reveal $M^+(A_1\cup A_0)$, the set of non-core vertices joined
by paths to $A_1\cup A_0$. Also (although this is not necessary), reveal all hyperedges outside
$\excore(H)\cup M^+(A_1\cup A_0)$. We write $\cF=\cF_n$ for the $\sigma$-algebra generated by the information
revealed so far. Note that the event $\cB$ defined above is $\cF$-measurable.

What have we not yet revealed? Let $F$ be the subgraph of $H$ induced by $V=A_1\cup A_0\cup M^+(A_1\cup A_0)$
with any edges inside $A_1\cup A_0$ removed (these removed edges are in $\excore(H)$).
By Proposition~\ref{excore} and the definition of $M^+(A_1\cup A_0)$, the hypergraph $F$ is an $(A_1\cup A_0)$-rooted
$r$-forest on $V$. Moreover, replacing one such forest by another does not affect $\excore(H)$,
or indeed any information revealed earlier. Thus, conditional on $\cF$, the distribution of $F$
is uniform over all $(A_1\cup A_0)$-rooted $r$-forests on $V$; this uniform choice of the forest
$F$ is the only remaining randomness.

When $\cB$ holds, $|A_1|=|A_0|=a=\Theta(\eps^2 n)$, while $m=e(F)=|M^+(A_1\cup A_0)|/(r-1)=\Theta(\eps n)$.
Since $\eps=O(1)$ we have $m=\Omega(a)$. Also, $a^2/m=\Theta(\eps^3 n)\to\infty$, so $m=o(a^2)$.
Hence the conditions of Lemma~\ref{ldist} are satisfied.
Let $Y_n=|M^+(A_1)|$ be the number of vertices in $V\setminus (A_1\cup A_0)$ joined to $A_1$ (rather than
to $A_0$).
Since $m^{3/2}a^{-1}=\Theta(\sqrt{n/\eps})$,
Lemma~\ref{ldist} tells us that when $\cB$ holds and $x_n'-y_n'=o(\sqrt{n/\eps})$, then
\begin{equation}\label{csmooth}
 \Pr(Y_n = x_n' \mid \cF)-\Pr(Y_n=y_n' \mid \cF) = o(\sqrt{\eps/n})
\end{equation}
and
\begin{equation}\label{csmooth2}
 \Pr(Y_n=y_n'\mid \cF)=O(\sqrt{\eps/n}).
\end{equation}

Let $\cC^*$ be the component of $H=\Hrnp$ containing $\eco$, and let $L^*=|\cC^*|$
and $N^*$ denote the order and nullity of $\cC^*$. Since $\cC^*$ consists of $\eco$
with a forest attached, $N^*$ is also the nullity of $\eco$ and so is an $\cF$-measurable
random variable.
Let $\cE$ denote the event that $\Hrnp\setminus\cL_1$ has a complex component,
so $\{N^*=t_n\}\subset \{N_1=t_n\}\cup\cE$.
Theorem~\ref{thN1} implies that $\Pr(N_1=t_n)=O((\eps^3n)^{-1/2})$.
By the last part of Lemma~\ref{C1unique}, we have $\Pr(\cE)=O(1/(\eps^3n))$, so
\[
  \Pr(N^*=t_n) \le \Pr(N_1=t_n) + \Pr(\cE) = O((\eps^3n)^{-1/2}).
\]
It follows from this and \eqref{PB2} that
\begin{equation}\label{N*}
 \Pr(N^*=t_n\mid\cB)  \le  2\Pr(N^*=t_n) = O((\eps^3n)^{-1/2}).
\end{equation}

Given $\cF$, the only uncertainly (i.e., not-yet-revealed information)
affecting $L^*$ is which vertices of $M^+(A_1\cup A_0)$ join to $A_1$
rather than to $A_0$. Thus
we may write $L^*$ as $X_n+Y_n$ where $X_n$ is $\cF$-measurable and $Y_n$ is defined as above. 
Hence, when $\cB$ holds,
\begin{multline}\label{Lxy}
 \Pr(L^* = x_n \mid \cF)-\Pr(L^*=y_n \mid \cF) \\ = 
 \Pr(Y_n = x_n -X_n \mid \cF)-\Pr(Y_n=y_n -X_n \mid \cF) = 
 o(\sqrt{\eps/n}),
\end{multline}
by \eqref{csmooth} with $x_n'=x_n-X_n$ and $y_n'=y_n-X_n$.
Taking the expectation\footnote{Again, this requires a uniform bound, but we have that by considering the worst-case
$\omega_n\in\cB$ in \eqref{csmooth} and \eqref{Lxy}.}
 over the $\cF$-measurable
event $\cB\cap\{N^*=t_n\}$, it follows that
\begin{multline}\label{sB}
  \Pr(L^* = x_n, N^*=t_n \mid \cB)-\Pr(L^*=y_n,N^*=t_n \mid \cB) \\ 
 =  o\bb{\sqrt{\eps/n}\ \Pr(N^*=t_n\mid\cB)}  = o(1/(\eps n)),
\end{multline}
where the last step is from \eqref{N*}.
Similarly, from \eqref{csmooth2} and \eqref{N*} we see that
\begin{equation}\label{sB2}
 \Pr(L^*=y_n, N^*=t_n\mid \cB)=O\bb{\sqrt{\eps/n}\ \Pr(N^*=t_n\mid\cB)} = O(1/(\eps n)).
\end{equation}
It remains to remove the conditioning, and to replace $L^*$ by $L_1$.

Recall that when $\ucx$ holds, then $\eco=\excore(\cL_1)$, so $\cC^*=\cL_1$, and hence $L_1=L^*$ and $N_1=N^*$.
Let $s=s(n)\in\range$, and let $t=t(n)\ge 2$.
If $(L^*,N^*)=(s,t)$ but $(L_1,N_1)\ne (s,t)$, then there is a component with
$s$ vertices and nullity $t$ which is not the unique largest component.
By Lemma~\ref{lbigep} (in particular from \eqref{ep2}), we thus have
\[
 \Pr\bb{(L^*,N^*)=(s,t),\,(L_1,N_1)\ne (s,t)} = o\bb{ \Pr((L_1,N_1)=(s,t)) }.
\]
Using \eqref{cpBU} for the first inequality, and 
recalling that
\[
 \{L_1=s, N_1=t\}\cap\cB\cap\ucx =  \{L^*=s, N^*=t\}\cap\cB\cap\ucx,
\]
we have
\begin{eqnarray*}
 (1-4\eta)\Pr(L_1=s, N_1=t) 
&\le&  \Pr(\{L_1=s, N_1=t\}\cap \cB\cap\ucx)  \\
&=&    \Pr(\{L^*=s, N^*=t\}\cap \cB\cap\ucx)  \\
&\le&  \Pr(\{L^*=s, N^*=t\}\cap \cB) \\
&\le&  \Pr(L^*=s, N^*=t) \\
&\le&  \Pr(L_1=s, N_1=t) \\
&&\hskip 0.5cm +\ \Pr(L^*=s, N^*=t, (L_1,N_1)\ne (s,t))  \\
&=&    (1+o(1))\Pr(L_1=s, N_1=t).
\end{eqnarray*}
Hence, for $n$ large,
\begin{equation}\label{L*N*LN}
 \Pr(\{L^*=s, N^*=t\}\cap \cB) = (1\pm 4\eta)\Pr(L_1=s, N_1=t).
\end{equation}
Relations \eqref{L*N*LN} and \eqref{PB1} imply that
\begin{multline}\label{transf}
 \Pr(L^*=s, N^*=t\mid \cB) = \frac{\Pr(\{L^*=s, N^*=t\}\cap \cB)}{\Pr(\cB)} \\
  = (1\pm 10\eta) \Pr(L_1=s, N_1=t),
\end{multline}
since $0<\eta<1/50$.
Applying \eqref{transf} (backwards) with $s=x_n$ and $t=t_n$,
then \eqref{sB}, then \eqref{transf} with $s=y_n$ and $t=t_n$, we deduce that
\begin{multline*}
 \Pr(L_1=x_n, N_1=t_n) \\
 = (1\pm 10\eta)^{-1}\bb{ (1\pm 10\eta)\Pr(L_1=y_n, N_1=t_n) + o(1/(\eps n)) }.
\end{multline*}
Since $0<\eta<1/30$ this implies \eqref{etaapp} for $n$ large enough.
Similarly, from \eqref{transf} and \eqref{sB2} we deduce \eqref{bigo},
completing the proof of Theorem~\ref{thL1}.
\end{proof}

Finally, let us comment briefly on the proof of Theorem~\ref{thL1only}. The arguments in this
section and the previous one can be modified to prove Theorem~\ref{thL1only}, by omitting all conditioning
on $N_1$, and replacing the quantity $1/(\eps n)$ where it appears as the order
of a point probability (for example in \eqref{sB} and \eqref{sB2}) by $\sqrt{\eps/n}$,
which is (within a constant factor) the probability that $L_1$ takes a given typical value.
At almost all points nothing needs to be added to the argument. Two exceptions
are in the proof of Lemma~\ref{ecbig}, and that
in place of \eqref{Ulikely} we need $\Pr(\ucx\mid L_1=s)=1-o(1)$.
See the footnote to the proof of Lemma~\ref{ecbig} for an argument covering both of these.

\section{Proof of Theorem~\ref{thenum}}\label{sec_deduce}

In this section we shall deduce Theorem~\ref{thenum} from Theorem~\ref{thprob}. The only additional
result needed for this is Lemma~\ref{lbigep}; however, the formulae are rather messy and we will devote some
space to calculations aimed at simplifying them.

\begin{proof}[Proof of Theorem~\ref{thenum}.]
Let $r\ge 2$ be fixed, and suppose that $t=t(s)\to\infty$ as $s\to\infty$;
our aim is to give an asymptotic formula for the number $C_r(s,t)$ of connected 
$r$-uniform hypergraphs on $[s]$ having nullity $t$. From \eqref{nulldef} the number
$m$ of edges of any such hypergraph satisfies
\[
 m=\frac{s+t-1}{r-1}.
\]
In particular, we must have $s+t$ congruent to $1$ modulo $r-1$ for $C_r(s,t)$ to be non-zero. We assume this
from now on. We also assume that $t=o(s)$ and $t\to\infty$.
More precisely, we fix a function $t=t(s)$ with these properties; we shall define a number
of other quantities in terms of $s$ and $t$. Except where otherwise specified, all limits and asymptotic notation then 
refer to $s\to\infty$.

The function $\Psi_r(x)$ defined in \eqref{Prdef}
is continuous on $(0,1)$ and tends to $0$ as $x\to 0$ and to infinity as $x\to 1$.
Also, as mentioned in the introduction, $\Psi_r(x)$ is strictly increasing on $(0,1)$; hence,
for $s$ large enough that $t\ge 2$, the equation 
\eqref{rdefc} has a unique positive solution $\rho=\rho(s)$.
Expanding about $x=0$ we see that $\Psi_r(x)=\frac{r-1}{12}x^2+O(x^3)$, uniformly in $0<x\le 1/2$, say.
Thus
\begin{equation}\label{rstsim}
 \rho \sim 2\sqrt{\frac{3}{r-1}\frac{t}{s}}
\end{equation}
as $s\to\infty$.

Define
\begin{equation}\label{rr2}
 \rho_2 = \rho_2(s) = 1-(1-\rho)^{r-1}
\end{equation}
and
\begin{equation}\label{rla}
 \la = \la(s) = \frac{-\log(1-\rho_2)}{\rho_2} = \frac{-(r-1)\log(1-\rho)}{1-(1-\rho)^{r-1}}.
\end{equation}
Note that $\la>1$; comparing \eqref{rr2} and \eqref{rla} with \eqref{rldef} and \eqref{rkldef}
we see that in the notation of the rest of the paper,
\[
 \rho_2 = \rho_\la=\rho_{2,\la} \hbox{\qquad and\qquad}\rho=\rho_{r,\la}.
\]

As $s\to\infty$, from \eqref{rstsim} we have $\rho=\rho(s)\to 0$. Thus, from \eqref{rr2},
$\rho_2 \sim (r-1)\rho$. Hence
\[
 \la = 1+\rho_2/2+O(\rho_2^2)
\]
and
\begin{equation}\label{neps}
 \eps=\la-1\sim \frac{\rho_2}{2} \sim \frac{r-1}{2}\rho \sim \sqrt{3(r-1)\frac{t}{s}} \to 0.
\end{equation}

Set
\[
 n=n(s)=\floor{s/\rho}\hbox{\qquad and\qquad}p=p(s)=\la\frac{(r-2)!}{n^{r-1}}.
\]
Since $\rho\to 0$ as $s\to\infty$, certainly $n\to\infty$ and $n\sim s/\rho$. Hence, 
from \eqref{neps},
\[
 \eps n \sim \frac{r-1}{2} s.
\]
From \eqref{neps} we also have $\eps\to 0$. In addition,
\[
 \eps^3n = \eps^2 (\eps n) = \Theta( (t/s) s)=\Theta(t) \to \infty.
\]
Hence our Standard Assumption~\ref{A0} is satisfied, i.e., we have the conditions needed to apply Theorem~\ref{thprob}.
(Of course, here we consider a sequence $(n(s),\eps(s))_{s\ge 1}$ of values rather than a sequence $(n,\eps(n))_{n\ge 1}$.
This causes no problems since we can pass to subsequences on which $n(s)$ is strictly increasing.)

We have chosen the parameters $n$ and $p$ so that the `typical' order and nullity of the largest
component of $\Hrnp$ will be very close to $s$ and $t$, respectively. More precisely, 
for the `typical' number $\rho_{r,\la}n$ of vertices we have
\[
 \rho_{r,\la} n = \rho n = \rho\floor{s/\rho} = s+O(\rho) = s+O(\eps).
\]
For the nullity, recalling \eqref{rhosdef} and \eqref{rla} we see that
the formula \eqref{Prdef} defining $\Psi_r$ may be written as
\begin{equation}\label{Prdef2}
 \Psi_r(\rho_{r,\la})=\rho_{r,\la}^*/\rho_{r,\la}.
\end{equation}
Indeed, this is how we arrived at this formula. Since $\rho_{r,\la}=\rho$ it follows using
\eqref{rdefc} that
\[
 \rho_{r,\la}^* n = \Psi_r(\rho) \rho n = \frac{t-1}s \rho n = t-1+O(\eps^3) = t+O(1).
\]
The standard deviations $\sigma_n$ and $\sigma_n^*$ appearing in Theorem~\ref{thprob} tend to
infinity, so certainly we have $s=\rho_{r,\la}n+o(\sigma_n)$ and $t=\rho_{r,\la}^* n +o(\sigma^*_n)$.
Hence, by Theorem~\ref{thprob}, and in particular the formula \eqref{pp2} (with $a=b=0$),
\begin{equation}\label{f1}
 \Pr\bb{ L_1(\Hrnp)=s,\, N_1(\Hrnp)=t } \sim \frac{\sqrt{6}}{8\pi}\frac{(r-1)^2}{\eps n}
 \sim \frac{\sqrt{6}}{4\pi}\frac{r-1}{s}.
\end{equation}
On the other hand, applying Lemma~\ref{lbigep} with $\cQ_s$ the set of all $r$-uniform hypergraphs with
$s$ vertices and nullity $t$, writing $N_{s,t}$ for the number of components of $\Hrnp$
with the property $\cQ_s$, we have
\begin{equation}\label{f2}
 \Pr\bb{ L_1(\Hrnp)=s,\, N_1(\Hrnp)=t } \sim \E[N_{s,t}].
\end{equation}
By linearity of expectation,
\begin{equation}\label{f3}
 \E[N_{s,t}] = \binom{n}{s} C_r(s,t) p^m (1-p)^{\binom{n}{r}-\binom{n-s}{r}-m}.
\end{equation}
Combining \eqref{f1}--\eqref{f3} we see that
\begin{equation}\label{f4}
 C_r(s,t) \sim \frac{\sqrt{6}}{4\pi}\frac{r-1}{s} 
 \binom{n}{s}^{-1} p^{-m} (1-p)^{-\left(\binom{n}{r}-\binom{n-s}{r}-m\right)}.
\end{equation}
In the rest of this section
we simplify this formula, in particular by showing that we can replace
$n=\floor{s/\rho}$ by $s/\rho$, for example.

Working in terms of $n$ and $\eps$ (the more familiar parameters from the bulk of the paper) we have
\[
 s=\Theta(\eps n),\quad t=\Theta(\eps^3 n),\quad m=\Theta(\eps n), \quad p=O(n^{-r+1})= O(n^{-1}).
\]
It follows immediately that $pm=o(1)$, so $(1-p)^m\sim 1$.
Also,
\[
 r!\binom{n}{r} = n(n-1)\cdots(n-r+1) = n^r-\binom{r}{2}n^{r-1}+O(n^{r-2})
\]
and
\[
 r!\binom{n-s}{r} = (n-s)^r -\binom{r}{2}(n-s)^{r-1}+O(n^{r-2}) = (n-s)^r - \binom{r}{2} n^{r-1}  + O(sn^{r-2}).
\]
Subtracting, we see that
\[
 \binom{n}{r}-\binom{n-s}{r} = \frac{n^r-(n-s)^r}{r!} + O(sn^{r-2}) 
 = \frac{n^r-(n-s)^r}{r!} + o(n^{r-1}) = o(n^r).
\]
Since $\log\bb{(1-p)^k}=-pk+O(p^2k)$ it follows easily that
\begin{multline*}
 a= -\log\left( (1-p)^{\binom{n}{r}-\binom{n-s}{r}-m} \right)
  = p\frac{n^r-(n-s)^r}{r!}+o(1) \\
  = \frac{\la n}{r(r-1)}(1-(1-s/n)^r) = \frac{\la s}{r(r-1)} f(s/n)
\end{multline*}
where $f(x)=x^{-1}(1-(1-x)^r)$. Since $f'(x)=O(1)$ for $x=O(1)$ and $s/n-\rho=O(\eps/n)$
we have $f(s/n)-f(\rho)=O(\eps/n)$,  so $sf(s/n)-sf(\rho)=O(\eps^2)=o(1)$.
Hence
\[
 a= \frac{\la s}{r(r-1)}f(\rho)+o(1).
\]
From \eqref{Prdef}, \eqref{rla} and \eqref{rdefc} it follows that
\[
 a = s\frac{\Psi_r(\rho)+1}{r-1} +o(1) =  s\frac{(t-1)/s+1}{r-1} +o(1) = \frac{s+t-1}{r-1} +o(1) = m+o(1).
\]
From \eqref{f4} we now obtain the formula
\[
 C_r(s,t) \sim \frac{\sqrt{6}}{4\pi} \frac{r-1}{s} e^m p^{-m} \binom{n}{s}^{-1}.
\]

By Stirling's formula,
\[
 \binom{n}{s}^{-1} \sim \sqrt{2\pi}\sqrt{\frac{s(n-s)}{n}} \frac{s^s(n-s)^{n-s}}{n^n}
 \sim \sqrt{2\pi s} \left(\frac{s}{n}\right)^s  \left(1-\frac{s}{n}\right)^{n-s}.
\] 
Since $s=\rho n+O(\eps)$, we have $s/n=\rho(1+O(1/n))$. Also, $1-s/n=(1-\rho)(1+O(\eps/n))$, and it follows
that
\[
 \binom{n}{s}^{-1} \sim \sqrt{2\pi s} \rho^s (1-\rho)^{s(1-\rho)/\rho}.
\]

Using again that $s/n=\rho(1+O(1/n))$, and that $m=O(\eps n)=o(n)$, we have
\[
 p^{-m} = \frac{n^{(r-1)m}}{\la^m(r-2)!^m} \sim \frac{s^{(r-1)m}}{\la^m(r-2)!^m\rho^{(r-1)m}}.
\]
Next, we shall eliminate $\la$ from this expression. From \eqref{rdefc}, \eqref{Prdef2}
and \eqref{rhosdef} we have
\[
 \frac{(r-1)m}{s}= \frac{s+t-1}{s} = 1+\Psi_r(\rho) = 1+\frac{\rho^*_{r,\la}}{\rho} =
 \frac{\la}{r\rho}\bb{1-(1-\rho)^r}.
\]
Hence
\[
 \la^m = ((r-1)m/s)^m r^m \rho^m \bb{1-(1-\rho)^r}^{-m}.
\]

Putting the pieces together we obtain the asymptotic formula
\begin{eqnarray}
 C_r(s,t) &\sim& \frac{\sqrt{6}}{4\pi} \frac{r-1}{s} e^m
  \frac{s^{(r-1)m}}{\la^m (r-2)!^m \rho^{(r-1)m}} 
\sqrt{2\pi s} \rho^s (1-\rho)^{s(1-\rho)/\rho} \nonumber \\
 &=& \frac{\sqrt{3}}{2\sqrt{\pi}} \frac{r-1}{\sqrt{s}} e^m
  \frac{s^{(r-1)m}}{\la^m (r-2)!^m \rho^{(r-1)m}} 
 \rho^s (1-\rho)^{s(1-\rho)/\rho} \nonumber \\
 &=& \frac{\sqrt{3}}{2\sqrt{\pi}} \frac{r-1}{\sqrt{s}} e^m
  \frac{(1-(1-\rho)^r)^m s^{(r-1)m}}{((r-1)m/s)^m r^m \rho^m (r-2)!^m \rho^{(r-1)m}} 
 \rho^s(1-\rho)^{s(1-\rho)/\rho} \nonumber \\
 &=& \frac{\sqrt{3}}{2\sqrt{\pi}} \frac{r-1}{\sqrt{s}} e^m
  \frac{(1-(1-\rho)^r)^m s^{rm}}{m^m r!^m \rho^{rm}}
 \rho^s (1-\rho)^{s(1-\rho)/\rho}, \label{Cr1}
\end{eqnarray}
proving the main formula \eqref{enumform} of Theorem~\ref{thenum}.

Turning to \eqref{Pform}, let
\[
 N=\binom{s}{r} = \frac{s(s-1)\cdots (s-r+1)}{r!} = \frac{s^r}{r!}e^{-\binom{r}{2}/s+O(s^{-2})}.
\]
Since $m\sim s/(r-1)$, it follows that
\[
 N^m \sim \frac{s^{rm}}{r!^m} e^{-\binom{r}{2}m/s} \sim \frac{s^{rm}}{r!^m} e^{-r/2}.
\]
Since $N=\Theta(s^r)$, for $r\ge 3$ we have $m^2=o(N)$, and it follows that
\[
 \binom{N}{m} = \frac{N(N-1)\cdots (N-m+1)}{m!} \sim \frac{N^m}{m!}.
\]
On the other hand, if $r=2$ then $m\sim s$ and $N\sim s^2/2$, so 
\[
  \binom{N}{m} = \frac{N^m}{m!} e^{-\binom{m}{2}/N+o(1)} \sim e^{-1} \frac{N^m}{m!}.
\]
We may write the last two formulae together as $\binom{N}{m}\sim e^{-\ind{r=2}} N^m/m!$, where
$\ind{A}$ denotes the indicator function of $A$.
Hence, using Stirling's formula, and recalling that $m=(s+t-1)/(r-1) \sim s/(r-1)$, 
\[
 \binom{N}{m} \sim \frac{e^{-r/2-\ind{r=2}}}{\sqrt{2\pi m}} \frac{e^ms^{rm}}{m^mr!^m} 
 \sim \frac{e^{-r/2-\ind{r=2}}}{\sqrt{2\pi s/(r-1)}}  \frac{e^ms^{rm}}{m^mr!^m} .
\]
From this and \eqref{Cr1} we obtain the expression
\begin{eqnarray*}
 P_r(s,t) &\sim&  e^{r/2+\ind{r=2}}\sqrt{2\pi s/(r-1)} \frac{\sqrt{3}}{2\sqrt{\pi}} \frac{r-1}{\sqrt{s}}
   \frac{(1-(1-\rho)^r)^m }{ \rho^{rm}}
 \rho^s(1-\rho)^{s(1-\rho)/\rho} \\
 &=& e^{r/2+\ind{r=2}} \sqrt{\frac{3(r-1)}{2}}  \left(\frac{1-(1-\rho)^r}{\rho^r}\right)^m
  \left(\rho(1-\rho)^{(1-\rho)/\rho}\right)^s,
\end{eqnarray*}
completing the proof.
\end{proof}

\noindent
{\bf Acknowledgements.}
We would like to thank the referee for a careful reading of the paper, and for suggestions that
led to significant improvements in the presentation.

\newpage
\appendix
\section{Appendix}

In this appendix we show that Theorem~\ref{thenum} is compatible with previous results and,
in Subsection~\ref{Alr}, give a proof of Lemma~\ref{l_rtree}.

As in the statement of Theorem~\ref{thenum}, we write $C_r(s,t)$ for the number
of connected $r$-uniform hypergraphs on $[s]=\{1,2,\ldots,s\}$ having nullity $t$.
Also, with $m=(s+t-1)/(r-1)$ the number of edges of such a hypergraph,
we write $P_r(s,t)$ for the probability that a random $m$-edge $r$-uniform hypergraph
on $[s]$ is connected.

\subsection{The Behrisch--Coja-Oghlan--Kang formula}

Behrisch, Coja-Oghlan and Kang~\cite{BC-OK2pre,BC-OK2abs,BC-OK2b} gave
an asymptotic formula
for the number of connected $r$-uniform hypergraphs with $s$ vertices and nullity
$t=\Theta(s)$. As noted below, their result implies asymptotic formulae
for $C_r(s,t)$ and $P_r(s,t)$ valid if $t/s\to 0$ sufficiently slowly
as $s\to\infty$. Here we show that Theorem~\ref{thenum} is consistent with the 
(single) formula given in the preprint~\cite{BC-OK2pre}, extended abstract~\cite{BC-OK2abs},
and corrected version of~\cite{BC-OK2b}.\footnote{In a previous draft of this appendix
we showed that Theorem~\ref{thenum} is not consistent with a different formula
given in the original
published version of~\cite{BC-OK2b}; Behrisch, Coja-Oghlan and Kang have
since published a corrigendum.}

Behrisch, Coja-Oghlan and Kang~\cite{BC-OK2pre,BC-OK2abs,BC-OK2b}
write $\zeta$ for the average degree of the hypergraphs
under consideration; in our notation this is $rm/s=(r/(r-1))(s+t-1)/s$.
They write $d$ rather than $r$ for the number of vertices
in each hyperedge, and define a quantity~$r$ implicitly by the equation
\begin{equation}\label{theireq}
 r= \exp\left(-\zeta \frac{(1-r)(1-r^{d-1})}{1-r^d}\right).
\end{equation}
Transforming to our notation by writing $r$ instead of $d$, and substituting
$1-\rho$ for the variable $r$ being solved for, this equation becomes
\[
 1-\rho= \exp\left(- \frac{r}{r-1}\frac{s+t-1}{s} \frac{\rho(1-(1-\rho)^{r-1})}{1-(1-\rho)^r}\right).
\]
Taking logs, this is easily seen to be equivalent to \eqref{rdefc}, so the quantity $r$
appearing in their results is exactly $1-\rho$ where $\rho$ is defined as in Theorem~\ref{thenum}.

Behrisch, Coja-Oghlan and Kang~\cite{BC-OK2pre,BC-OK2abs,BC-OK2b}
give an asymptotic formula for $P_r(s,t)$ of the following form, valid whenever $t=\Theta(s)$.
Here we have partially translated to our notation, writing $r$ for the size of a hyperedge
and replacing their $r$ by $1-\rho$:
\begin{equation}\label{theirgen}
 P_r(s,t) \sim f_r(\rho,\zeta) \exp(g_r(\rho,\zeta)) \Phi_r(\rho,\zeta)^s,
\end{equation}
where $f_r$, $g_r$ and $\Phi_r$ are algebraic functions of $\rho$ and $\zeta$.
Translating from their notation 
\[
 \Phi_d(r,\zeta) = r^{\frac{r}{1-r}}(1-r)^{1-\zeta}(1-r^d)^{\zeta/d}
\]
to our notation, we obtain
\[
 \Phi_r(\rho,\gamma) = (1-\rho)^{\frac{1-\rho}{\rho}} \rho^{1-r\gamma} (1-(1-\rho)^r)^\gamma,
\]
where $\gamma=\zeta/r=m/s$. Hence, the factor $\Phi_r(\rho,\zeta)^s$ in \eqref{theirgen}
is exactly the factor
\[
 \left(\frac{1-(1-\rho)^r}{\rho^r}\right)^m  \bigl( \rho(1-\rho)^{(1-\rho)/\rho}\bigr)^s
\]
in \eqref{Pform}, and Theorem~\ref{thenum} states that if $t=o(s)$ then
\[
 P_r(s,t) \sim c_r \Phi_r(\rho,\zeta)^s
\]
where
\[
 c_r= e^{r/2} \sqrt{\frac{3(r-1)}{2}}
\]
for $r\ge 3$ and $c_2=e^2\sqrt{3/2}$.

For any constant $a$, the asymptotic formula~\eqref{theirgen} is valid
for $t=t(s)$ in the range $[s/a,a s]$. It follows that it must also be valid for $t=t(s)$
such that $t/s$, or equivalently $(t-1)/s$, tends to zero at some rate,
though we cannot say what. Hence, the combination of our
result and \eqref{theirgen} imply that
\[
 f_r(\rho,\zeta) \exp(g_r(\rho,\zeta)) \to c_r
\]
in the appropriate limit. Since
\begin{equation}\label{zeta}
 \zeta= \frac{r m}{s} = \frac{r}{r-1}\frac{s+t-1}{s} = \frac{r}{r-1}\left(1+\frac{t-1}{s}\right)
\end{equation}
depends only on the ratio $\alpha=(t-1)/s$ and not on $s$,
and $\rho$ is a function of $\zeta$ and hence of $\alpha$, we see that the limit
above must hold as $\alpha\to 0$; the quantity $s$ does not appear in this statement.

Since $\alpha\to 0$ and $\rho\to 0$ are equivalent, it is convenient to work instead in
terms of $\rho$. Defining $\zeta(\rho)$ by \eqref{rdefc} and \eqref{zeta} or,
equivalently, by \eqref{theireq} with $r=1-\rho$, we must have
\[
 f_r(\rho,\zeta(\rho)) \exp({g_r(\rho,\zeta(\rho))}) \to c_r
\]
as $\rho\to 0$. 

In checking this, let us mix notation in such a way that all symbols
are unambiguous. Thus we write $d$ for the number of vertices in a hyperedge,
and avoid $r$, replacing it by $1-\rho$.
Rearranging \eqref{theireq} for $\zeta$ as a function of $\rho=1-r$, we find
that
\begin{equation}\label{zr}
 \zeta = -\frac{\log(1-\rho)}{\rho} \frac{1-(1-\rho)^d}{1-(1-\rho)^{d-1}} =
 \frac{d}{d-1}\left(1+\frac{d-1}{12}\rho^2 +O(\rho^3)\right).
\end{equation}

For $d\ge 3$, substituting this and $r=1-\rho$ into the formulae
\[
 g_d(r,\zeta) = \frac{\zeta(d-1)(r-2r^d+r^{d-1})}{2(1-r^d)}
\]
and
\[
 f_d(r,\zeta) = a_d(r,\zeta)/\sqrt{b_d(r,\zeta)}
\]
where
\[
 a_d(r,\zeta) = 1-r^d-(1-r)(d-1)\zeta r^{d-1}
\]
and
\[
 b_d(r,\zeta) = \bb{1-r^d+\zeta(d-1)(r-r^{d-1}) }(1-r^d) - d\zeta r(1-r^{d-1})^2
\]
given in~\cite[Theorem 5]{BC-OK2pre}, \cite[Theorem 3]{BC-OK2abs}
and the corrected version of~\cite[Theorem 1.1]{BC-OK2b},
we see that
\[
  a_d(r,\zeta) \sim \frac{d(d-1)}{2}\rho^2,
 \hbox{\qquad}
  b_d(r,\zeta) \sim \frac{d^2(d-1)}{6}\rho^4
 \hbox{\qquad and\qquad} g_d(r,\zeta)\to d/2,
\]
which combine to give
\begin{equation}\label{Gdlim}
  f_d(\rho,\zeta(\rho)) \exp(g_d(\rho,\zeta(\rho))) \to \sqrt{\frac{3(d-1)}{2}} e^{d/2} =  c_d.
\end{equation}
A similar but simpler calculation for the graph case $d=2$ gives
\begin{multline}\label{G2lim}
 f_2 \exp(g_2) = \frac{1+r-\zeta r}{\sqrt{(1+r)^2-2\zeta r}}
\exp\left( \frac{2\zeta r+\zeta^2r}{2(1+r)} \right) \\
 \sim \frac{\rho}{\sqrt{\tfrac{2}{3}\rho^2}} e^2 \to e^2\sqrt{3/2} = c_2.
\end{multline}
In other words, our results are consistent with those of Behrisch, Coja-Oghlan and Kang.
Of course, since the ranges of applicability are different, our results neither imply, nor are implied by, theirs.

\medskip
Although in this section we concentrate on comparing enumerative formulae, we should like to point out that,
like our Theorem~\ref{thenum}, the enumerative results of Behrisch, Coja-Oghlan and Kang
are deduced from a probabilistic result, the local limit theorem in~\cite{BC-OK2a}.
Bearing in mind the relationship $N_1=(r-1)M_1-L_1+1$ between the number $M_1$
of edges, number $L_1$ of vertices, and nullity $N_1$ of the largest component
of the random hypergraph $\Hrnp$,
\cite[Theorem 1.1]{BC-OK2a} translates to a local limit result for $(L_1,N_1)$
with variance $\sigma_{\cN}^2$ for $L_1$, variance
\[
 (r-1)^2\sigma_{\cM}^2 + \sigma_{\cN}^2 -2(r-1)\sigma_{\cN\cM}\
\]
for $N_1$, and covariance $(r-1)\sigma_{\cM\cN}-\sigma_{\cN}^2$.
Noting that $\rho$ in \cite{BC-OK2a} is what we call $1-\rho$, we have checked using Maple
that the formulae given in~\cite{BC-OK2a} give the right asymptotics (matching Theorem~\ref{thprob}) when
the branching factor tends to~$1$.

\subsection{The Bender--Canfield--McKay formula}

For graphs, Bender, Canfield and McKay~\cite{BCMcK} give the following asymptotic formula
for the probability $P_2(s,t)$ that a random graph on $[s]$ with $m=s+t-1$ edges
is connected:
\begin{equation}\label{bck}
 P_2(s,t) \sim e^{a(x)}\left(\frac{2e^{-x}y^{1-x}}{\sqrt{1-y^2}}\right)^s,
\end{equation}
where $x=m/s$, $y=y(x)$ is defined implicitly by
\begin{equation}\label{BCKdef}
 2xy=\log\left(\frac{1+y}{1-y}\right),
\end{equation}
and
\begin{equation}\label{adef}
 a(x) =x(x+1)(1-y)+\log(1-x+xy)-\tfrac{1}{2}\log(1-x+xy^2).
\end{equation}
Here we have changed the notation to match ours, and have simplified the more precise error term
given in~\cite{BCMcK}. The formula \eqref{bck} is valid whenever $t\to\infty$ and $m\le \binom{s}{2}-s$.
In particular, it is certainly valid in the range $t=o(s)$ that we consider.

Recall that we define $\rho$ by \eqref{rdefc}, i.e., by 
\[
 \Psi_r(\rho) = \frac{t-1}{s} =\frac{m}{s} -1 =x-1,
\]
where, substituting $r=2$ into \eqref{Prdef},
\[
 \Psi_2(\rho)=-\frac{1}{2}\frac{\log(1-\rho)}{\rho}\frac{2\rho-\rho^2}{\rho}-1 = 
 -\frac{\log(1-\rho)}{2} \frac{2-\rho}{\rho} -1.
\]
Hence, $\rho=\rho(x)$ satisfies
\begin{equation}\label{xdef}
 2x = -\log(1-\rho)\frac{2-\rho}{\rho}.
\end{equation}

Let
\begin{equation}\label{ydef}
 y= \frac{\rho}{2-\rho}.
\end{equation}
Then $\rho=2y/(y+1)$, and it is easy to check that \eqref{BCKdef} is satisfied, so this $y=y(x)$ coincides
with that defined in \cite{BCMcK}.
Substituting \eqref{xdef} and \eqref{ydef} into \eqref{adef} gives an explicit formula for $a(x)$ in terms
of $\rho$; expanding around $\rho=0$ (using Maple), it turns out that
\[
 a(x) \to 2+\log(3/2)/2
\]
as $\rho\to 0$, so in our setting \eqref{bck} simplifies to
\begin{equation}\label{BCK2}
 P_2(s,t) \sim e^2\frac{\sqrt{3}}{\sqrt{2}} 
   \left(\frac{2e^{-x}y^{1-x}}{\sqrt{1-y^2}}\right)^s
   = e^2\frac{\sqrt{3}}{\sqrt{2}} 
 y^{-xs} \left(\frac{2e^{-x}y}{\sqrt{1-y^2}}\right)^s.
\end{equation}
Now from \eqref{xdef}
\[
 e^{-x} = (1-\rho)^{\frac{2-\rho}{2\rho}} = (1-\rho)^{\frac{1}{\rho}-\frac{1}{2}}.
\]
Also, since $1-y^2=(4-4\rho+\rho^2-\rho^2)/(2-\rho)^2=4(1-\rho)/(2-\rho)^2$, we have
\[
 \frac{2y}{\sqrt{1-y^2}} = \frac{2\rho}{2-\rho} \frac{2-\rho}{2\sqrt{1-\rho}} = \frac{\rho}{\sqrt{1-\rho}}.
\]
Thus
\[ 
 \frac{2e^{-x}y}{\sqrt{1-y^2}} = \rho (1-\rho)^{\frac{1}{\rho}-1} = \rho (1-\rho)^{\frac{1-\rho}{\rho}}.
\]
Since $xs=m$ and $1/y=(2-\rho)/\rho$, the formula \eqref{BCK2} may be written as
\[
 P_2(s,t) \sim e^2\frac{\sqrt{3}}{\sqrt{2}}  \left(\frac{2-\rho}{\rho}\right)^m 
 \left( \rho (1-\rho)^{\frac{1-\rho}{\rho}} \right)^s,
\]
which is exactly what \eqref{Pform} states when $r=2$. Hence the graph case of Theorem~\ref{thenum}
is consistent with (and indeed implied by) the results of Bender, Canfield and McKay~\cite{BCMcK}.

\subsection{The Sato--Wormald formula}

Sato and Wormald~\cite{SatoWormald} give an asymptotic formula for $C(N,M)$, the number
of connected $3$-uniform hypergraphs with $N$ vertices and $M$ edges, valid when $M=N/2+R$
with $R=o(N)$ and $R/(N^{1/3}\log^2 N)\to\infty$. Translating to our notation, $N=s$
and 
\[
 \frac{s+t-1}{2} = m = M = N/2+R = s/2 +R,
\]
so $R=(t-1)/2$. They define a quantity $\la^{**}$, which we shall write as $\mu$,
to be the unique positive solution to 
\[
 \mu \frac{e^{2\mu}+e^\mu+1}{(e^\mu-1)(e^\mu+1)}=3M/N = 3m/s.
\]
Rewriting this equation as
\[
 \mu \frac{1+e^{-\mu}+e^{-2\mu}}{(1-e^{-\mu})(1+e^{-\mu})} = 3m/s,
\]
it is easy to see that the solution is $\mu=-\log(1-\rho)$, where we define $\rho$ by the $r=3$ case of \eqref{rdefc},
i.e., by
\[
  \Psi_3(\rho) = -\frac{2}{3} \frac{\log(1-\rho)}{\rho} \frac{1-(1-\rho)^3}{1-(1-\rho)^2} -1 = \frac{t-1}{s} = \frac{2m}{s}-1.
\]

Sato and Wormald then define
\[
 \check{n}^* = \frac{e^{2\mu}-1-2\mu}{(e^\mu-1)(e^\mu+1)} = \frac{1-(1+2\mu)e^{-2\mu}}{1-e^{-2\mu}},
\]
so in our notation
\[
 \check{n}^* = 1+\frac{2\log(1-\rho)(1-\rho)^2}{\rho(2-\rho)}.
\]
From this point we use Maple to rewrite the Sato--Wormald formula in terms of $\rho$ and $s$.
We may rewrite their main formula for $C(N,M)=C_3(s,t)$ as
\begin{equation}\label{SWform}
 C(N,M) \sim \sqrt{\frac{3}{\pi N}} \exp\left(N\tilde{\phi}(\check{n}^*)\right) \exp\left((2R/N+1)N\log N\right),
\end{equation}
where
\begin{multline*}
 \tilde{\phi}(x) = -\frac{1-x}{2}\log(1-x)+\frac{1-x}{2} -(\log 2+2)\frac{R}{N} -\frac{\log 2}{2}x \\
 + \frac{R}{N}\log\left(\frac{e^\mu+1}{\mu(e^\mu-1)}\right)
 + \frac{1}{2}x\log\left(\frac{(e^\mu-1)(e^\mu+1)}{\mu}\right) -1.
\end{multline*}
[Here we have added $1-2R/N\log(N)$ to their $\phi$ to define $\tilde\phi$, and adjusted \eqref{SWform}
accordingly.]
Now, in our notation, the quantity $R/N$ appearing in \cite{SatoWormald} is
\begin{equation}\label{RN}
 \frac{R}{N}=\frac{m-s/2}{s} = \frac{(s+t-1)/2-s/2}{s} = \frac{t-1}{2s} = \frac{\Psi_3(\rho)}{2}.
\end{equation}
Since $2R/N+1=2m/s$, in our notation we may rewrite \eqref{SWform} as
\[
 C_3(s,t) \sim \sqrt{\frac{3}{\pi s}}  \exp\left(s\tilde{\phi}(\check{n}^*)\right) s^{2m}. 
\]
In the case $r=3$ we may write \eqref{enumform} as
\[
 C_3(s,t) \sim \sqrt{\frac{3}{\pi s}} \psi^s s^{2m},
\]
where
\[
 \psi=\psi(t/s)=  \left( \frac{ e (1-(1-\rho)^3) }{6 (m/s) \rho^3} \right)^{m/s}
  \rho(1-\rho)^{(1-\rho)/\rho} .
\]

Since $m/s=(1+\Psi_3(\rho))/2$, we can write $\psi$ explicitly as a function of $\rho$ only.
Using \eqref{RN} and the formula $\mu=-\log(1-\rho)$, we can also write $\tilde\phi(\check{n}^*)$
as a function of $\rho$ only. Since each formula only involves $\rho$, it follows that
our formula and that of Sato and Wormald are consistent if and only if 
$\tilde\phi(\check{n}^*)$ and $\log \psi$ reduce to the same function of $\rho$. 
At this point we enlist the help of Maple,
which assures us that they do.
We hope that the reader will take this on trust (or check it themselves), especially given
that Sato and Wormald~\cite{SatoWormald} themselves check consistency of their result
with the $r=3$ case of the result in~\cite{BC-OK2pre}, and, as we have shown, ours
is also consistent with this.

Note that the check above shows that the formula given in~\cite{SatoWormald} is not only
asymptotically equal to ours in the range in which it applies (as it must be if our
results and theirs are correct): the expressions are equal, although this is far
from obvious. Our Theorem~\ref{thenum} says that the formula in~\cite{SatoWormald}
applies much more widely than shown in~\cite{SatoWormald}.

\subsection{The Karo\'nski--\L uczak formula}

Karo\'nski and \L uczak~\cite{KL_sparse} gave an asymptotic
formula for $C_r(s,t)$ valid when $r\ge 2$ is constant, $t\to\infty$, and
$t=o(\log s/\log\log s)$. (They also give formulae for $t$ constant.)
Mixing their notation and ours, writing $r$ for the number of vertices
in a hyperedge, $s$ for the number of vertices of the hypergraphs
being counted, and $k=t-1$ for their excess (nullity minus 1), their formula becomes
\[
 \sqrt{\frac{3}{4\pi}} \left(\frac{e}{12 k}\right)^{k/2}\frac{(r-1)^{k/2+1}}{(r-2)!^{k/(r-1)}}
\left(\frac{e^{2-r}}{(r-2)!}\right)^{s/(r-1)} s^{s+3k/2-1/2}.
\]
Noting that $m=(s+t-1)/(r-1)=(s+k)/(r-1)$, we may rewrite this as
\begin{equation}\label{KLform}
 \sqrt{\frac{3}{4\pi}} \frac{r-1}{\sqrt{s}} \left( \frac{(r-1)e}{12}\right)^{k/2} e^{(2-r)s/(r-1)} (r-2)!^{-m} s^{s+3k/2}k^{-k/2}.
\end{equation}
Aiming to separate out the factors that grow superexponentially in $s$ and/or in $k$, 
letting
\begin{equation}\label{afdef}
 f(\rho) = \frac{1-(1-\rho)^r}{r\rho} = 1 - \frac{r-1}{2}\rho + O(\rho^2),
\end{equation}
we may write \eqref{enumform} as
\begin{eqnarray*}
 C_r(s,t) &\sim& \frac{\sqrt{3}}{2\sqrt{\pi}} \frac{r-1}{\sqrt{s}} 
   \left(\frac{e f(\rho) s^{r-1}} {(m/s) (r-1)! \rho^{r-1}}\right)^m  \bigl( \rho(1-\rho)^{(1-\rho)/\rho}\bigr)^s \\
&=& \sqrt{\frac{3}{4\pi}} \frac{r-1}{\sqrt{s}} 
   \left(\frac{e f(\rho)}{ (m/s)(r-1)(r-2)!}\right)^m \bigl( (1-\rho)^{(1-\rho)/\rho}\bigr)^s \rho^{s-(r-1)m}s^{(r-1)m} \\
&=& \sqrt{\frac{3}{4\pi}} \frac{r-1}{\sqrt{s}} 
   \left(\frac{e f(\rho)}{ (r-1)m/s }\right)^m \bigl( (1-\rho)^{(1-\rho)/\rho}\bigr)^s (r-2)!^{-m} s^{s+k}\rho^{-k}.
\end{eqnarray*}

Recall from \eqref{rdefc} that $\Psi_r(\rho)=(t-1)/s=k/s$, where from simple calculus,
\[
 \Psi_r(\rho) = \frac{r-1}{12}\rho^2+\frac{r-1}{12}\rho^3+O(\rho^4).
\]
It follows that we may write
\[
 \rho = \tau \sqrt{\frac{12}{r-1}\frac{k}{s}}
\]
where $\tau=1+O(\rho)$ as $\rho\to 0$. (Of course, we can expand $\tau$ further in powers of $\rho$ if we wish.)
Then
\[
 s^{s+k}\rho^{-k} = s^{s+3k/2}k^{-k/2} \left(\frac{r-1}{12}\right)^{k/2} \tau^{-k},
\]
so
\begin{multline*}
 C_r(s,t) \sim \sqrt{\frac{3}{4\pi}} \frac{r-1}{\sqrt{s}} 
   \left(\frac{r-1}{12}\right)^{k/2} \left(\frac{e f(\rho)}{(r-1)m/s}\right)^m \\
 \bigl( (1-\rho)^{(1-\rho)/\rho}\bigr)^s (r-2)!^{-m}
 s^{s+3k/2}k^{-k/2}\tau^{-k}.
\end{multline*}
Comparing this with \eqref{KLform}, we see that our asymptotic formula and that of Karo\'nski and \L uczak agree whenever
\[
 \exp\left(k/2+\frac{(2-r)s}{r-1}\right)
  \sim \left( e f(\rho) \frac{s}{(r-1)m}\right)^m \bigl( (1-\rho)^{(1-\rho)/\rho}\bigr)^s \tau^{-k}.
\]
Noting that $(r-1)m=s+k$, raising both sides to the power $r-1$ this is equivalent to
\begin{multline}\label{cs}
 \exp\left( (2-r)s +\frac{r-1}{2}k\right) \\
 \sim \left( e f(\rho) \frac{s}{s+k}\right)^{s+k} 
\bigl( (1-\rho)^{(1-\rho)/\rho}\bigr)^{(r-1)s} \tau^{-(r-1)k}.
\end{multline}
Now we follow our earlier strategy of obtaining explicit formulae in terms of $\rho$ and then expanding.
Using $k/s=\Phi_r(\rho)$,
\[
 \tau = \rho\sqrt{\frac{r-1}{12}\frac{s}{k}}
\]
and \eqref{afdef}, with Maple we find that
after taking logarithms and dividing by $s$, the two sides of \eqref{cs} differ by $\Theta(\rho^4)$ as
$\rho\to 0$. Noting that $s\rho^4\to 0$ if and only if $s(\sqrt{k/s})^4\to 0$, i.e., if and only if $k=o(\sqrt{s})$,
this implies that our formula and that of Karo\'nski and \L uczak agree if $k=o(\sqrt{s})$,
i.e., if $t=o(\sqrt{s})$, but not in general. Thus our results are consistent
with theirs.
Furthermore, our result shows that their formula, which they prove only for $k=o(\log s/\log\log s)$,
remains valid for any $k=o(\sqrt{s})$. Note that Karo\'nski and \L uczak~\cite{KL_sparse}
state that they expect their formula to remain true for $k=o(s^{1/3})$, but to be hard to prove.
Note also that Andriamampianina and Ravelomanana~\cite{AR} give such an extension
to $k=o(s^{1/3})$ in an extended abstract.

\subsection{Proof of Lemma~\ref{l_rtree}}\label{Alr}

Although Selivanov~\cite{Selivanov} gives a proof of Lemma~\ref{l_rtree},
we include a proof here, since the
reference is a little obscure and the result is straightforward,

\begin{proof}[Proof of Lemma~\ref{l_rtree}.]
In the trivial case $k=0$ we have $n=a$ so \eqref{Fan} evaluates to $1$, as required; from
now on suppose $k\ge 1$.

Any $[a]$-rooted $r$-forest $H$ on $[n]$ may be constructed by starting from the
hypergraph with vertex set $[a]$ and no edges, and adding edges one-by-one so
that each edge consists of one old vertex (a vertex already present) and a
\emph{group} of $r-1$ new vertices. Although there are in general many possible
orders in which the edges may be added to form a given $H$, the groups will
always be the same -- in each edge the old vertex is the unique vertex at minimal
graph distance from $\{1,2,\ldots,a\}$ in $H$. Let $\pi(H)$ denote the partition
of $[n]\setminus[a]$ formed by the groups. From now on we fix one of the $\partit{k}{r-1}$
possible partitions $\pi$ that may arise in this way, and consider the set
$\cH_\pi$ of $[a]$-rooted $r$-forests on $[n]$ with $\pi(H)=\pi$. Our aim is to
show that $|\cH_\pi|=a n^{k-1}$.

Fix an arbitrary order $\prec$ on the $(r-1)$-element subsets of $[n]\setminus [a]$.
By a \emph{leaf part} of $H$ we mean a part of $\pi(H)$ all of whose
vertices have degree $1$ in $H$.
Let $c(H)$ be the sequence defined as follows: pick the leaf part of $H$
earliest in the order $\prec$, write down the old vertex $v$ appearing
in the corresponding edge $e$, delete $e$, and continue until no edges
remain. The last edge deleted clearly has its old vertex in $[a]$,
so $c(H)$ consists of $k-1$ elements of $[n]$ followed by an element of $[a]$.
It is simple to check that this Pr\"ufer-type code gives a bijection between $\cH_\pi$
and $[n]^{k-1}\times [a]$, and the result follows.
\end{proof}

An alternative way of proving Lemma~\ref{l_rtree} is to map each $H\in \cH_\pi$ to a
$2$-forest on $a+k$ vertices ($[a]$ and the parts of $\pi$). This map
is many-to-one, but the multiplicity depends only on the number
of edges incident with $[a]$, and (surprisingly) one can apply R\'enyi's formula
for $r=2$ to deduce Lemma~\ref{l_rtree}.

\end{document}